\documentclass[12pt]{amsart}

\usepackage{mathtools,euscript,dsfont,amssymb,upref,graphicx,mathdots}
\usepackage{comment}
\usepackage[multiple]{footmisc}
\usepackage{enumitem}
\newlist{enumerata}{enumerate}{2}
\setlist[enumerata,1]{label=\protect\normalfont{(\alph*)},ref=\protect\normalfont{\alph*}}
\setlist[enumerata,2]{label=\protect\normalfont{(\roman*)},ref=\protect\normalfont{\roman*}}
\newlist{enumerati}{enumerate}{1}
\setlist[enumerati,1]{label=\protect\normalfont{(\roman*)}, ref=\protect\normalfont{\roman*}}
\usepackage[all]{xy}
\usepackage{tikz}
\usepackage[normalem]{ulem}
\usepackage{mathrsfs}
\newcommand{\Pscr}{\mathscr{P}}
\newcommand{\Qscr}{\mathscr{Q}}

\newcommand{\Cscr}{\mathscr{C}}

\usepackage{color}

\newtheorem{theorem}{Theorem}[section]
\newtheorem{proposition}[theorem]{Proposition}
\newtheorem{lemma}[theorem]{Lemma}
\newtheorem{corollary}[theorem]{Corollary}

\theoremstyle{definition}

\newtheorem{definition}[theorem]{Definition}

\newtheorem{parag}[theorem]{}

\newtheorem{example}[theorem]{Example}

\newtheorem{notation}[theorem]{Notation}

\newtheorem{remark}[theorem]{Remark}

\theoremstyle{remark}

\newtheorem*{smallremark}{Remark}

\newcommand{\image}{	\operatorname{{\rm im}}}

\newcommand{\lcm} {\operatorname{{\rm lcm}}}

\newcommand{\Spl}{	\operatorname{{\rm Spl}}}

\newcommand{\ENSpl}{	\operatorname{{\rm ENSpl}}}
\newcommand{\Edz}{	\operatorname{{\rm Edz}}}

\newcommand{\M}{{\bbM}}
\newcommand{\Mpol}{{\bbM_{\text{\rm pol}}}}

\newcommand{\AL}{A}

\newcommand{\EN}{\operatorname{{\rm EN}}}

\newcommand{\DT}{	\operatorname{{\rm\bf DT}}}
\newcommand{\DTr}{	\operatorname{{\rm\bf DT}_{\text{\rm r}}}}
\newcommand{\DTpr}{	\operatorname{{\rm\bf DT}_{\text{\rm pr}}}}
\newcommand{\DTprp}{	\operatorname{{\rm\bf DT}^*_{\text{\rm pr}}}}
\newcommand{\DTpru}{	\operatorname{{\rm\bf DT}^1_{\text{\rm pr}}}}

\newcommand{\Cent} {\operatorname{{\rm Cent}}}

\newcommand{\bX}{\mathbf{X}}

\newlength{\mylength}
\newcommand{\sgn}{\operatorname{{\rm sgn}}}

\newcommand{\setspec}[2]{\big\{\,#1\, \mid \,#2\, \big\}}

\newcommand{\Integ}{\ensuremath{\mathbb{Z}}}
\newcommand{\Nat}{\ensuremath{\mathbb{N}}}
\newcommand{\Rat}{\ensuremath{\mathbb{Q}}}
\newcommand{\Comp}{\ensuremath{\mathbb{C}}}
\newcommand{\Reals}{\ensuremath{\mathbb{R}}}

\newcommand{\bbA}{\ensuremath{\mathbb{A}}}
\newcommand{\bbB}{\ensuremath{\mathbb{B}}}

\newcommand{\bbM}{\ensuremath{\mathbb{M}}}

\newcommand{\bbX}{\ensuremath{\mathbb{X}}}

\newcommand{\Dgoth}{{\ensuremath{\mathfrak{D}}}}

\newcommand{\Fgoth}{{\ensuremath{\mathfrak{F}}}}
\newcommand{\Ggoth}{{\ensuremath{\mathfrak{G}}}}

\newcommand{\Mgoth}{{\ensuremath{\mathfrak{M}}}}

\newcommand{\Aeul}{\EuScript{A}}

\newcommand{\Ceul}{\EuScript{C}}

\newcommand{\Eeul}{\EuScript{E}}

\newcommand{\Neul}{\EuScript{N}}
\newcommand{\Oeul}{\EuScript{O}}

\newcommand{\Seul}{\EuScript{S}}
\newcommand{\Teul}{\EuScript{T}}

\newcommand{\Veul}{\EuScript{V}}

\renewcommand{\epsilon}{\varepsilon}
\renewcommand{\phi}{\varphi}
\renewcommand{\emptyset}{\varnothing}

\newcommand{\rien}[1]{}

\swapnumbers

\begin{document}

\title{Decorated trees}

\author{Pierrette Cassou-Nogu\`es}
\address{Univ.\ Bordeaux, CNRS, Bordeaux INP, IMB, UMR 5251,  F-33400, Talence, France}
\email{Pierrette.Cassou-nogues@math.u-bordeaux.fr}
\thanks{Research of the first author partially supported by  Spanish grants PID2020-114750GB-C31 (Zaragoza) and PID2020-114750GB-C32 (Madrid).}

\author{Daniel Daigle}
\address{Department of Mathematics and Statistics\\ University of Ottawa\\ Ottawa, Canada K1N 6N5}
\email{ddaigle@uottawa.ca}

\date{\today}



\begin{abstract}
We study a class of combinatorial objects that we call ``decorated trees''.
These consist of vertices, arrows and edges, 
where each edge is decorated by two integers (one near each of its endpoints), each arrow is decorated by an integer,
and the decorations are required to satisfy certain conditions.
The class of decorated trees includes different types of trees used in algebraic geometry,
such as the Eisenbud and Neumann diagrams for links of singularities
and the Neumann diagrams for links at infinity of algebraic plane curves.
By purely combinatorial means, we recover some formulae that were previously understood to be ``topological''.
In this way, we extend the generality of those formulae and show that they are in fact ``combinatorial''.
\end{abstract}

\maketitle

\vfuzz=2pt


\section*{Introduction}

We study a class of combinatorial objects that we call ``decorated trees''.
These consist of vertices, arrows and edges, 
where each edge is decorated by two integers (one near each of its endpoints), each arrow is decorated by an integer,
and the decorations are required to satisfy certain conditions.
The class of decorated trees includes different types of trees used in algebraic geometry,
such as the Eisenbud and Neumann diagrams for links of singularities \cite{Eisenbud_Neumann_1985}
and the Neumann diagrams for links at infinity of algebraic plane curves \cite{Neumann:LinksInfinity}.
It also includes the larger set of trees that arise in the study of numerical semigroups, in particular in the calculation of the Frobenius number of 
complete intersection numerical semigroups.
By purely combinatorial means, we recover some formulae that were previously understood to be ``topological''.
In this way, we extend the generality of those formulae and show that they are in fact ``combinatorial''.
This article is used  in \cite{ArtalBartoloCassouNogues2024}, as well as in work in preparation \cite{CND:memoir} and \cite{CassouNogues2026}.

The first two sections are mostly dedicated to definitions.
In the first section we define what is a decorated tree $\Teul$,
its multiplicity $M(\Teul) \in \Integ$ and a number $F(\Teul) \in \Integ$
(in the special case where all arrows of $\Teul$ are decorated by $(0)$ or $(1)$, $F(\Teul)$ is simply the number of arrows in $\Teul$ decorated with $(1)$).
Then we describe some operations of simplification that do not change the values of $M(\Teul)$ and $F(\Teul)$
(the simpler trees produced by these operations are called ``minimal trees'' in the context of singularities).
The second section explains Splitting and ENsplitting, which are operations for breaking a tree into two pieces.
The splitting operations allow the use of induction and are very useful tools.

The remaining sections consider formulae that are known for singularities and algebraic curves,
and prove them in a purely combinatorial way and for more general trees.

Section $3$  proves that $M(\Teul)+F(\Teul)$ is even.
This fundamental result allows us to define 
(in section $4$) the genus and the delta-invariant of a decorated tree $\Teul$ by 
$$
\textstyle g(\Teul) = \frac12 ( 2 - M(\Teul) - F(\Teul) )
\quad \text{and} \quad
\textstyle \delta(\Teul) = \frac12 ( F(\Teul) - M(\Teul) ) .
$$
These two invariants are important not only in the case of complex plane curves,
but also for plane curves in any characteristic.
Thm \ref{8374te23nfAiUoMnbed53oifn} and Prop.\ \ref{9837487239rpejf9d88}(a)
are used in particular in \cite{ArtalBartoloCassouNogues2024} to prove that if $f(x,y)\in K[[x,y]]$ with $K$ of arbitrary characteristic, then the $\delta$-invariant of its Newton tree is equal to the $\delta$-invariant of the singularity. 
Moreover, the invariants $g(\Teul)$ and $\delta(\Teul)$ are relevant in the study of numerical semigroups.
Section 4  proves formulae that compute these invariants in terms of subtrees with fewer arrows.

By analogy to the Alexander polynomial of a fibered multilink \cite[Prop.\ 5.1(iii)]{Eisenbud_Neumann_1985},
Section $5$ defines and studies the Alexander polynomial $\AL(\Teul)$ of $\Teul$, for  $\Teul$ in the class of ``positive'' decorated tree defined
in \ref{27hed93wtr37zn2wd0owdkg}.
We show, among other things, that $\AL(\Teul)$ is a finite product of cyclotomic polynomials.

Section $6$ proves a genus formula for a large class of decorated trees,
which recovers the known formula for algebraic plane curves in characteristic zero.

In addition to giving the results mentioned in the above paragraphs,
the present work also serves as the basis for further development of the theory of decorated trees,
in particular in works in progress by the two authors \cite{CND:memoir} and by the first author \cite{CassouNogues2026}.

\medskip

\noindent{\bf Conventions. }
We write ``$\setminus$'' for set difference, ``$\subset$'' for strict inclusion and  ``$\subseteq$'' for general inclusion.
We follow the convention that $0 \in \Nat$.


\section{Definition of decorated trees}
\label {Sec:PreliminariesAbstractNewtontreesatinfinity}

This section defines what is a decorated tree $\Teul$,
its multiplicity $M(\Teul) \in \Integ$ and a number $F(\Teul) \in \Integ$
(in the special case where all arrows of $\Teul$ are decorated by $(0)$ or $(1)$, $F(\Teul)$ is simply the number of arrows in $\Teul$ decorated with $(1)$).
Then we describe some operations of simplification that do not change the values of $M(\Teul)$ and $F(\Teul)$.
We begin with some terminology and notation for general graphs and trees.

CAUTION: our use of the word ``vertex'' differs from the standard usage of graph theory (see \ref{97863cws981bode9}--\ref{pco9v0239jd0OiiJq0wjd}).

\begin{parag} \label {97863cws981bode9}
In this work, a {\it graph\/} is a pair $X = (X_0,X_1)$ where $X_0$ and $X_1$ are finite sets and each element
of $X_1$ is a subset of $X_0$ of cardinality exactly $2$. The elements of $X_1$ are called the {\it edges},
and those of $X_0$ are called {\it $0$-dimensional cells\/} (we do NOT use the word ``vertex'' here!).
If $x \in X_0$ and $e \in X_1$ are such that $x \in e$, we say that the edge $e$ is {\it incident\/} to $x$.
If $x \in X_0$, the number of edges incident to $x$ is called the {\it valency\/} of $x$ and is denoted $\delta_{x}$.
If $x,y \in X_0$ are such that $\{x,y\}$ is an edge, we say that $x,y$ are {\it adjacent},
or that $x,y$ are \textit{linked by an edge}. If $e=\{x,y\}$ is an edge then $x,y$ are the {\it endpoints\/} of $e$.

A {\it path\/} in the graph $X=(X_0,X_1)$ is an ordered tuple $(x_0,\dots,x_n)$ of elements of $X_0$ satisfying
$n\ge0$ and the two conditions:
\begin{itemize}

\item if $n\ge1$ then for each $i \in \{0,\dots,n-1\}$ we have $\{ x_i, x_{i+1} \} \in X_1$;
\item if $n\ge2$ then the edges $\{x_0,x_1\}$, \dots, $\{x_{n-1},x_n\}$ are distinct.

\end{itemize}
Note that what we call a ``path'' is called a ``simple path'' in standard terminology of graph theory.
Given $x,y \in X_0$, a {\it path from $x$ to $y$} is a path $(x_0,\dots,x_n)$ satisfying 
$x_0=x$ and $x_n=y$.

Suppose that $\gamma = (x_0,\dots,x_n)$ is a path. 
We say that a $0$-dimensional cell $x \in X_0$ ``is in $\gamma$'' if $x \in \{ x_0,\dots,x_n \}$;
we say that an edge $e \in X_1$ ``is in $\gamma$'' if $e$ is one of the edges
$\{x_0,x_1\}$, \dots, $\{x_{n-1},x_n\}$.
If $\xi \in X_0 \cup X_1$ and $\xi$ is in a path $\gamma$, we also say that $\gamma$ \textit{contains} $\xi$.

A graph $X = (X_0,X_1)$ is a {\it tree\/} if for every choice of $x,y \in X_0$ there exists a unique path from $x$ to $y$.
If $X$ is a tree and $x,y \in X_0$ then
\begin{center}
the unique path from $x$ to $y$ is denoted $\gamma_{x,y}$.
\end{center}
\end{parag}

\begin{definition} \label {cJioI0qobfoWvvnxzhb83edfAiE7oqnkl}
A \textit{decorated tree} is a $5$-tuple $\Teul = (\Veul,\Aeul,\Eeul,f,q)$ satisfying
\begin{enumerati}
\item $\Veul$ and $\Aeul$ are finite sets, $\Veul \cap \Aeul = \emptyset$, and 
the pair $(\Veul \cup\Aeul,\Eeul)$ is a tree in the sense of paragraph \ref{97863cws981bode9};
\item $\delta_\alpha = 1$ for all $\alpha \in \Aeul$ ($\delta_\alpha$ is the valency of $\alpha$ in the tree $(\Veul\cup\Aeul,\Eeul)$);
\item $f : \Aeul \to \Integ$ is a map;
\item $q : \setspec{ (e,x) \in \Eeul \times (\Veul \cup \Aeul) }{ x \in e } \to \Integ$ is a map;

\item \label {9823r89vb281283a} if $e \in \Eeul$, $\alpha \in \Aeul$ and $\alpha \in e$, then $q(e,\alpha)=1$; 

\item \label {CDT-ie20udf01}
for every $v \in \Veul$ and every choice of distinct edges $e,e'$ incident to $v$, we have\linebreak
\mbox{$\gcd\big( q(e,v), q(e',v) \big) = 1$.}

\end{enumerati}
The set of decorated trees is denoted $\DT$.
Given $\Teul = (\Veul,\Aeul,\Eeul,f,q) \in \DT$,
the elements of $\Veul$ and $\Aeul$ are called, respectively, the \textit{vertices} and \textit{arrows} of $\Teul$.
If $(e,x) \in \Eeul \times (\Veul\cup\Aeul)$ and $x \in e$ then the integer $q(e,x)$ is called {\it the decoration of $e$ near $x$}
(so each edge has two decorations, one near each of its endpoints).
If $\alpha \in \Aeul$ then the integer $f(\alpha)$ is called {\it the decoration of $\alpha$.}
Define
$$
\Aeul_0 = \setspec{ \alpha \in \Aeul }{ f(\alpha) = 0 } .
$$
An edge incident to an element of $\Aeul_0$ is called a {\it dead end.}

The notations of \ref{97863cws981bode9} are valid for decorated trees:
if $x \in \Veul \cup \Aeul$ then $\delta_x$ denotes the valency of $x$ in the tree $(\Veul\cup\Aeul,\Eeul)$;
if $x,y \in \Veul \cup \Aeul$ then $\gamma_{x,y}$ is the unique path from $x$ to $y$.
\end{definition}

\begin{parag}   \label {pco9v0239jd0OiiJq0wjd}
{\bf Pictures.}
In pictures, vertices are represented by
``\begin{picture}(2,1)(-1,-.5) \put(0,0){\circle{1}} \end{picture}''
and arrows are represented by arrowheads
``\begin{picture}(2,1)(-1,-.5) \put(0.5,0){\vector(1,0){0}} \end{picture}'' (not by arrows ``\begin{picture}(5,1)(0,-.5) \put(0,0){\vector(1,0){5}} \end{picture}'').
So ``\begin{picture}(5.5,1)(-.5,-.5) \put(0,0){\circle{1}} \put(0.5,0){\vector(1,0){4.5}} \end{picture}''
represents an edge joining a vertex 
``\begin{picture}(2,1)(-1,-.5) \put(0,0){\circle{1}} \end{picture}''
to an arrow
``\begin{picture}(2,1)(-1,-.5) \put(0.5,0){\vector(1,0){0}} \end{picture}''.
If $\alpha \in \Aeul$ then the decoration of $\alpha$ (i.e., the integer $f(\alpha)$) is enclosed between parentheses.
Part (a) of the following picture is an example of a decorated tree $\Teul = (\Veul,\Aeul,\Eeul,f,q)$:
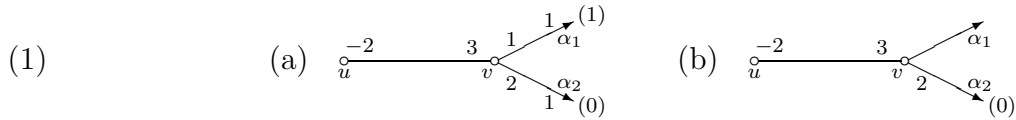
\begin{equation}  \label {dpd99238gve912R2tw}
\text{\rm (a)} \ \ \raisebox{-6mm}{\begin{picture}(38,14)(-2,-7)
\put(0,0){\circle{1}}
\put(20,0){\circle{1}}
\put(.5,0){\line(1,0){19}}
\put(20.4472,.2236){\vector(2,1){10}}
\put(20.4472,-.2236){\vector(2,-1){10}}

\put(0,-1){\makebox(0,0)[t]{\tiny $u$}}
\put(19,-1){\makebox(0,0)[t]{\tiny $v$}}
\put(31,6){\makebox(0,0)[l]{\tiny $(1)$}}
\put(31,-6){\makebox(0,0)[l]{\tiny $(0)$}}
\put(30,3.5){\makebox(0,0)[t]{\tiny $\alpha_1$}}
\put(30,-4){\makebox(0,0)[b]{\tiny $\alpha_2$}}

\put(2,1){\makebox(0,0)[b]{\tiny $-2$}}
\put(17,1){\makebox(0,0)[b]{\tiny $3$}}

\put(28,4.5){\makebox(0,0)[br]{\tiny $1$}}
\put(23,2){\makebox(0,0)[br]{\tiny $1$}}
\put(23,-2){\makebox(0,0)[tr]{\tiny $2$}}
\put(28,-4.5){\makebox(0,0)[tr]{\tiny $1$}}

\end{picture}}
\qquad
\text{\rm (b)}\ \ \raisebox{-6mm}{\begin{picture}(38,14)(-2,-7)
\put(0,0){\circle{1}}
\put(20,0){\circle{1}}
\put(.5,0){\line(1,0){19}}
\put(20.4472,.2236){\vector(2,1){10}}
\put(20.4472,-.2236){\vector(2,-1){10}}

\put(0,-1){\makebox(0,0)[t]{\tiny $u$}}
\put(19,-1){\makebox(0,0)[t]{\tiny $v$}}
\put(31,-6){\makebox(0,0)[l]{\tiny $(0)$}}
\put(30,3.5){\makebox(0,0)[t]{\tiny $\alpha_1$}}
\put(30,-4){\makebox(0,0)[b]{\tiny $\alpha_2$}}

\put(2,1){\makebox(0,0)[b]{\tiny $-2$}}
\put(17,1){\makebox(0,0)[b]{\tiny $3$}}

\put(23,-2){\makebox(0,0)[tr]{\tiny $2$}}

\end{picture}}
\end{equation}
with $\Veul = \{ u, v \}$, $\Aeul = \{ \alpha_1, \alpha_2 \}$,
$\Eeul = \{ \{u, v\}, \{v, \alpha_1\}, \{v,\alpha_2\} \}$, 
$f(\alpha_1) = 1$, 
$f(\alpha_2) = 0$, 
$q(\{u,v\},u) = -2$, 
$q(\{u,v\},v) = 3$, 
$q(\{v,\alpha_1\},\alpha_1) = 1$, etc. 
We have $\Aeul_0 = \{ \alpha_2 \}$, and the edge $\{v,\alpha_2\}$ is a dead end.
We have $\delta_v=3$ and $\delta_u = \delta_{\alpha_1} = \delta_{\alpha_2} = 1$.

We use the following convention to simplify pictures:
{\it it is not necessary to display the decorations of edges that are equal to $1$ and the decorations of arrows that are equal to $(1)$.}
Thus, the tree \eqref{dpd99238gve912R2tw}(a) can be simplified to the one depicted in \eqref{dpd99238gve912R2tw}(b).
\end{parag}

\begin{definition}
Let $\Teul = (\Veul,\Aeul,\Eeul, f,q) \in \DT$.
Given $x \in \Veul \cup \Aeul$, let $E_x$ temporarily denote the set of all edges incident to $x$.
For each $e \in E_x$, define 
$Q(e,x) = \prod_{e' \in E_x \setminus \{e\}} q(e',x)$.
{\bf Here and throughout this paper, empty products are equal to \boldmath $1$.}
This defines a map
$$
Q : \setspec{ (e,x) \in \Eeul \times (\Veul \cup \Aeul) }{ x \in e } \to \Integ .
$$
Note that the maps $q$ and $Q$ have the same domain and same codomain, and that $Q$ is determined by $q$.
Also note that if $x \in \Veul \cup \Aeul$ satisfies $\delta_x=1$, and if $e$ is the edge incident to $x$, then $Q(e,x)=1$ because $Q(e,x)$ is an empty product.

For instance, consider the edges $e = \{u,v\}$, $e' = \{v,\alpha_1\}$ and $e''=\{v,\alpha_2\}$ in the tree $\Teul$ shown in
\eqref{dpd99238gve912R2tw}(a), in paragraph \ref{pco9v0239jd0OiiJq0wjd}.
Then $Q(e,v) = 2$ and (since $Q(e,u)$ is an empty product) $Q(e,u)=1$.
Also, $Q(e',v) = 6$ and $Q(e',\alpha_1)=1=Q(e'',\alpha_2)$.
\end{definition}

\begin{definition}  \label {defdetedgej923hb912edl}  
Let $\Teul \in \DT$.
If $e=\{x,y\}$ is an edge in $\Teul$, we define the {\it determinant of $e$} by 
$$
\det(e) = q(e,x)q(e,y) - Q(e,x)Q(e,y) .
$$
\end{definition}

\begin{definition} \label {c9v39rf0eX9e4np8glr9t8}
Let $\Teul = (\Veul,\Aeul,\Eeul,f,q) \in \DT$.
\begin{enumerati}

\item
We say that an edge $\epsilon$ is {\it incident\/} to a path $\gamma$ if $\epsilon$ is not in $\gamma$
and $\epsilon$ is incident to some vertex of $\gamma$.
If $\epsilon$ is incident to $\gamma$,
we define $q(\epsilon,\gamma) = q(\epsilon,u)$ where $u$ is the unique vertex of $\gamma$ to which
$\epsilon$ is incident.

\item Given $v \in \Veul \cup \Aeul$ and $\alpha \in \Aeul$ such that $v \neq \alpha$, we set
$$
\textstyle
x_{v,\alpha} = f(\alpha) \prod_{\epsilon \in E} q(\epsilon,\gamma_{v,\alpha})
\quad \text{and} \quad
\hat x_{v,\alpha} = f(\alpha) \prod_{\epsilon \in \hat E} q(\epsilon,\gamma_{v,\alpha})
$$
where $E$ is the set of edges incident to $\gamma_{v,\alpha}$ and
$\hat E$ is the set of edges incident to $\gamma_{v,\alpha}$ but not incident to $v$.
Observe that $x_{v,\alpha} = Q(e,v) \hat x_{v,\alpha}$, where $e$ is the unique edge in $\gamma_{v,\alpha}$ incident to $v$.
Also note that if $\alpha \in \Aeul_0$ then $x_{v,\alpha} = 0 = \hat x_{v,\alpha}$.

\item Given $v\in \Veul\cup \Aeul_0$, we define the {\it multiplicity $N_v$ of $v$} by
$$
\textstyle N_v
= \sum_{ \alpha \in \Aeul \setminus \Aeul_0} x_{v,\alpha} 
= \sum_{ \alpha \in \Aeul \setminus\{v\} } x_{v,\alpha} .
$$

\item  We define the {\it multiplicity\/} $M(\Teul)$ of $\Teul$ by
$$
\textstyle M(\Teul) = -\sum_{v \in \Veul\cup \Aeul_0} N_v(\delta_v-2) .
$$

\item Given $u \in \Veul \cup \Aeul$ and an edge $e$ incident to $u$, define:
\begin{itemize}

\item $\Aeul(u,e) = \setspec{ \alpha \in \Aeul \setminus \Aeul_0 }{ \text{$e$ is in $\gamma_{u,\alpha}$} }$

\item $\Aeul^*(u,e) = (\Aeul \setminus \Aeul_0) \setminus \Aeul(u,e)
=\setspec{ \alpha \in \Aeul \setminus \Aeul_0 }{ \text{$e$ is not in $\gamma_{u,\alpha}$} }$

\item $p(u,e) = \sum_{ \alpha \in \Aeul(u,e) } \hat x_{ u,\alpha }$

\item $p^*(u,e) = \sum_{ \alpha \in \Aeul^*(u,e) } \hat x_{ v,\alpha } = p(v,e)$ where  $v$ is defined by $e = \{u,v\}$.

\end{itemize}

\item For each $\alpha \in \Aeul \setminus \Aeul_0$ we define $F(\alpha) = \gcd\big( f(\alpha), p(\alpha,e_\alpha) \big)$,
where $e_\alpha$ denotes the unique edge incident to $\alpha$.

\item Define $F(\Teul) = \sum_{ \alpha \in \Aeul \setminus \Aeul_0 } F(\alpha)$.

\end{enumerati}
\end{definition}


\begin{remark}  \label {52retfxcxntgBmA98WjmHkkj498736762}
Let $\Teul \in \DT$.
If $e = \{x,y\}$ is an edge with $x \in \Veul \cup \Aeul_0$ and $y \in \Veul \cup \Aeul$, then 
$$
N_x = Q(e,x) p(x,e) + q(e,x) p^*(x,e) = Q(e,x) p(x,e) + q(e,x) p(y,e) .
$$
\end{remark}

\begin{remark}  \label {DdjkfblowicblwFT983764t1}
Let $\Teul \in \DT$.
For all $\alpha \in \Aeul \setminus \Aeul_0$ we have $F(\alpha) \in \Nat\setminus \{0\}$,
and $F(\alpha)=1$ whenever $f(\alpha)=\pm1$.
It follows that if $f( \Aeul ) \subseteq \{-1,0,1\}$ then $F(\Teul) = | \Aeul \setminus \Aeul_0 |$.
\end{remark}


\begin{parag}   \label {CROCHETpco9v0239jd0OiiJq0wjd}
{\bf Pictures.}
It is well known that vertices $v \in \Veul$ satisfying $N_v=0$ play a special role in the theory.
To make it easier to find them in a picture, we sometimes  represent them by ``\begin{picture}(2,1)(-1,-.5) \put(0,0){\circle*{1}} \end{picture}''
(as opposed to ``\begin{picture}(2,1)(-1,-.5) \put(0,0){\circle{1}} \end{picture}'').
To simplify pictures, we also use the following convention.
If $n \in \Nat$ and $v$ is a vertex satisfying $N_v=0$ then
{\setlength{\unitlength}{1.5mm}
$$
\textit{$\raisebox{.5mm}{\begin{picture}(14,1)(-.9,-.5)
\put(0,0){\circle*{1}}
\put(0,1){\makebox(0,0)[b]{\footnotesize $v$}}
\put(0.5,0){\line(1,0){6.5}}
\put(6.7,0.05){\makebox(0,0)[l]{\footnotesize $<\! n$}}
\end{picture}}$ is an abbreviation for}\qquad
\scalebox{.8}{\raisebox{-6\unitlength}{\begin{picture}(21,12.5)(-.9,-6.5)
\put(-1,0){\makebox(0,0)[r]{\small $v$}}
\put(0,0){\circle*{1}}
\put(0.447,0.224){\vector(2,1){10}}
\put(0.447,-0.224){\vector(2,-1){10}}
\put(7,.6){\makebox(0,0)[c]{\tiny $\vdots$}}
\put(11,5){\makebox(0,0)[l]{\tiny $(1)$}}
\put(11,-5){\makebox(0,0)[l]{\tiny $(1)$}}
\put(2,2.4){\makebox(0,0)[rd]{\tiny $1$}}
\put(7,5){\makebox(0,0)[rd]{\tiny $1$}}
\put(1,-2.4){\makebox(0,0)[ld]{\tiny $1$}}
\put(7,-4.7){\makebox(0,0)[ru]{\tiny $1$}}
\put(15,0){\rotatebox{90}{\makebox(0,0){$\displaystyle\underbrace{\rule{12\unitlength}{0mm}}$}}}
\put(16.5,0){\makebox(0,0)[l]{\small $n$}}
\end{picture}}} .
$$
That is, {\begin{picture}(11,1)(-.9,-.5)
\put(0,0){\circle*{1}}
\put(0,1){\makebox(0,0)[b]{\footnotesize $v$}}
\put(0.5,0){\line(1,0){6.5}}
\put(6.7,0.05){\makebox(0,0)[l]{\footnotesize $<\! n$}}
\end{picture}}}
replaces $n$ edges $e_i = \{v,\alpha_i\}$ ($1 \le i \le n$) where, for each $i$,
we have $\alpha_i \in \Aeul$, $f(\alpha_i)=1$ and $q(e_i,v)=1$.
For instance, the tree in \eqref{8754bcbnnsjdbceg3746t6487gd}(a) belongs to $\DT$, and  \eqref{8754bcbnnsjdbceg3746t6487gd}(b) shows a 
different picture of the same tree, using the above convention:
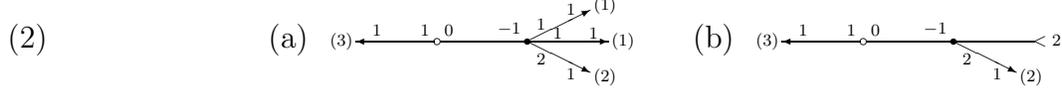
\begin{equation}  \label {8754bcbnnsjdbceg3746t6487gd}
\text{\rm (a)} \ \ \scalebox{.8}{\raisebox{-6mm}{\begin{picture}(45,14)(-13,-7)
\put(5,0){\circle{1}}
\put(20,0){\circle*{1}}
\put(5.5,0){\line(1,0){14}}
\put(20.4472,.2236){\vector(2,1){10}}
\put(20.5,0){\vector(1,0){13}}
\put(20.4472,-.2236){\vector(2,-1){10}}

\put(4.5,0){\vector(-1,0){13}}
\put(-9,0){\makebox(0,0)[r]{\tiny $(3)$}}
\put(-5,1){\makebox(0,0)[b]{\tiny $1$}}

\put(31,6){\makebox(0,0)[l]{\tiny $(1)$}}
\put(34,0){\makebox(0,0)[l]{\tiny $(1)$}}
\put(31,-6){\makebox(0,0)[l]{\tiny $(2)$}}

\put(3,1){\makebox(0,0)[b]{\tiny $1$}}
\put(7,1){\makebox(0,0)[b]{\tiny $0$}}
\put(17,1){\makebox(0,0)[b]{\tiny $-1$}}

\put(28,4.5){\makebox(0,0)[br]{\tiny $1$}}
\put(23,2){\makebox(0,0)[br]{\tiny $1$}}
\put(25,.5){\makebox(0,0)[b]{\tiny $1$}}
\put(31,.5){\makebox(0,0)[b]{\tiny $1$}}
\put(23,-2){\makebox(0,0)[tr]{\tiny $2$}}
\put(28,-4.5){\makebox(0,0)[tr]{\tiny $1$}}

\end{picture}}}
\qquad\quad
\text{\rm (b)} \ \ \scalebox{.8}{\raisebox{-6mm}{\begin{picture}(45,14)(-13,-7)
\put(5,0){\circle{1}}
\put(20,0){\circle*{1}}
\put(5.5,0){\line(1,0){14}}
\put(20.5,0){\line(1,0){13}}
\put(20.4472,-.2236){\vector(2,-1){10}}

\put(4.5,0){\vector(-1,0){13}}
\put(-9,0){\makebox(0,0)[r]{\tiny $(3)$}}
\put(-5,1){\makebox(0,0)[b]{\tiny $1$}}

\put(33.2,0.1){\makebox(0,0)[l]{\tiny $< 2$}}
\put(31,-6){\makebox(0,0)[l]{\tiny $(2)$}}

\put(3,1){\makebox(0,0)[b]{\tiny $1$}}
\put(7,1){\makebox(0,0)[b]{\tiny $0$}}
\put(17,1){\makebox(0,0)[b]{\tiny $-1$}}

\put(23,-2){\makebox(0,0)[tr]{\tiny $2$}}
\put(28,-4.5){\makebox(0,0)[tr]{\tiny $1$}}

\end{picture}}}
\end{equation}
We may further simplify the picture by using the convention of \ref{pco9v0239jd0OiiJq0wjd}, as follows:
$$
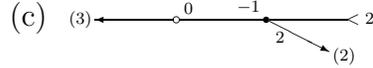
\text{\rm (c)} \ \ \scalebox{.8}{\raisebox{-6mm}{\begin{picture}(45,14)(-13,-7)
\put(5,0){\circle{1}}
\put(20,0){\circle*{1}}
\put(5.5,0){\line(1,0){14}}
\put(20.5,0){\line(1,0){13}}
\put(20.4472,-.2236){\vector(2,-1){10}}

\put(4.5,0){\vector(-1,0){13}}
\put(-9,0){\makebox(0,0)[r]{\tiny $(3)$}}

\put(33.2,0.1){\makebox(0,0)[l]{\tiny $< 2$}}
\put(31,-6){\makebox(0,0)[l]{\tiny $(2)$}}

\put(7,1){\makebox(0,0)[b]{\tiny $0$}}
\put(17,1){\makebox(0,0)[b]{\tiny $-1$}}

\put(23,-2){\makebox(0,0)[tr]{\tiny $2$}}

\end{picture}}}
$$
We stress that (a), (b) and (c) are pictures of {\it the same tree},
and that the valency of the vertex  ``\begin{picture}(2,1)(-1,-.5) \put(0,0){\circle*{1}} \end{picture}'' in (b) and (c) is $4$, not $3$.
\end{parag}

\begin{example}  \label {894Y3yhbQwc34eeruGVfef8}
The following picture of a decorated tree $\Teul \in \DT$ uses the conventions of both \ref{pco9v0239jd0OiiJq0wjd} and \ref{CROCHETpco9v0239jd0OiiJq0wjd}:
$$
\scalebox{.85}{\begin{picture}(60,26.5)(-11,-18)
\put(0,0){\circle*{1}}
\put(15,0){\circle{1}}
\put(30,0){\circle{1}}
\put(0,0){\line(-1,0){8.0}} \put(-7.5,0.13){\makebox(0,0)[r]{\tiny $3 >$}}
\put(0,-.5){\vector(0,-1){9.5}}
\put(1,-9){\makebox(0,0)[l]{\tiny $(0)$}}
\put(2,.5){\makebox(0,0)[bl]{\tiny $-1$}}
\put(22.5,-1){\makebox(0,0)[t]{\tiny $e$}}
\put(31,2){\makebox(0,0)[b]{\tiny $2$}}
\put(14.5,0){\line(-1,0){14}}
\put(29.5,0){\line(-1,0){14}}
\put(30.35355,0.35355){\line(1,1){8}}
\put(30.35355,-0.35355){\line(1,-1){8}}
\put(38,8){\circle*{1}}
\put(38,7.5){\vector(0,-1){9.5}}
\put(39,-1){\makebox(0,0)[l]{\tiny $(0)$}}
\put(36,6){\makebox(0,0)[rb]{\tiny $-4$}}
\put(38,8){\line(1,0){8.0}} \put(45.7,8.07){\makebox(0,0)[l]{\tiny $< 1$}}
\put(38,-8){\line(1,0){8.0}} \put(45.7,-7.94){\makebox(0,0)[l]{\tiny $< 1$}}
\put(38,-8){\circle*{1}}
\put(38,-8.5){\vector(0,-1){9.5}}
\put(39,-17){\makebox(0,0)[l]{\tiny $(0)$}}
\put(35.5,-6){\makebox(0,0)[rt]{\tiny $-7$}}
\end{picture}}
$$
(One of the edges is called ``$e$'', because Ex.\ \ref{8f790hbYBiD9485thBdfdcy82fjx3sBhNt87} refers to that edge.)
Note that there are eight arrows in this tree: three are decorated by $(0)$ and five are decorated by $(1)$.
The reader may verify that the vertices $v$ that satisfy $N_v=0$ are precisely those that are represented by
 ``\begin{picture}(2,1)(-1,-.5) \put(0,0){\circle*{1}} \end{picture}'' (as explained in \ref{CROCHETpco9v0239jd0OiiJq0wjd}).
\end{example}

\begin{remark} \label {es5s6xjc17lcy4l}
Let $e=\{v,\alpha\}$ be a dead end, where $v \in \Veul$ and $\alpha \in \Aeul_0$.  Then 
$$
N_v = q(e,v) N_{\alpha}.
$$
To see this, simply observe that $x_{v,\beta}=q(e,v)x_{\alpha,\beta}$ for all $\beta \in \Aeul \setminus \Aeul_0$.
\end{remark}

\begin{lemma} \label {pc09vvh230e9dfq99912we2e3}
Let $\Teul \in \DT$ and let $e = \{v,\alpha\}$ be an edge with  $v \in \Veul$ and $\alpha \in \Aeul \setminus \Aeul_0$.
If $N_v = 0$, then $q(e,v) \mid f(\alpha)$.
\end{lemma}

\begin{proof}
If $e$ is the only edge incident to $v$ then $0 = N_v = x_{v,\alpha} = f(\alpha) \neq 0$, a contradiction.
So $e$ is not the only edge incident to $v$.
Let $e,e_1,\dots,e_n$ ($n\ge1$) be the distinct edges incident to $v$, let $y=q(e,v)$, and for each $i \in \{1, \dots, n\}$, let
$y_i = q(e_i,v)$. 
For each $\beta \in (\Aeul \setminus \Aeul_0) \setminus\{\alpha\}$, we have $y \mid x_{v,\beta}$;
so $N_v = y_1\cdots y_n f(\alpha) + y z$ for some $z \in \Integ$.
As $N_v=0$, $y$ divides $y_1\cdots y_n f(\alpha)$.
Since $y,y_1,\dots,y_n$ are pairwise relatively prime by \ref{cJioI0qobfoWvvnxzhb83edfAiE7oqnkl}\eqref{CDT-ie20udf01}, it follows that $y \mid f(\alpha)$.
\end{proof}

Consider a dead end $e=\{v,\alpha\}$, where $v \in \Veul$ and $\alpha \in \Aeul_0$.
If the decoration of $e$ near $v$ is $1$, we call $e$ a {\it dead end decorated by $1$}.

\begin{definition}  \label {oO092hrdn7e30q9mwgfnzPloa7Wg7}
Let $\Teul = (\Veul,\Aeul,\Eeul, f,q) \in \DT$.
In each of the following four cases, we define a tree $\Teul' = (\Veul',\Aeul',\Eeul', f',q')$ that belongs to $\DT$.
\begin{enumerata}

\item Let $e = \{v,\alpha\}$ be a dead end, where $v \in \Veul$, $\alpha \in \Aeul_0$ and $q(e,v)=1$.
Let $\Teul'$ be the tree obtained from $\Teul$ by deleting the arrow $\alpha$ and the edge $e$.
We say that $\Teul'$ is obtained from $\Teul$ {\it by deleting a dead end decorated by $1$.}

\item Let $e = \{v,t\}$ be an edge such that $v,t \in \Veul$, $\delta_t=1$ and $q(e,v)=1$.
Let $\Teul'$ be the tree obtained from $\Teul$ by deleting the vertex $t$ and the edge $e$.
We say that $\Teul'$ is obtained from $\Teul$ {\it by deleting a pending vertex decorated by $1$.}

\item Let $v \in \Veul$ be such that $\delta_v=2$, and let $e_1 = \{u_1,v\}$ and $e_2=\{v,u_2\}$ be the two edges incident to $v$ (where $u_1,u_2 \in \Veul \cup \Aeul$).
Let $\Teul'$ be the tree obtained from $\Teul$ by deleting the vertex $v$ and the edges $e_1$ and $e_2$,
adding the edge $e' = \{u_1,u_2\}$, and defining the decorations of $e'$ near $u_1$ and $u_2$ to be $q(e_1,u_1)$ and $q(e_2,u_2)$ respectively.
We say that $\Teul'$ is obtained from $\Teul$ {\it by deleting a vertex of valency $2$.}

\item Let $e = \{v_1,v_2\}$ be an edge such that
\begin{equation}  \label {f985bbvAbYqbpqkqlqy6542}
v_1,v_2 \in \Veul, \quad q(e,v_1) = Q(e,v_2) \quad \text{and} \quad q(e,v_2) = Q(e,v_1).
\end{equation}
Observe that $\det(e)=0$.
Let $\Teul'$ be the tree obtained from $\Teul$ by deleting the edge $e$ and the vertices $v_1$ and $v_2$, adding a new vertex $v$,
and replacing each edge of the form $\epsilon = \{x,v^*\}$ with  $v^* \in \{v_1,v_2\}$ and $x \notin \{v_1,v_2\}$
by an edge $\epsilon' = \{x,v\}$ with decorations $q'(\epsilon',x) = q(\epsilon,x)$ and $q'(\epsilon',v) = q(\epsilon,v^*)$.
We say that $\Teul'$ is obtained from $\Teul$ {\it by contracting an edge of determinant zero.}
\end{enumerata}
\end{definition}

\begin{remark}
Not all edges of determinant zero can be contracted as in \ref{oO092hrdn7e30q9mwgfnzPloa7Wg7}(d).
Consider an edge $e = \{v_1,v_2\}$ of $\Teul\in\DT$ such that $v_1,v_2 \in \Veul$ and $\det(e)=0$. 
Define $q_i = q(e,v_i)$ and $Q_i = Q(e,v_i)$ for $i = 1,2$.
It is not hard to see that either $(q_1,q_2) = (Q_2,Q_1)$ or $(q_1,q_2) = (-Q_2,-Q_1)$.
The contraction operation of \ref{oO092hrdn7e30q9mwgfnzPloa7Wg7}(d) is permitted if and only if
condition \eqref{f985bbvAbYqbpqkqlqy6542} is satisfied, if and only if $(q_1,q_2) = (Q_2,Q_1)$.
\end{remark}

\begin{lemma}  \label {0v82u395u4i83hrf993md4tj09}
Let $\Teul, \Teul' \in \DT$, and suppose that $\Teul'$ is obtained from $\Teul$ by 
performing one of the four operations of Def.\ \ref{oO092hrdn7e30q9mwgfnzPloa7Wg7}. 
Then  $M(\Teul) = M(\Teul')$ and $F(\Teul) = F(\Teul')$.
\end{lemma}

\begin{proof}
We use the notations $\Teul = (\Veul,\Aeul,\Eeul, f,q)$ and $\Teul' = (\Veul',\Aeul',\Eeul', f',q')$.
We use primes ($'$) to indicate that a quantity is computed in $\Teul'$ (for instance $x_{v,\alpha}$ is computed in $\Teul$
and $x'_{v,\alpha}$ is computed in $\Teul'$).
We prove that $M(\Teul) = M(\Teul')$. The proof of $F(\Teul) = F(\Teul')$ is much easier, and left to the reader.

In the first three cases, it is clear that $N_w' = N_w$ for every $w \in \Veul' \cup \Aeul_0' \subset \Veul \cup \Aeul_0$.

The first two cases are treated together:
let $e = \{v,t\}$ be an edge in $\Teul$ such that $v \in \Veul$, $t \in \Veul \cup \Aeul_0$, $\delta_t=1$, and $q(e,v)=1$,
and suppose that $\Teul'$ is obtained by deleting $t$ and $e$.
Since $q(e,v)=1$, we have $x_{t,\alpha} = x_{v,\alpha}$ for all $\alpha \in \Aeul \setminus \Aeul_0$, so $N_t = N_v$.
Since $\delta_v = \delta'_v+1$, $N_v(\delta_v-2) + N_t(\delta_t-2) = N_v(\delta_v-2) - N_v = N'_v(\delta'_v-2)$.
Since $N_x(\delta_x-2) = N'_x(\delta'_x-2)$ for all $x \in  (\Veul \cup \Aeul_0) \setminus \{v,t\} = (\Veul' \cup \Aeul_0') \setminus \{v\}$,
we get $M(\Teul) = M(\Teul')$.

If $\Teul'$ is obtained by deleting a vertex $v$ of valency $2$,
then  $N_v(\delta_v-2) = 0$ and $N_x(\delta_x-2) = N'_x(\delta'_x-2)$ for all $x \in  (\Veul \cup \Aeul_0) \setminus \{v\} = \Veul' \cup \Aeul_0'$,
so $M(\Teul) = M(\Teul')$.

If $\Teul'$ is obtained from $\Teul$ by contracting an edge $e = \{v_1,v_2\}$ of determinant zero then it is easily verified that 
$N_v' = N_{v_1} = N_{v_2}$ and $(\delta_{v_1}-2) + (\delta_{v_2}-2) = (\delta_{v}'-2)$, so
$N_{v_1}(\delta_{v_1}-2) + N_{v_2}(\delta_{v_2}-2) = N_v'(\delta_{v}'-2)$; moreover,
$N_w' = N_w$ and $\delta_w' = \delta_w$ for every $w \in \Veul' \cup \Aeul_0' \setminus \{v\} = \Veul \cup \Aeul_0 \setminus \{v_1,v_2\}$,
so $M(\Teul) = M(\Teul')$.
\end{proof}

It is possible to define an inverse of the contraction operation of Def.\ \ref{oO092hrdn7e30q9mwgfnzPloa7Wg7}(d).
We do this in Def.\ \ref{983tet3w23hwryhj4djne}. First, we need:

\begin{definition}  \label {87tr432hb8re85g78s87jbuf}
A {\it $2$-prepartition\/} of a set $E$ is a set $\Pscr = \{E_1,E_2\}$ such that $E_1,E_2 \subseteq E$, $E_1 \cup E_2 = E$ and  $E_1 \cap E_2 = \emptyset$.
Note that $\{ \emptyset, E \}$ is a $2$-prepartition of $E$.
\end{definition}

\begin{definition}  \label {983tet3w23hwryhj4djne}
Let $\Teul \in \DT$, let $v \in \Veul$, let $E_v$ denote the set of all edges of $\Teul$ that are incident to $v$,
and let $\Pscr$ be a $2$-prepartition of $E_v$ (see Def.\ \ref{87tr432hb8re85g78s87jbuf}).
The triple $(\Teul,v,\Pscr)$ determines a pair $(\Teul_0,e)$ where $\Teul_0 \in \DT$ and $e$ is an edge of $\Teul_0$ satisfying \eqref{f985bbvAbYqbpqkqlqy6542}.
We proceed to define $(\Teul_0,e)$.

Write $\Pscr = \{ E_1,E_2 \}$. For each $i \in \{1,2\}$, let $X_i = \setspec{ x \in \Veul \cup \Aeul }{ \{x,v\} \in E_i }$
and $a_i = \prod_{e \in E_i} q(e,v)$.
We write $q$ and $q_0$ for the edge decorations in $\Teul$ and $\Teul_0$ respectively.
Let $\Teul_0$ be the tree obtained by performing the following operations on~$\Teul$:
\begin{itemize}

\item delete $v$ and all edges incident to $v$;

\item add two vertices $v_1$ and $v_2$ and the edge $e = \{v_1, v_2\}$, and define\\
$q_0( e , v_1 ) = a_2$ and $q_0( e , v_2 ) = a_1$;

\item for each $i \in \{1,2\}$ and $x \in X_i$, add the edge $\{v_i,x\}$ and define\\
$q_0( \{v_i,x\} , v_i ) = q( \{v,x\}, v )$ and $q_0( \{v_i,x\} , x ) = q( \{v,x\}, x )$.

\end{itemize}

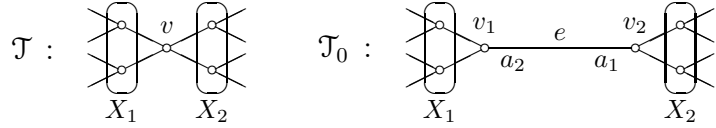
$$
{\begin{picture}(32,14)(-20,-9)

\put(-15,0){\makebox(0,0)[r]{$\Teul\,:$}}

\put(0,0){\circle{1}}
\put(0,2){\makebox(0,0)[b]{\footnotesize $v$}}

\put(-0.4472,0.2236){\line(-2,1){5}}
\put(-0.4472,-0.2236){\line(-2,-1){5}}

\put(0.4472,0.2236){\line(2,1){5}}
\put(0.4472,-0.2236){\line(2,-1){5}}

\put(6,3){\circle{1}}
\put(6,-3){\circle{1}}
\put(-6,3){\circle{1}}
\put(-6,-3){\circle{1}}

\put(-6.4472,3.2236){\line(-2,1){4}}
\put(-6.4472,2.7764){\line(-2,-1){4}}

\put(-6.4472,-3.2236){\line(-2,-1){4}}
\put(-6.4472,-2.7764){\line(-2,1){4}}

\put(6.4472,3.2236){\line(2,1){4}}
\put(6.4472,2.7764){\line(2,-1){4}}

\put(6.4472,-3.2236){\line(2,-1){4}}
\put(6.4472,-2.7764){\line(2,1){4}}

\put(-6,0){\oval(4,12)}
\put(6,0){\oval(4,12)}

\put(-6,-7){\makebox(0,0)[t]{\footnotesize $X_1$}}
\put(6,-7){\makebox(0,0)[t]{\footnotesize $X_2$}}

\end{picture}}
\qquad 
{\begin{picture}(54,14)(-32,-9)

\put(-25,0){\makebox(0,0)[r]{$\Teul_0\,:$}}

\put(-10,0){\circle{1}}
\put(-10,2){\makebox(0,0)[b]{\footnotesize $v_1$}}
\put(-8,-1){\makebox(0,0)[tl]{\footnotesize $a_2$}}
\put(8,-1){\makebox(0,0)[tr]{\footnotesize $a_1$}}

\put(-9.5,0){\line(1,0){19}}

\put(10,0){\circle{1}}
\put(10,2){\makebox(0,0)[b]{\footnotesize $v_2$}}

\put(0,1){\makebox(0,0)[b]{\footnotesize $e$}}

\put(-10.4472,0.2236){\line(-2,1){5}}
\put(-10.4472,-0.2236){\line(-2,-1){5}}

\put(10.4472,0.2236){\line(2,1){5}}
\put(10.4472,-0.2236){\line(2,-1){5}}

\put(16,3){\circle{1}}
\put(16,-3){\circle{1}}
\put(-16,3){\circle{1}}
\put(-16,-3){\circle{1}}

\put(-16.4472,3.2236){\line(-2,1){4}}
\put(-16.4472,2.7764){\line(-2,-1){4}}

\put(-16.4472,-3.2236){\line(-2,-1){4}}
\put(-16.4472,-2.7764){\line(-2,1){4}}

\put(16.4472,3.2236){\line(2,1){4}}
\put(16.4472,2.7764){\line(2,-1){4}}

\put(16.4472,-3.2236){\line(2,-1){4}}
\put(16.4472,-2.7764){\line(2,1){4}}

\put(-16,0){\oval(4,12)}
\put(16,0){\oval(4,12)}

\put(-16,-7){\makebox(0,0)[t]{\footnotesize $X_1$}}
\put(16,-7){\makebox(0,0)[t]{\footnotesize $X_2$}}

\end{picture}}
$$
Observe that $\Teul_0 \in \DT$ and that $e = \{v_1,v_2\}$ is an edge of  $\Teul_0$ satisfying \eqref{f985bbvAbYqbpqkqlqy6542}.
Since $(\Teul_0,e)$ is obtained from $(\Teul,v,\Pscr)$ by creating an \underline{e}dge of \underline{d}eterminant \underline{z}ero, we shall use 
the notation 
$$
\Edz(\Teul,v,\Pscr) = (\Teul_0,e) .
$$
Note that the tree obtained from $\Teul_0$ by contracting $e$ (Def.\ \ref{oO092hrdn7e30q9mwgfnzPloa7Wg7}(d)) is $\Teul$.
By Lemma \ref{0v82u395u4i83hrf993md4tj09}, it follows that 
$$
M(\Teul_0) = M(\Teul) \quad \text{and} \quad  F(\Teul_0) = F(\Teul) .
$$
Also, it is straightforward to verify that 
\begin{equation}  \label {8765rfdmpldnpwrfp8jmlrwpuy65dr43gfd4w228834yj}
N_{v_1} = N_{v_2} = N_{v}
\end{equation}
where $N_{v_1}$ and $N_{v_2}$ are computed in $\Teul_0$ and $N_v$ is computed in $\Teul$.
\end{definition}

\begin{notation}
Let $\gamma = (x_0, \dots, x_n)$ be a path in a tree $\Teul \in \DT$, with $n>0$.
Define 
$$
q(\gamma,x_0) = q( \{x_0,x_1\}, x_0) \text{\ \  and\ \ } q(\gamma,x_n) = q( \{x_{n-1},x_n\}, x_n)
$$
and for each $u \in \{ x_0, \dots, x_n \}$,  
$$
Q(\gamma,u) = \prod_{\epsilon \in \Eeul(\gamma,u)}  q(\epsilon,u)
$$
where $\Eeul(\gamma,u)$ is the set of edges that are incident to $u$ but are not in $\gamma$,
and where empty products are equal to $1$ by convention.
Also, define
$$
Q^*(\gamma) = \prod_{0<i<n} Q(\gamma,x_i) .
$$
\end{notation}


\begin{definition} \label {dpc9vin2o309d093detpathd920}
Let $\gamma = (x_0, \dots, x_n)$ be a path in a tree $\Teul \in \DT$, with $n>0$.
We define the determinant of $\gamma$ as:
$$
\det\gamma = q(\gamma,x_0)  q(\gamma,x_n) - Q^*(\gamma)^2  Q(\gamma,x_0)  Q(\gamma,x_n) .
$$
\end{definition}

The determinant of a path generalizes the determinant of an edge (Def.\ \ref{defdetedgej923hb912edl}).

\begin{definition} 
Let $\Teul \in \DT$ and consider a path $\gamma = (x_0, \dots, x_n)$ in $\Teul$ with $n>0$.
We say that $\gamma$ is a {\it linear path\/} if $\delta_{x_i}=2$ for all $i$ such that $0<i<n$. 
If $\gamma$ is a linear path then the definition of $\det(\gamma)$ given in \ref{dpc9vin2o309d093detpathd920}
simplifies to $\det(\gamma) = q(e_1,x_0) q(e_n,x_n) - Q(e_1,x_0)Q(e_n,x_n)$,
where $e_1 = \{x_0,x_1\}$ and $e_n = \{x_{n-1},x_n\}$.
\end{definition}

\begin{proposition} \label {kuwdhr12778}
Let $\Teul\in\DT$, consider distinct $v,v' \in \Veul \cup \Aeul_0$,
and suppose that the path $\gamma = \gamma_{v,v'}$ is linear. 
If $e$ is the unique edge in $\gamma$ which is incident to~$v$ then
$$
\left| \begin{smallmatrix} q(\gamma,v) & Q(\gamma,v') \\ N_v & N_{v'} \end{smallmatrix} \right| = \det(\gamma) p(v,e) \, .
$$
\end{proposition}

\begin{proof}
Write $\gamma = (x_0, \dots, x_n)$ and $e_i = \{x_{i-1},x_i\}$ for $i=1,\dots,n$, and define
$$
Q = Q(e_1,x_0) , \quad q = q(e_1,x_0) , \quad q' = q(e_n,x_n) , \quad Q' = Q'(e_n,x_n) .
$$
Since $\gamma$ is linear, we have $p(x_0,e_1) = p(x_{n-1},e_n)$ and $p(x_n,e_n) = p(x_1,e_1)$.
We use the abbreviations $p = p(x_0,e_1) = p(x_{n-1},e_n)$ and $p' = p(x_n,e_n) = p(x_1,e_1)$.
Then Rem.\ \ref{52retfxcxntgBmA98WjmHkkj498736762} gives $N_v = Qp + qp'$ and $N_{v'} = q'p + Q'p'$, so
$$
\left| \begin{smallmatrix} q(\gamma,v) & Q(\gamma,v') \\ N_v & N_{v'} \end{smallmatrix} \right|
=\left| \begin{smallmatrix} q & Q' \\ Qp + qp' & q'p + Q'p' \end{smallmatrix} \right|
=\left| \begin{smallmatrix} q & Q' \\ Qp  & q'p \end{smallmatrix} \right|
=\left| \begin{smallmatrix} q & Q' \\ Q  & q' \end{smallmatrix} \right| p
= \det(\gamma) p(v,e) \, .
$$
\end{proof}

\begin{figure}[h]
\setlength{\unitlength}{1mm}
\scalebox{.8}{\begin{picture}(74,30)(-27,-15)
\put(-20,0){\circle{1}}
\put(0,0){\circle{1}}
\put(-19.5,0){\line(1,0){19}}
\put(-20.632,0.316){\line(-2,1){5}}
\put(-20.632,-.316){\line(-2,-1){5}}
\put(-19,-1){\makebox(0,0)[tl]{\tiny $q(\gamma,v)$}}
\put(-20,1){\makebox(0,0)[b]{\footnotesize $v$}}
\put(-10,1){\makebox(0,0)[b]{\footnotesize $e$}}

\put(0,8){\makebox(0,0){$\displaystyle\overbrace{\rule{44\unitlength}{0mm}}^{\text{\footnotesize $\gamma$ is a linear path}}$}}

\put(20,0){\circle{1}}
\put(0.5,0){\line(1,0){6.3}}
\put(19.5,0){\line(-1,0){6.3}}
\put(10.5,0){\makebox(0,0){\dots}}
\put(20.632, 0.316){\line(2,1){16.367}}
\put(20.632,-0.316){\line(2,-1){16.367}}
\put(42,0){\oval(10,30)}
\put(42,0){\makebox(0,0){\tiny $\Aeul(v,e)$}}
\put(20,1){\makebox(0,0)[b]{\footnotesize $v'$}}
\put(23,0){\makebox(0,0)[l]{\tiny $Q(\gamma,v')$}}
\end{picture}}
\caption{Schematic representation of the situation of Proposition \ref{kuwdhr12778}.}
\label {83473yruer938d}
\end{figure}

\section{Splitting and EN-Splitting}

Splitting and EN-splitting are two operations that can be performed on decorated trees.
Each one of these operations breaks a tree $\Teul \in \DT$ into two pieces $\Teul_1,\Teul_2 \in \DT$.
There are two types of splitting (at an edge and at a vertex) and two types of EN-splitting (at an edge and at a vertex).

\subsection*{Splitting}

\begin{definition} \label {5bwfj9gxjqk2cybwowc4ygr}
Let $\Teul = (\Veul,\Aeul,\Eeul,f,q) \in \DT$. By a {\it good pair\/} of $\Teul$ we mean an ordered pair $(z,v)$ such that

\begin{enumerati}

\item $z \in \Veul$, $v \in \Veul \cup \Aeul$ and $\{z,v\}$ is an edge of $\Teul$;

\item exactly one element $\alpha_0$ of $\Aeul_0 \setminus \{v\}$ is adjacent to $z$;

\item each element $\alpha$ of $(\Veul \cup \Aeul) \setminus \{ \alpha_0, v \}$ adjacent to $z$
satisfies $\alpha \in \Aeul$, $f(\alpha)=1$ and $q( \{z,\alpha\} , z ) = 1$;

\item $q(e,z) d + p(z,e) = 0$, where $e = \{z,v\}$ and $d = \delta_z - 2$. 

\end{enumerati}
\end{definition}

\begin{remark}  \label {filledcirclegoodpair9283479837489}
If $(z,v)$ is a good pair of $\Teul \in \DT$ then, near $z$, $\Teul$ must look like one of the following pictures:
$$
\begin{array}{ccc}
\scalebox{.9}{\begin{picture}(32,15)(-35,-11)
\put(-30,0){\circle{1}}
\put(-32,0){\makebox(0,0)[r]{\footnotesize $v$}}
\put(-29.5,0){\line(1,0){14}}
\put(-15,0){\circle*{1}}
\put(-14.5,0){\line(1,0){7.7}} \put(-7,0.18){\makebox(0,0)[l]{\footnotesize $<\! d$}}
\put(-15,-.5){\vector(0,-1){9.5}}
\put(-16,-9){\makebox(0,0)[r]{\footnotesize $\alpha_0$}}
\put(-14,-9){\makebox(0,0)[l]{\footnotesize $(0)$}}
\put(-15,1){\makebox(0,0)[b]{\footnotesize $z$}}
\put(-30.35355,0.35355){\line(-1,1){4}}
\put(-30.35355,-0.35355){\line(-1,-1){4}}
\end{picture}}
& \quad &
\scalebox{.9}{\begin{picture}(32,15)(-35,-11)
\put(-28,1){\makebox(0,0)[b]{\footnotesize $v$}}
\put(-15.5,0){\vector(-1,0){14}}
\put(-15,0){\circle*{1}}
\put(-14.5,0){\line(1,0){7.7}} \put(-7,0.18){\makebox(0,0)[l]{\footnotesize $<\! d$}}
\put(-15,-.5){\vector(0,-1){9.5}}
\put(-14,-9){\makebox(0,0)[l]{\footnotesize $(0)$}}
\put(-16,-9){\makebox(0,0)[r]{\footnotesize $\alpha_0$}}
\put(-30,0){\makebox(0,0)[r]{\footnotesize $(s)$}}
\put(-15,1){\makebox(0,0)[b]{\footnotesize $z$}}
\end{picture}} \\
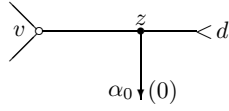
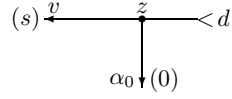
\text{\footnotesize Good pair $(z,v)$ with $v \in \Veul$.}
&&
\text{\footnotesize Good pair $(z,v)$ with $v \in \Aeul$.}
\end{array}
$$
where $d \in \Nat$ and $s \in \Integ$, and where we assume that $q(\{z,v\},z)d + p(z,\{z,v\}) = 0$ in both pictures.
The case $d=0$ is allowed.
Condition (iii) of Def.\ \ref{5bwfj9gxjqk2cybwowc4ygr} allows us to use the notation ``\begin{picture}(11.5,1)(0,-.5)
\put(0.5,0){\line(1,0){6.5}}
\put(6.7,0.2){\makebox(0,0)[l]{\footnotesize $<\! d$}}
\end{picture}'' defined in paragraph \ref{CROCHETpco9v0239jd0OiiJq0wjd}.
The fact that $z$ is represented by a filled circle
``$\begin{picture}(2,1)(-1,-.5) \put(0,0){\circle*{1}} \end{picture}$'' suggests that $N_z=0$ (see \ref{CROCHETpco9v0239jd0OiiJq0wjd}),
which is indeed the case by the next remark.
\end{remark}

\begin{remark} \label {o987543wwwm2kokqwdftxz8gij0i}
If $(z,v)$ is a good pair of $\Teul \in \DT$ (with notation as in Def.\ \ref{5bwfj9gxjqk2cybwowc4ygr}), then
$$
N_{\alpha_0}=0 \quad \text{and} \quad N_{z}= 0 .
$$
Indeed, $N_{\alpha_0}=q(e,z) d + p(z,e) = 0$ and $N_{z}= q( \{z,\alpha_0\},z ) N_{\alpha_0} = 0$.
In particular, $-N_z(\delta_z-2)-N_{\alpha_0}(\delta_{\alpha_0}-2) = 0$, i.e.,
the contribution of $z$ and $\alpha_0$ in the calculation of $M(\Teul)$ is zero.
\end{remark}

\begin{definition}  \label {jkc65432o396r839f0f6n30}
Let $\Teul = (\Veul,\Aeul,\Eeul,f,q) \in \DT$.
Given an edge $e = \{v_1,v_2\}$ of $\Teul$ (where $v_1,v_2 \in \Veul \cup \Aeul$), 
we proceed to define two trees $\Teul_1, \Teul_2 \in \DT$.
\begin{equation} \label {picture9893647239dj943-1}
\raisebox{-6\unitlength}{\begin{picture}(53,14)(-30,-7)
\put(-25,0){\makebox(0,0)[r]{$\Teul\,:$}}
\put(-18,4){\makebox(0,0)[rb]{$\mathscr{B}_1$}}
\put(18,4){\makebox(0,0)[lb]{$\mathscr{B}_2$}}

\put(-10,0){\circle{1}}
\put(-10,2){\makebox(0,0)[b]{\footnotesize $v_1$}}
\put(-09.5,0){\line(1,0){19}}
\put(0,2){\makebox(0,0)[b]{\footnotesize $e$}}

\put(10,0){\circle{1}}
\put(10,2){\makebox(0,0)[b]{\footnotesize $v_2$}}
\put(09.5,0){\line(-1,0){14}}

\put(-10.35355,0.35355){\line(-1,1){4}}
\put(-10.35355,-0.35355){\line(-1,-1){4}}
\put(-12,0){\oval(13,12)}

\put(10.35355,0.35355){\line(1,1){4}}
\put(10.35355,-0.35355){\line(1,-1){4}}
\put(12,0){\oval(13,12)}

\end{picture}} \text{(edge decorations not indicated).}
\end{equation}
(Note that $v_1,v_2$ are allowed to be arrows, even though the pictures \eqref{picture9893647239dj943-1}
and \eqref{picture9893647239dj943-2} suggest the contrary.)
Define 
$$
d = \gcd\big( p(v_1,e), p(v_2,e) \big) . 
$$ 
Note that $d \in \Nat$.
If $d=0$, define $(x_1,a_1) = (0,1) = (x_2,a_2)$.
If $d \neq 0$, define
$$
\text{$a_i = \frac{p(v_i,e)}{d}$ for each $i=1,2$,\quad $x_1=-a_2$,\ \ $x_2=-a_1$.}
$$
Observe that $\gcd(x_1,a_1) = 1 = \gcd(x_2,a_2)$ in all cases.

To define $\Teul_1$ and $\Teul_2$,
we begin by deleting $e$ from $\Teul$ (we do not delete $v_1$ and $v_2$); we now have two connected components, $\Ceul_1$ and $\Ceul_2$,
where $\Ceul_1$ contains $v_1$ and $\Ceul_2$ contains $v_2$.

Let $i \in \{1,2\}$; we obtain $\Teul_i$ by performing the following operations on $\Ceul_i$:
\begin{itemize}

\item add one vertex $z_i$ and one edge $\{v_i,z_i\}$;
\item add $d+1$ arrows $\alpha_{i,0}, \alpha_{i,1}, \dots, \alpha_{i,d}$ and $d+1$ edges $\{z_i,\alpha_{i,j}\}$ $(0 \le j \le d)$;
\item decorate $\alpha_{i,0}$ by $(0)$ and each of $\alpha_{i,1}, \dots, \alpha_{i,d}$ by $(1)$;
\item let the decoration of the edge $\{z_i,\alpha_{i,0}\}$ near $z_i$ be $a_i$;
\item let the decoration of the edge $\{v_i,z_i\}$ near $z_i$ be $x_i$;
\item let the decoration of the edge $\{v_i,z_i\}$ near $v_i$ be $q(e,v_i)$ (computed in $\Teul$);
\item for each $j=1,\dots,d$, let the decoration of $\{z_i,\alpha_{i,j}\}$ near $z_i$ be $1$;
\item for each $j=0,\dots,d$, let the decoration of $\{z_i,\alpha_{i,j}\}$ near $\alpha_{i,j}$ be $1$.
\end{itemize}
It is understood that the part of $\Teul_i$ in the oval $\mathscr{B}_i$, in picture \eqref{picture9893647239dj943-2},
is identical to the part of $\Teul$ in the oval $\mathscr{B}_i$ in picture \eqref{picture9893647239dj943-1}
(this includes all edge decorations near $v_i$).
This defines $\Teul_i \in \DT$.

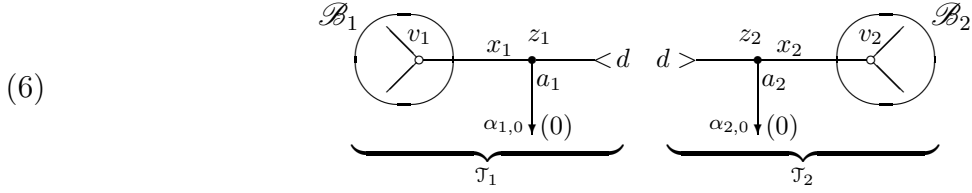
\begin{equation} \label {picture9893647239dj943-2}
\raisebox{-12\unitlength}{\begin{picture}(84,24)(-42,-17)
\put(-38,4){\makebox(0,0)[rb]{$\mathscr{B}_1$}}
\put(38,4){\makebox(0,0)[lb]{$\mathscr{B}_2$}}

\put(-30,0){\circle{1}}
\put(-30,2){\makebox(0,0)[b]{\footnotesize $v_1$}}
\put(-29.5,0){\line(1,0){14}}
\put(-15,0){\circle*{1}}
\put(-14.5,0){\line(1,0){7.7}} \put(-7,0.18){\makebox(0,0)[l]{\footnotesize $<\! d$}}
\put(-15,-.5){\vector(0,-1){9.5}}
\put(-14,-9){\makebox(0,0)[l]{\footnotesize $(0)$}}
\put(-16,-9){\makebox(0,0)[r]{\tiny $\alpha_{1,0}$}}
\put(-21,-14){\makebox(0,0){$\displaystyle\underbrace{\rule{36\unitlength}{0mm}}_{\Teul_1}$}}

\put(-14,2){\makebox(0,0)[b]{\footnotesize $z_1$}}
\put(14,2){\makebox(0,0)[b]{\footnotesize $z_2$}}

\put(-17.5,1.5){\makebox(0,0)[dr]{\footnotesize $x_1$}}
\put(-14.5,-3){\makebox(0,0)[l]{\footnotesize $a_1$}}

\put(17.5,1.5){\makebox(0,0)[dl]{\footnotesize $x_2$}}
\put(15.5,-3){\makebox(0,0)[l]{\footnotesize $a_2$}}

\put(30,0){\circle{1}}
\put(30,2){\makebox(0,0)[b]{\footnotesize $v_2$}}
\put(29.5,0){\line(-1,0){14}}
\put(15,0){\circle*{1}}
\put(14.5,0){\line(-1,0){7.7}} \put(7,0.18){\makebox(0,0)[r]{\footnotesize $d >$}}
\put(15,-.5){\vector(0,-1){9.5}}
\put(16,-9){\makebox(0,0)[l]{\footnotesize $(0)$}}
\put(14,-9){\makebox(0,0)[r]{\tiny $\alpha_{2,0}$}}
\put(21,-14){\makebox(0,0){$\displaystyle\underbrace{\rule{36\unitlength}{0mm}}_{\Teul_2}$}}

\put(-30.35355,0.35355){\line(-1,1){4}}
\put(-30.35355,-0.35355){\line(-1,-1){4}}
\put(-32,0){\oval(13,12)}

\put(30.35355,0.35355){\line(1,1){4}}
\put(30.35355,-0.35355){\line(1,-1){4}}
\put(32,0){\oval(13,12)}

\end{picture}}
\end{equation}
Observe the following (refer to Def.\ \ref{5bwfj9gxjqk2cybwowc4ygr}):
\begin{equation} \label {pd9iui237125fs72b32376iefcbuz}
\text{$(z_i,v_i)$ is a good pair of $\Teul_i$, for each $i=1,2$.}
\end{equation}

As we already noted, $\Teul_1, \Teul_2 \in \DT$.
We say that $\Teul_1$ and $\Teul_2$ are obtained by {\it splitting $\Teul$ at $e$,}
and we write 
$$
\Spl(\Teul,e) = \{ \Teul_1, \Teul_2 \} .
$$
The natural number $d$ is called the {\it degree\/} of the splitting $\Spl(\Teul,e)$.
\end{definition}

See Ex.\ \ref{8f790hbYBiD9485thBdfdcy82fjx3sBhNt87} for an example of splitting of the type defined above, i.e., a splitting {\it at an edge}.
The next definition introduces the notion of splitting {\it at a vertex\/}:

\begin{definition} \label {4512943mnbvxclfayqu53678xaw3cdrvubhlpmhlm0hkjnt7yfgd}
Let $\Teul \in \DT$, let $v \in \Veul$, let $E_v$ denote the set of all edges of $\Teul$ that are incident to $v$,
and let $\Pscr$ be a $2$-prepartition of $E_v$ (see Def.\ \ref{87tr432hb8re85g78s87jbuf}).

Consider the pair $(\Teul_0,e) = \Edz(\Teul,v,\Pscr)$ (see Def.\ \ref{983tet3w23hwryhj4djne}) and note that
it makes sense to consider $\Spl(\Teul_0,e)$, the splitting of $\Teul_0$ at $e$ as in Def.\ \ref{jkc65432o396r839f0f6n30}.

The {\it splitting of $\Teul$ at $(v,\Pscr)$} is denoted $\Spl(\Teul,v,\Pscr)$, and is defined by 
$$
\Spl(\Teul,v,\Pscr) = \Spl(\Teul_0,e) .
$$
We define the {\it degree\/} of the splitting $\Spl(\Teul,v,\Pscr)$ to be equal to the degree of the splitting $\Spl(\Teul_0,e)$.
\end{definition}

\begin{lemma}  \label {dBfo23bn6i8404f8hfwA7ylfaroj}
Let $\Teul \in \DT$ and suppose that $\xi$ satisfies one of the following conditions:
\begin{enumerati}

\item $\xi$ is an edge of $\Teul$;

\item $\xi = (v, \Pscr)$ for some vertex $v$ of $\Teul$ and some $2$-prepartition $\Pscr$ of the set of all edges of $\Teul$ incident to $v$.

\end{enumerati}
Consider $\Spl(\Teul,\xi) = \{ \Teul_1, \Teul_2 \}$ and let $d$ be the degree of the splitting $\Spl(\Teul,\xi)$.
Then the following hold:
\begin{enumerata}

\item $\Teul_1, \Teul_2 \in \DT$

\item $M( \Teul_1 ) + M( \Teul_2 ) = M( \Teul )$

\item $F( \Teul_1 ) + F( \Teul_2 ) = F( \Teul ) + 2d$

\item Assume that $d \neq 0$, let the notation for $\Teul_1,\Teul_2$ be as in \eqref{picture9893647239dj943-2},
and consider the edge $\{v_i,z_i\}$ of $\Teul_i$ for each $i=1,2$.
\begin{enumerata}

\item If $\xi$ is the edge $\{v_1,v_2\}$ of $\Teul$ and $v_1,v_2 \in \Veul$ then for each $i \in \{1,2\}$
the determinant of the edge $\{v_i,z_i\}$ of $\Teul_i$ is equal to  $-\frac{N_{v_i}}{d}$, where $N_{v_i}$ is computed in $\Teul$.

\item If $\xi = (v, \Pscr)$ then for each $i \in \{1,2\}$ the determinant of the edge $\{v_i,z_i\}$ of $\Teul_i$
is equal to  $-\frac{N_{v}}{d}$, where $N_{v}$ is computed in $\Teul$.

\end{enumerata}

\end{enumerata}
\end{lemma}

\begin{proof}
Case (i). We use the following notation:  $\xi = e = \{v_1,v_2\}$, the $z_i$ and $\alpha_{i,j}$ are as in Def.\ \ref{jkc65432o396r839f0f6n30},
$\Teul = (\Veul,\Aeul,\Eeul,f,q)$ and $\Teul_i = (\Veul^{(i)},\Aeul^{(i)},\Eeul^{(i)},f^{(i)},q^{(i)})$ ($i=1,2$).
The fact that $\Teul_1, \Teul_2 \in \DT$ was noted in  Def.\ \ref{jkc65432o396r839f0f6n30}, so (a) does not need a proof.

(b) Let $i \in \{1,2\}$.
Let $X_i =  (\Veul^{(i)} \cup \Aeul_0^{(i)}) \setminus \{ z_i, \alpha_{i,0} \}$.
By \eqref{pd9iui237125fs72b32376iefcbuz}, $(z_i,v_i)$ is a good pair of $\Teul_i$; so, by Rem.\ \ref{o987543wwwm2kokqwdftxz8gij0i}, 
$z_i$ and $\alpha_{i,0}$  contribute $0$ to the calculation of $M(\Teul_i)$;
so $M(\Teul_i) = -\sum_{w \in X_i} N_w^{(i)}(\delta_w^{(i)}-2)$.
For each $w \in X_i$ we have  $N_w^{(i)} = N_w$ and $\delta^{(i)}_w = \delta_w$,
so $M(\Teul_i) = -\sum_{w \in X_i} N_w(\delta_w-2)$.
Since $X_1 \cup X_2 = \Veul \cup \Aeul_0$ and $X_1 \cap X_2 = \emptyset$, we obtain $M(\Teul_1) + M(\Teul_2) = M(\Teul)$.

(c) Let $A_1 = \Aeul(v_2,e)$ and $A_2 = \Aeul(v_1,e)$ (see Def.\ \ref{c9v39rf0eX9e4np8glr9t8}). Then $A_1 \cup A_2 = \Aeul \setminus \Aeul_0$ and $A_1 \cap A_2 = \emptyset$,
so $F(\Teul) = \sum_{ \alpha \in A_1 } F( \alpha )  +  \sum_{ \alpha \in A_2 } F( \alpha )$.
Moreover, for each $i \in \{1,2\}$ we have
$\Aeul^{(i)} \setminus \Aeul^{(i)}_0 = A_i \cup \{ \alpha_{i,1}, \dots, \alpha_{i,d} \}$,
$A_i \cap \{ \alpha_{i,1}, \dots, \alpha_{i,d} \} = \emptyset$,
$F^{(i)}( \alpha_{i,j} ) = 1$ for all $j \in \{1,\dots,d\}$, and $F^{(i)}(\alpha) = F(\alpha)$ for all $\alpha \in A_i$,
so $F(\Teul_i) =  \sum_{j=1}^d F^{(i)}( \alpha_{i,j} ) + \sum_{ \alpha \in A_i } F^{(i)}( \alpha )
= d + \sum_{ \alpha \in A_i } F( \alpha )$.  Thus, $F( \Teul_1 ) + F( \Teul_2 ) = F( \Teul ) + 2d$. 

(d) Let $(i,j) \in \{ (1,2), (2,1) \}$.  If we write $e_i = \{v_i,z_i\}$ then
\begin{align*}
\textstyle
\det^{(i)}( e_i ) &= q^{(i)}(e_i,v_i) x_i - Q^{(i)}(e_i,v_i) a_i \\
&= \textstyle q(e,v_i) \left( - \frac{ p(v_j,e) }{d} \right) - Q(e,v_i) \frac{ p(v_i,e) }{d} = -N_{v_i}/d.
\end{align*}

This proves case (i).

Case (ii).  Let $(\Teul_0,e) = \Edz(\Teul,v,\Pscr)$.
We have $\Teul_0 \in \DT$ and $M(\Teul_0) = M(\Teul)$ by Def.\ \ref{983tet3w23hwryhj4djne}.
Since $\Spl(\Teul_0,e) = \{ \Teul_1, \Teul_2 \}$, Case (i) implies that 
$\Teul_1, \Teul_2 \in \DT$,
$M( \Teul_1 ) + M( \Teul_2 ) = M( \Teul_0 ) = M( \Teul )$
and $F( \Teul_1 ) + F( \Teul_2 ) = F( \Teul_0 ) + 2d$.
Since $F(\Teul_0) = F(\Teul)$ is clear, assertions (a--c) are proved.
Assertion (d) follows from Case (i) together with \eqref{8765rfdmpldnpwrfp8jmlrwpuy65dr43gfd4w228834yj}.
\end{proof}

\begin{remark}
Let $\Teul \in \DT$, let $v \in \Veul$, let $e$ be a dead end incident to $v$ and let $a = q(e,v)$.
Consider the splitting of $\Teul$ at $e$, $\Spl(\Teul,e) = \{ \Teul_1, \Teul_2 \}$, as in Def.\ \ref{jkc65432o396r839f0f6n30}.
\begin{equation*} 
\scalebox{.9}{\raisebox{-1\unitlength}{\begin{picture}(24,22)(-15,-17)

\put(-10,0){\circle{1}}
\put(-12,0){\makebox(0,0)[r]{\footnotesize $v$}}
\put(-9.5,0){\vector(1,0){14}}
\put(-3,-1){\makebox(0,0)[t]{\footnotesize $e$}}

\put(-7,1){\makebox(0,0)[b]{\footnotesize $a$}}

\put(5,0){\makebox(0,0)[l]{\footnotesize $(0)$}}
\put(3,1){\makebox(0,0)[b]{\footnotesize $\alpha$}}

\put(-10.35355,0.35355){\line(-1,1){4}}
\put(-10.35355,-0.35355){\line(-1,-1){4}}
\put(-4,-9){\makebox(0,0){$\displaystyle\underbrace{\rule{24\unitlength}{0mm}}_{\Teul}$}}

\end{picture}}} \qquad
\scalebox{.9}{\begin{picture}(84,22)(-42,-17)

\put(-30,0){\circle{1}}
\put(-32,0){\makebox(0,0)[r]{\footnotesize $v$}}
\put(-29.5,0){\line(1,0){14}}
\put(-15,0){\circle*{1}}
\put(-14.5,1){\makebox(0,0)[b]{\footnotesize $z_1$}}
\put(-14.5,0){\line(1,0){7.7}} \put(-7,0.18){\makebox(0,0)[l]{\footnotesize $<\! d$}}
\put(-15,-.5){\vector(0,-1){9.5}}
\put(-14,-9){\makebox(0,0)[l]{\footnotesize $(0)$}}
\put(-21,-14){\makebox(0,0){$\displaystyle\underbrace{\rule{36\unitlength}{0mm}}_{\Teul_1}$}}

\put(-17.5,1.5){\makebox(0,0)[dr]{\footnotesize $x_1$}}
\put(-14.5,-3){\makebox(0,0)[l]{\footnotesize $a_1$}}

\put(17.5,1.5){\makebox(0,0)[dl]{\footnotesize $x_2$}}
\put(15.5,-3){\makebox(0,0)[l]{\footnotesize $a_2$}}

\put(15.5,0){\vector(1,0){14}}
\put(15,0){\circle*{1}}
\put(14.5,1){\makebox(0,0)[b]{\footnotesize $z_2$}}
\put(14.5,0){\line(-1,0){7.7}} \put(7,0.18){\makebox(0,0)[r]{\footnotesize $d >$}}
\put(15,-.5){\vector(0,-1){9.5}}
\put(16,-9){\makebox(0,0)[l]{\footnotesize $(0)$}}
\put(30,0){\makebox(0,0)[l]{\footnotesize $(0)$}}
\put(28,1){\makebox(0,0)[b]{\footnotesize $\alpha$}}
\put(21,-14){\makebox(0,0){$\displaystyle\underbrace{\rule{36\unitlength}{0mm}}_{\Teul_2}$}}

\put(-30.35355,0.35355){\line(-1,1){4}}
\put(-30.35355,-0.35355){\line(-1,-1){4}}

\end{picture}}
\end{equation*}
Note that $(z_1,v)$ (resp.\ $(z_2,\alpha)$) is a good pair of $\Teul_1$ (resp.\ $\Teul_2$).
Assuming that $a\neq 0$, we have $p(v,e)=0$ and $p(\alpha,e) = \frac{N_v}{a}$ (in $\Teul$),
so $d = \left| \frac{N_v}{a} \right|$.  From now-on, assume that $a\neq0$ and $N_v \neq 0$.
Then Def.\ \ref{jkc65432o396r839f0f6n30} gives\footnote{We use the notation $\sgn(x) = \frac{x}{|x|}$ for $x \in \Reals \setminus\{0\}$.}
$$
\textstyle
a_1=0, \quad x_1 = - \sgn(\frac{N_v}{a}), \quad a_2 = \sgn(\frac{N_v}{a}), \quad \text{and} \quad x_2=0 .
$$
It immediately follows that $M(\Teul_2) = \frac{N_v}a$ and $F(\Teul_2) = d = \left| \frac{N_v}{a} \right|$,
so Lemma \ref{dBfo23bn6i8404f8hfwA7ylfaroj} implies that $M(\Teul_1) = M(\Teul) -\frac{N_v}a$
and $F(\Teul_1) = F(\Teul) +  \left| \frac{N_v}{a} \right|$.
\end{remark}

\begin{example}  \label {8f790hbYBiD9485thBdfdcy82fjx3sBhNt87}
Consider the tree $\Teul \in \DT$ and the edge $e$ of $\Teul$ depicted in Ex.\ \ref{894Y3yhbQwc34eeruGVfef8}.
Then $\Spl(\Teul,e) = \{\Teul_1,\Teul_2\}$ is as follows:
$$
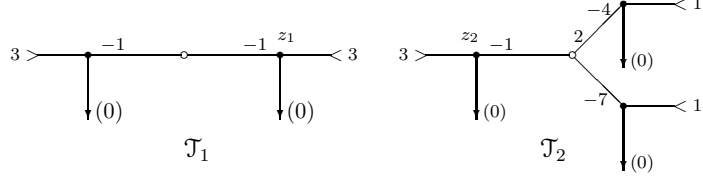
\scalebox{.85}{\begin{picture}(50,26)(-25,-18)
\put(0,0){\circle{1}}
\put(-15,0){\circle*{1}}
\put(15,0){\circle*{1}}
\put(.5,0){\line(1,0){14.5}}
\put(-.5,0){\line(-1,0){14.5}}
\put(15,0){\vector(0,-1){10}}
\put(16,-9){\makebox(0,0)[l]{\footnotesize $(0)$}}
\put(-15,0){\vector(0,-1){10}}
\put(-14,-9){\makebox(0,0)[l]{\footnotesize $(0)$}}
\put(-13,.5){\makebox(0,0)[bl]{\tiny $-1$}}
\put(13,.5){\makebox(0,0)[br]{\tiny $-1$}}
\put(16,2){\makebox(0,0)[b]{\tiny $z_{1}$}}
\put(-15,0){\line(-1,0){8.0}} \put(-22.5,0.13){\makebox(0,0)[r]{\tiny $3 >$}}
\put(15,0){\line(1,0){8.0}} \put(22.5,0.13){\makebox(0,0)[l]{\tiny $<3$}}
\put(0,-16){$\Teul_1$}
\end{picture}}
\qquad
\scalebox{.85}{\begin{picture}(45,26)(4,-18)
\put(14,2){\makebox(0,0)[b]{\tiny $z_{2}$}}
\put(15,0){\circle*{1}}
\put(30,0){\circle{1}}
\put(15,0){\line(-1,0){8.0}} \put(7.5,0.13){\makebox(0,0)[r]{\tiny $3 >$}}
\put(15,-.5){\vector(0,-1){9.5}}
\put(16,-9){\makebox(0,0)[l]{\tiny $(0)$}}
\put(17,.5){\makebox(0,0)[bl]{\tiny $-1$}}
\put(31,2){\makebox(0,0)[b]{\tiny $2$}}
\put(29.5,0){\line(-1,0){14}}
\put(30.35355,0.35355){\line(1,1){8}}
\put(30.35355,-0.35355){\line(1,-1){8}}
\put(38,8){\circle*{1}}
\put(38,7.5){\vector(0,-1){9.5}}
\put(39,-1){\makebox(0,0)[l]{\tiny $(0)$}}
\put(36,6){\makebox(0,0)[rb]{\tiny $-4$}}
\put(38,8){\line(1,0){8.0}} \put(45.7,8.07){\makebox(0,0)[l]{\tiny $< 1$}}
\put(38,-8){\line(1,0){8.0}} \put(45.7,-7.94){\makebox(0,0)[l]{\tiny $< 1$}}
\put(38,-8){\circle*{1}}
\put(38,-8.5){\vector(0,-1){9.5}}
\put(39,-17){\makebox(0,0)[l]{\tiny $(0)$}}
\put(35.5,-6){\makebox(0,0)[rt]{\tiny $-7$}}
\put(25,-16){$\Teul_2$}
\end{picture}}
$$
Direct calculation from the pictures of $\Teul,\Teul_1,\Teul_2$ gives:
\begin{equation}  \label {X87G65edfgOhjkc2E093C8j7we}
\begin{array}{ccc}
M(\Teul) = -9, & M(\Teul_{1}) = 0, & M(\Teul_{2}) = -9, \\
F(\Teul) = 5, & F(\Teul_{1}) = 6, &  F(\Teul_{2}) = 5.
\end{array}
\end{equation}
Lemma \ref{dBfo23bn6i8404f8hfwA7ylfaroj} asserts that $M( \Teul_1 ) + M( \Teul_2 ) = M( \Teul )$ and
(since the splitting is of degree $3$) $F( \Teul_1 ) + F( \Teul_2 ) = F( \Teul ) + 6$.
By \eqref{X87G65edfgOhjkc2E093C8j7we}, we see that these equalities do hold in this example.
\end{example}

\subsection*{EN-splitting}

Eisenbud and Neumann have an operation that they call ``splicing'', which is a way to glue trees together in order to form more complicated trees.
We shall now define an operation that is the inverse of splicing (it cuts a tree into two pieces).
We call it the Eisenbud-Neumann splitting, or EN-splitting, denoted ``$\ENSpl$''.

\begin{definition}  \label {654n3227g8f98dnby4yt6yjud7un3e7yhsgwr}
Let $\Teul = (\Veul,\Aeul,\Eeul,f,q) \in \DT$ and let $e = \{v_1,v_2\}$ be an edge of $\Teul$ ($v_1,v_2 \in \Veul \cup \Aeul$).
We proceed to define trees $\Teul_1, \Teul_2 \in \DT$.

Delete $e$ from $\Teul$ (but do not delete $v_1$ and $v_2$); we now have two connected components, $\Ceul_1$ and $\Ceul_2$,
where $\Ceul_1$ contains $v_1$ and $\Ceul_2$ contains $v_2$.
Let $i \in \{1,2\}$; we obtain $\Teul_i$ by performing the following operations on $\Ceul_i$:
\begin{itemize}

\item add one arrow $\alpha_{i}$ and one edge $\{v_i, \alpha_i\}$;
\item let the decoration of $\alpha_{i}$ be $(p_i)$, where $p_i = p(v_i,e)$;
\item let the decoration of the edge $\{v_i,\alpha_{i}\}$ near $v_i$ be $q(e,v_i)$;
\item let the decoration of the edge $\{v_i,\alpha_{i}\}$ near $\alpha_i$ be $1$;
\item all other decorations remain as they are in $\Teul$.
\end{itemize}
This defines $\Teul_i$, and it is obvious that $\Teul_i \in \DT$.
\begin{equation} \label {picture9893647239dj943-3}
\raisebox{-7\unitlength}{\begin{picture}(78,14)(-39,-10)

\put(-20,0){\circle{1}}
\put(-22,0){\makebox(0,0)[r]{\footnotesize $v_1$}}
\put(-19.5,0){\vector(1,0){12}}
\put(-7,0){\makebox(0,0)[l]{\footnotesize $(p_1)$}}
\put(-9,1){\makebox(0,0)[b]{\tiny $\alpha_1$}}
\put(9,1){\makebox(0,0)[b]{\tiny $\alpha_2$}}

\put(20,0){\circle{1}}
\put(22,0){\makebox(0,0)[l]{\footnotesize $v_2$}}
\put(19.5,0){\vector(-1,0){12}}
\put(7,0){\makebox(0,0)[r]{\footnotesize $(p_2)$}}

\put(-20.35355,0.35355){\line(-1,1){4}}
\put(-20.35355,-0.35355){\line(-1,-1){4}}

\put(20.35355,0.35355){\line(1,1){4}}
\put(20.35355,-0.35355){\line(1,-1){4}}

\put(-16,-8){\makebox(0,0){$\displaystyle\underbrace{\rule{27\unitlength}{0mm}}_{\Teul_1}$}}
\put(16,-8){\makebox(0,0){$\displaystyle\underbrace{\rule{27\unitlength}{0mm}}_{\Teul_2}$}}

\end{picture}}
\end{equation}

We use the notation 
$$
\ENSpl(\Teul,e) = \{ \Teul_1, \Teul_2 \} 
$$
and we say that $\{ \Teul_1, \Teul_2 \}$ is the {\it EN-splitting of $\Teul$ at $e$.}
The natural number $d = \gcd(p_1,p_2)$ is called the {\it degree\/} of the EN-splitting $\ENSpl(\Teul,e)$.
We also define the {\it type\/} of $\ENSpl(\Teul,e)$ to be the element $t$ of $\{-1,0,1\}$ given by
$$
t = \begin{cases}
0, & \text{if the number of elements $i$ of $\{1,2\}$ that satisfy $p_i=0$ is even;} \\
-1, & \text{if $p_1 p_2 = 0$ and $p_1+p_2<0$;} \\
1, & \text{if $p_1 p_2 = 0$ and $p_1+p_2>0$.}
\end{cases}
$$
\end{definition}

\begin{definition}
Let $\Teul \in \DT$, let $v \in \Veul$, let $E_v$ denote the set of all edges of $\Teul$ that are incident to $v$,
and let $\Pscr$ be a $2$-prepartition of $E_v$ (see Def.\ \ref{87tr432hb8re85g78s87jbuf}).

Consider the pair $(\Teul_0,e) = \Edz(\Teul,v,\Pscr)$ (see Def.\ \ref{983tet3w23hwryhj4djne}) and note that
it makes sense to consider $\ENSpl(\Teul_0,e)$, the EN-splitting of $\Teul_0$ at $e$ defined in \ref{654n3227g8f98dnby4yt6yjud7un3e7yhsgwr}.

The {\it EN-splitting of $\Teul$ at $(v,\Pscr)$} is denoted $\ENSpl(\Teul,v,\Pscr)$, and is defined by 
$$
\ENSpl(\Teul,v,\Pscr) = \ENSpl(\Teul_0,e) .
$$
Recall from Def.\ \ref{654n3227g8f98dnby4yt6yjud7un3e7yhsgwr} that the EN-splitting $\ENSpl(\Teul_0,e)$ has a degree (which is a natural number)
and a type (which is $-1$, $0$ or $1$).
We define the {\it degree\/} (resp.\ {\it type\/}) of the EN-splitting $\ENSpl(\Teul,v,\Pscr)$ to be equal to
the degree (resp.\ type) of $\ENSpl(\Teul_0,e)$.
\end{definition}

\begin{lemma}  \label {SPLICINGdBfo23bn6i8404f8hfwA7ylfaroj}
Let $\Teul \in \DT$ and suppose that $\xi$ satisfies one of the following conditions:
\begin{enumerati}

\item $\xi$ is an edge of $\Teul$;

\item $\xi = (v, \Pscr)$ for some vertex $v$ of $\Teul$ and some $2$-prepartition $\Pscr$ of the set of all edges of $\Teul$ incident to $v$.

\end{enumerati}
Let $d \in \Nat$ and $t \in \{-1,0,1\}$ be, respectively, the degree and type of the EN-splitting $\ENSpl(\Teul,\xi) = \{ \Teul_1, \Teul_2 \}$.
Then the following hold:
\begin{enumerata}

\item $\Teul_1, \Teul_2 \in \DT$

\item $M( \Teul_1 ) + M( \Teul_2 ) = M( \Teul ) + td$

\item $F( \Teul_1 ) + F( \Teul_2 ) = F(\Teul) + (2 - |t|)d$. 

\item $M(\Teul) + F(\Teul) \equiv \big( M(\Teul_1) + F(\Teul_1) \big) + \big( M(\Teul_2) + F(\Teul_2) \big) \pmod2$.

\end{enumerata}
\end{lemma}

\begin{proof}
Note that (d) follows from (a--c) and the fact that $t + (2 - |t|)$ is even.
So it suffices to prove (a--c).

Case (i). We use the following notation:  $\xi = e = \{v_1,v_2\}$,
$\alpha_1, \alpha_2, p_1, p_2$ are as in Def.\ \ref{654n3227g8f98dnby4yt6yjud7un3e7yhsgwr},
$\Teul = (\Veul,\Aeul,\Eeul,f,q)$ and $\Teul_i = (\Veul^{(i)},\Aeul^{(i)},\Eeul^{(i)},f^{(i)},q^{(i)})$ ($i=1,2$).
The fact that $\Teul_1, \Teul_2 \in \DT$ was noted in  Def.\ \ref{654n3227g8f98dnby4yt6yjud7un3e7yhsgwr}, so (a) does not need a proof.
We prove (b) and (c) together.
Let $i \in \{1,2\}$ and let $j$ be the unique element of $\{1,2\} \setminus \{i\}$.
Let $X_i = \setspec{ x \in \Veul \cup \Aeul_0 }{ \text{$e$ is not in $\gamma_{v_i,x}$} }$
and $A_i = \setspec{ \alpha \in \Aeul \setminus \Aeul_0 }{ \text{$e$ is not in $\gamma_{v_i,\alpha}$} }$.
Then
$$
\Veul^{(i)} \cup \Aeul_0^{(i)} =
\begin{cases}
X_i, & \text{if $p_i \neq 0$}, \\
X_i \cup \{ \alpha_i \}, & \text{if $p_i = 0$},
\end{cases}
\quad
\Aeul^{(i)} \setminus \Aeul_0^{(i)} =
\begin{cases}
A_i \cup \{ \alpha_i \}, & \text{if $p_i \neq 0$}, \\
A_i, & \text{if $p_i = 0$}.
\end{cases}
$$
For every $x \in X_i$, we have $N_x = N^{(i)}_x$ and $\delta_x = \delta^{(i)}_x$;
for every $\alpha \in A_i$, we have $F(\alpha) = F^{(i)}(\alpha)$.
Moreover, if $p_i=0$ then $N^{(i)}_{\alpha_i} = p_j = td$,
and if $p_i\neq0$ then $F^{(i)}(\alpha_i) = \gcd\big( p_i, p^{(i)}( \alpha_i, \{\alpha_i,v_i\} ) \big) = \gcd( p_i, p_j ) = d$.
Thus,
\begin{align*}
M(\Teul_i) &=
\begin{cases}
\sum_{x \in X_i} N_x(\delta_x-2), & \text{if $p_i \neq 0$}, \\
td + \sum_{x \in X_i} N_x(\delta_x-2), & \text{if $p_i = 0$},
\end{cases} \\
F(\Teul_i) &=
\begin{cases}
d + \sum_{\alpha \in A_i} F(\alpha), & \text{if $p_i \neq 0$}, \\
\sum_{\alpha \in A_i} F(\alpha), & \text{if $p_i = 0$}.
\end{cases}
\end{align*}
Since $X_1 \cup X_2 = \Veul \cup \Aeul_0$ and $X_1 \cap X_2 = \emptyset$, we have $M(\Teul) = \sum_{x \in X_1} N_x(\delta_x-2) + \sum_{x \in X_2} N_x(\delta_x-2)$;
since $A_1 \cup A_2 = \Aeul \setminus \Aeul_0$ and $A_1 \cap A_2 = \emptyset$, we have $F(\Teul) = \sum_{\alpha \in A_1} F(\alpha) + \sum_{\alpha \in A_2} F(\alpha)$.
Assertions (b) and (c) follow from these observations.
This proves case (i).

Case (ii).  Let $(\Teul_0,e) = \Edz(\Teul,v,\Pscr)$.
We have $\Teul_0 \in \DT$ and $M(\Teul_0) = M(\Teul)$ by Def.\ \ref{983tet3w23hwryhj4djne}.
Since $\ENSpl(\Teul_0,e) = \{ \Teul_1, \Teul_2 \}$, Case (i) implies that 
$\Teul_1, \Teul_2 \in \DT$,
$M( \Teul_1 ) + M( \Teul_2 ) = M( \Teul_0 )+td = M( \Teul )+td$
and $F( \Teul_1 ) + F( \Teul_2 ) = F( \Teul_0 ) + (2-|t|)d$.
Since $F(\Teul_0) = F(\Teul)$ is clear, we see that (a--c) are true.
\end{proof}

\begin{example}
Consider the following $\Teul \in \DT$ and the edge $e$ of $\Teul$:
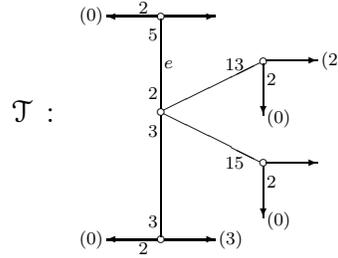
\begin{equation*}
\Teul\,:\ \ 
\raisebox{-17mm}{\scalebox{.85}{\begin{picture}(40,38.5)(-12,-6.5)

\put(0,-4.5){\line(0,1){19}}
\put(0,15.5){\line(0,1){14}}

\put(-.5,28){\makebox(0,0)[rt]{\tiny $5$}}
\put(-.5,17){\makebox(0,0)[rb]{\tiny $2$}}
\put(-.5,13){\makebox(0,0)[rt]{\tiny $3$}}
\put(-.5,-3){\makebox(0,0)[rb]{\tiny $3$}}
\put(.5,22.5){\makebox(0,0)[l]{\tiny $e$}}

\put(0,-5){\circle{1}}
\put(0,15){\circle{1}}
\put(0,30){\circle{1}}

\put(-.5,-5){\vector(-1,0){8}}
\put(-9,-5){\makebox(0,0)[r]{\tiny $(0)$}}
\put(-2,-5.5){\makebox(0,0)[rt]{\tiny $2$}}

\put(-.5,30){\vector(-1,0){8}}
\put(-9,30){\makebox(0,0)[r]{\tiny $(0)$}}
\put(-2,30.5){\makebox(0,0)[rb]{\tiny $2$}}

\put(.5,-5){\vector(1,0){8}}
\put(9,-5){\makebox(0,0)[l]{\tiny $(3)$}}

\put(.5,30){\vector(1,0){8}}

\put(0.4472,15.2236){\line(2,1){15}}
\put(0.4472,14.7764){\line(2,-1){15}}
\put(16,23){\circle{1}}
\put(16,7){\circle{1}}

\put(16,22.5){\vector(0,-1){8}}
\put(16.5,14){\makebox(0,0)[l]{\tiny $(0)$}}
\put(16.5,21){\makebox(0,0)[lt]{\tiny $2$}}
\put(16.5,23){\vector(1,0){8}}
\put(25,23){\makebox(0,0)[l]{\tiny $(2)$}}

\put(16,6.5){\vector(0,-1){8}}
\put(16.5,-2){\makebox(0,0)[l]{\tiny $(0)$}}
\put(16.5,5){\makebox(0,0)[lt]{\tiny $2$}}
\put(16.5,7){\vector(1,0){8}}

\put(13,22.5){\makebox(0,0)[dr]{\tiny $13$}}
\put(13,8){\makebox(0,0)[tr]{\tiny $15$}}

\end{picture}}}
\end{equation*}
We have $\ENSpl(\Teul,e) = \{ \Teul_1, \Teul_2 \}$ where
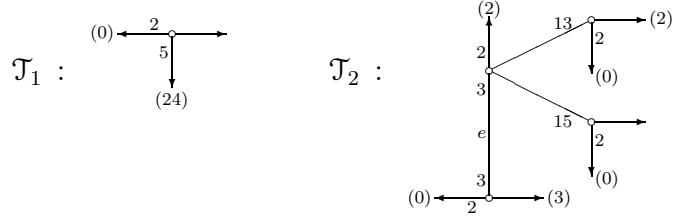
\begin{equation*}
\Teul_1\,:\ \ 
\raisebox{-4mm}{\scalebox{.85}{\begin{picture}(22,14)(-12,18)

\put(-.5,28){\makebox(0,0)[rt]{\tiny $5$}}
\put(0,30){\circle{1}}

\put(-.5,30){\vector(-1,0){8}}
\put(-9,30){\makebox(0,0)[r]{\tiny $(0)$}}
\put(-2,30.5){\makebox(0,0)[rb]{\tiny $2$}}

\put(.5,30){\vector(1,0){8}}
\put(0,29.5){\vector(0,-1){8}}
\put(0,21){\makebox(0,0)[t]{\tiny $(24)$}}

\end{picture}}}
\qquad\quad
\Teul_2\,:\ \ 
\raisebox{-17mm}{\scalebox{.85}{\begin{picture}(40,33)(-12,-6.5)

\put(0,-4.5){\line(0,1){19}}
\put(0,15.5){\vector(0,1){8}}
\put(0,26){\makebox(0,0)[t]{\tiny $(2)$}}

\put(-.5,17){\makebox(0,0)[rb]{\tiny $2$}}
\put(-.5,13){\makebox(0,0)[rt]{\tiny $3$}}
\put(-.5,-3){\makebox(0,0)[rb]{\tiny $3$}}
\put(-.5,5){\makebox(0,0)[r]{\tiny $e$}}

\put(0,-5){\circle{1}}
\put(0,15){\circle{1}}

\put(-.5,-5){\vector(-1,0){8}}
\put(-9,-5){\makebox(0,0)[r]{\tiny $(0)$}}
\put(-2,-5.5){\makebox(0,0)[rt]{\tiny $2$}}

\put(.5,-5){\vector(1,0){8}}
\put(9,-5){\makebox(0,0)[l]{\tiny $(3)$}}

\put(0.4472,15.2236){\line(2,1){15}}
\put(0.4472,14.7764){\line(2,-1){15}}
\put(16,23){\circle{1}}
\put(16,7){\circle{1}}

\put(16,22.5){\vector(0,-1){8}}
\put(16.5,14){\makebox(0,0)[l]{\tiny $(0)$}}
\put(16.5,21){\makebox(0,0)[lt]{\tiny $2$}}
\put(16.5,23){\vector(1,0){8}}
\put(25,23){\makebox(0,0)[l]{\tiny $(2)$}}

\put(16,6.5){\vector(0,-1){8}}
\put(16.5,-2){\makebox(0,0)[l]{\tiny $(0)$}}
\put(16.5,5){\makebox(0,0)[lt]{\tiny $2$}}
\put(16.5,7){\vector(1,0){8}}

\put(13,22.5){\makebox(0,0)[dr]{\tiny $13$}}
\put(13,8){\makebox(0,0)[tr]{\tiny $15$}}

\end{picture}}}
\end{equation*}
Calculating directly from these pictures, we find:
\begin{equation}  \label {876543wefoiujnlg3u29n84d765rfbg}
\begin{array}{ccc}
M(\Teul) = -273, & M(\Teul_{1}) = -29, & M(\Teul_{2}) = -244, \\
F(\Teul) = 5, & F(\Teul_{1}) = 3, &  F(\Teul_{2}) = 6.
\end{array}
\end{equation}
As the EN-splitting $\ENSpl(\Teul,e)$ has degree $d=2$ and type $t=0$, Lemma \ref{SPLICINGdBfo23bn6i8404f8hfwA7ylfaroj} asserts that 
$M( \Teul_1 ) + M( \Teul_2 ) = M( \Teul )$ and $F( \Teul_1 ) + F( \Teul_2 ) = F(\Teul) + 4$.
By \eqref{876543wefoiujnlg3u29n84d765rfbg}, we see that these equalities are verified in this example.
\end{example}

\section{The number $M(\Teul) + F(\Teul)$ is even}
\label{SEC:ThenumberMTeulFTeuliseven}

The EN-splitting is used in the proof of the following fundamental result.

\begin{theorem} \label {8374te23nfAiUoMnbed53oifn}
If $\Teul \in \DT$ then $M(\Teul) + F(\Teul)$ is even.
\end{theorem}

\begin{proof}
We use the following temporary notations:
for each $\Teul \in \DT$, let $\| \Teul \|$ denote the number of vertices of $\Teul$ and let $\Phi(\Teul) = M(\Teul) + F(\Teul)$.
Also remember that {\bf all congruences in this proof are tacitly assumed to be modulo \boldmath $2$.}

Suppose that $\Teul \in \DT$ satisfies $\| \Teul \| > 1$.
Then we can choose an edge $e = \{v_1,v_2\}$ of $\Teul$ such that $v_1,v_2$ are vertices.
Let $\{\Teul_1,\Teul_2\} = \ENSpl(\Teul,e)$;
then Lemma \ref{SPLICINGdBfo23bn6i8404f8hfwA7ylfaroj} implies that $\Teul_1, \Teul_2 \in \DT$ and $\Phi(\Teul) \equiv \Phi(\Teul_1) + \Phi(\Teul_2)$,
and it is clear that $\| \Teul_i \| < \| \Teul \|$ for $i=1,2$.
So, in order to prove the Theorem, it suffices to verify that
\begin{equation} \label {9o8b87nm6545n43m2dfvcxzswedcvgt56yhbnju78ijmki90plkjhre}
\text{$\Phi(\Teul) \equiv 0$ for every $\Teul \in \DT$ such that $\| \Teul \| \le 1$.}
\end{equation}
Let us prove \eqref{9o8b87nm6545n43m2dfvcxzswedcvgt56yhbnju78ijmki90plkjhre}. Let $\Teul \in \DT$ be such that $\| \Teul \| = 0$.
If $\Teul$ is empty then $M(\Teul) = 0 = F(\Teul)$, so $\Phi(\Teul) \equiv 0$.
Assume that $\Teul$ is not empty.
Then $\Teul$ is $(i) \! \longleftrightarrow \! (j)$ for some $i,j \in \Integ$.
If $i=0=j$ then again $M(\Teul) = 0 = F(\Teul)$, so $\Phi(\Teul) \equiv 0$.
If $i,j \neq 0$ then $M(\Teul) = 0$ and $F(\Teul) = \gcd(i,j) + \gcd(i,j)$ is even, so $\Phi(\Teul) \equiv 0$.
If $i=0$ and $j\neq0$ then $M(\Teul) \equiv j$ and $F(\Teul) = \gcd(j,0) = |j| \equiv j$, so  $\Phi(\Teul) \equiv 0$.
So $\Phi(\Teul) \equiv 0$ for every $\Teul \in \DT$ such that $\| \Teul \| = 0$.

Let $\Teul \in \DT$ be such that $\| \Teul \| = 1$. Let $v$ be the unique vertex of $\Teul$
and let $\alpha_1, \dots, \alpha_m, \beta_1, \dots, \beta_n$ ($m,n\ge0$) be the 
distinct arrows of $\Teul$, where each $\alpha_i$ is decorated by $(0)$ and each $\beta_i$ is decorated by $(s_i)$ for some $s_i \in \Integ \setminus \{0\}$.
Also let $a_i = q( \{v,\alpha_i\},v )$ for all $i \in \{1,\dots,m\}$ and $b_i = q( \{v,\beta_i\},v )$ for all $i \in \{1,\dots,n\}$.
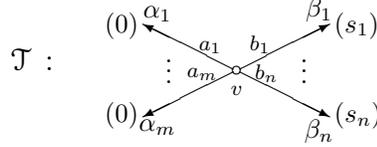
$$
{\begin{picture}(34,16)(-16,-8)
\put(-30,0){$\Teul\,:$}

\put(0,0){\circle{1}}
\put(0,-2){\makebox(0,0)[t]{\tiny $v$}}

\put(-10.5,6.5){\makebox(0,0)[b]{\footnotesize $\alpha_1$}}
\put(-13,6){\makebox(0,0)[r]{\footnotesize $(0)$}}
\put(-10.5,-6.5){\makebox(0,0)[t]{\footnotesize $\alpha_m$}}
\put(-13,-6){\makebox(0,0)[r]{\footnotesize $(0)$}}
\put(-5,2.25){\makebox(0,0)[bl]{\tiny $a_1$}}
\put(-2.5,-1.25){\makebox(0,0)[br]{\tiny $a_m$}}
\put(-9,1){\makebox(0,0){\footnotesize $\vdots$}}
\put(9,1){\makebox(0,0){\footnotesize $\vdots$}}

\put(11,6.5){\makebox(0,0)[b]{\footnotesize $\beta_1$}}
\put(13,6){\makebox(0,0)[l]{\footnotesize $(s_1)$}}
\put(11,-6.5){\makebox(0,0)[t]{\footnotesize $\beta_n$}}
\put(13,-6){\makebox(0,0)[l]{\footnotesize $(s_n)$}}
\put(4.5,2.25){\makebox(0,0)[br]{\tiny $b_1$}}
\put(2.5,-1.5){\makebox(0,0)[bl]{\tiny $b_n$}}

\put(0.4472,0.2236){\vector(2,1){12}}
\put(0.4472,-0.2236){\vector(2,-1){12}}
\put(-0.4472,0.2236){\vector(-2,1){12}}
\put(-0.4472,-0.2236){\vector(-2,-1){12}}

\end{picture}}
$$

Assume that $a_i \equiv 0$ for some $i$. We may assume that $a_1 \equiv 0$.
Then $a_i \equiv 1$ for all $i>1$ and $b_i \equiv 1$ for all $i$, so
$N_{\alpha_1} \equiv \sum_{i=1}^n s_i$, $N_{\alpha_i} \equiv 0$ for all $i>1$, and $N_v\equiv0$; so $M(\Teul) \equiv  \sum_{i=1}^n s_i$.
To compute $F(\Teul)$, we need the following trivial fact:
\begin{equation} \label {9875m12m3n45bv5x68swiuehgfehwg345v5syizjr6}
\text{if $a,a',b,b' \in \Integ$ satisfy $a \equiv a'$ and $b \equiv b'$, then $\gcd(a,b) \equiv \gcd(a',b')$.}
\end{equation}
Let $i \in \{1,\dots,n\}$; since $p(\beta_i, \{\beta_i,v\}) \equiv 0$, \eqref{9875m12m3n45bv5x68swiuehgfehwg345v5syizjr6} implies that
$$
F(\beta_i) = \gcd\big( s_i, p(\beta_i, \{\beta_i,v\}) \big) \equiv \gcd(s_i,0) \equiv s_i .
$$
So $F(\Teul) \equiv \sum_{i=1}^n s_i$ and hence $\Phi(\Teul) \equiv 0$.
This shows that  $\Phi(\Teul) \equiv 0$ whenever $a_i \equiv 0$ for some $i$.

Assume that $b_i \equiv 0$ for some $i$. We may assume that $b_1 \equiv 0$.
Then $b_i \equiv 1$ for all $i>1$ and $a_i \equiv 1$ for all $i$, so $N_v \equiv s_1$ and for each $i \in \{1,\dots,m\}$ we have
$N_{\alpha_i} \equiv s_1$. Thus, $M(\Teul) \equiv (m+n)s_1 + \sum_{i=1}^m s_1 \equiv n s_1$.
Using \eqref{9875m12m3n45bv5x68swiuehgfehwg345v5syizjr6}, we find that
$F(\beta_1) = \gcd\big( s_1, p(\beta_1, \{\beta_1,v\}) \big) \equiv \gcd\big( s_1, \sum_{i=2}^n s_i \big)$
and, for each $i>1$,
$F(\beta_i) = \gcd\big( s_i, p(\beta_i, \{\beta_i,v\}) \big) \equiv \gcd( s_i, s_1 )$, so 
$$
\textstyle
F(\Teul) \equiv \gcd\big( s_1, \sum_{i=2}^n s_i \big)  +  \sum_{i=2}^n \gcd(s_i,s_1) .
$$
If $s_1 \equiv 1$ then $F(\Teul) \equiv 1 +  \sum_{i=2}^n 1 \equiv n \equiv ns_1$,
and if $s_1 \equiv 0$ then $F(\Teul) \equiv \gcd\big( 0, \sum_{i=2}^n s_i \big)  +  \sum_{i=2}^n \gcd(s_i,0)
\equiv | \sum_{i=2}^n s_i |  +  \sum_{i=2}^n |s_i| \equiv \sum_{i=2}^n s_i  +  \sum_{i=2}^n s_i \equiv 0 \equiv ns_1$;
so $F(\Teul) \equiv ns_1$ in both cases. Consequently, $\Phi(\Teul)\equiv 0$
whenever $b_i \equiv 0$ for some $i$.

The last case to consider is $a_i \equiv 1$ and $b_j \equiv 1$ for all $i,j$.
Define $S = \sum_{i=1}^n s_i$. Then $N_{\alpha_i} \equiv S$ for all $i$ and $N_v \equiv S$, so $M(\Teul) \equiv (m+n)S + \sum_{i=1}^m S \equiv nS$.
For each $i \in \{1,\dots,n\}$, $F(\beta_i) \equiv \gcd( s_i, S-s_i ) \equiv \gcd( s_i, S )$, so 
$$
\textstyle
F(\Teul) \equiv \sum_{i=1}^n \gcd(s_i,S) .
$$
If $S\equiv0$ then $F(\Teul) \equiv \sum_{i=1}^n \gcd(s_i,0) \equiv \sum_{i=1}^n |s_i| \equiv \sum_{i=1}^n s_i \equiv 0 \equiv nS$,
and if $S\equiv1$ then $F(\Teul) \equiv \sum_{i=1}^n \gcd(s_i,1) \equiv \sum_{i=1}^n 1 \equiv n \equiv n S$; so $F(\Teul) \equiv nS$ in both cases.
Consequently, $\Phi(\Teul)\equiv 0$.

This proves \eqref{9o8b87nm6545n43m2dfvcxzswedcvgt56yhbnju78ijmki90plkjhre}, and completes the proof of the Theorem.
\end{proof}

\section{Genus and $\delta$-invariant}
\label {SEC:Genusanddeltainvariant}

\newlength{\longsomme}
\settowidth{\longsomme}{$\scriptstyle X \in \Pscr$}
\setlength{\longsomme}{1.3\longsomme}

\newcommand{\undersum}[1]{_{\makebox[\longsomme]{$\scriptstyle {#1}$}}}

This section proves formulae that compute $g(\Teul)$ and $\delta(\Teul)$ in terms of subtrees with fewer arrows.
Proposition \ref{9837487239rpejf9d88} treats the special case where all the arrows are decorated with $(0)$ or $(1)$,
and Theorem \ref{90323673738bbcgcbcnpss899aaw1a} gives the general case.
The section concludes with Examples \ref{o873646728378uejdwe9eo} and \ref{87652ewtg2s5tsx2qadfvg2xdj3eoo009pokl}, which illustrate these formulae.

\begin{notation}  \label {8927476hbgFfxfob93746dBBNE376rFuf82190uhsSQ}
Recall from the introduction that, given $\Teul \in \DT$, we define
$$
\textstyle g(\Teul) = \frac12 ( 2 - M(\Teul) - F(\Teul) )
\quad \text{and} \quad
\textstyle \delta(\Teul) = \frac12 ( F(\Teul) - M(\Teul) ) .
$$
Note that $g(\Teul), \delta(\Teul) \in \Integ$ by Theorem \ref{8374te23nfAiUoMnbed53oifn}.
We refer to these integers as the {\it genus\/} and {\it $\delta$-invariant\/} of $\Teul$, respectively.
\end{notation}

We shall now define (see \ref{pod293805h4s7b2w9310jdcjd}) the class $\DTr$ of rooted decorated trees and  the class $\DTpr$ of pseudo-rooted decorated trees,
where $\DTr \subset \DTpr$.
In fact we are mostly interested in $\DTr$, and we introduce the larger class $\DTpr$
because it is sometimes necessary to use the theory of pseudo-rooted trees for proving results about rooted trees.
Given $\Teul \in \DTpr$ and a nonempty subset $X$ of $\Aeul \setminus \Aeul_0$, an element $\Teul_X$ of $\DTpr$ is defined in \ref{092198c45w6evXgb4r8rhfShbg48}.
Then the remainder of the section is devoted to finding formulae that,
given $\Teul \in \DTpr$ and a partition $\{X_1, \dots, X_n\}$ of $\Aeul \setminus \Aeul_0$,
express $M(\Teul)$ (resp.\ $g(\Teul)$, $\delta(\Teul)$)
in terms of the numbers $M(\Teul_{X_i})$  (resp.\ $g(\Teul_{X_i})$, $\delta(\Teul_{X_i})$).

\begin{definition}  \label {pod293805h4s7b2w9310jdcjd}
Let $\Teul = (\Veul, \Aeul, \Eeul, f, q) \in \DT$.
A \textit{pseudo-root} of $\Teul$  is a vertex $v_0 \in \Veul$ that satisfies:
\begin{enumerati}

\item $q(e,v_0)=1$ for all edges $e$ incident to $v_0$;

\item for every $v \in \Veul \setminus \{v_0\}$, at most one edge $e$ incident to $v$ and not in $\gamma_{v_0,v}$
satisfies $q(e,v) \neq 1$.

\end{enumerati}
A \textit{root} of $\Teul$  is a vertex $v_0 \in \Veul$ that satisfies (i), (ii) and:
\begin{enumerati}

\item[(iii)] \label {condiiiofdefroot}
for every $v \in \Veul \setminus \{v_0\}$, all edges $e$ incident to $v$ and not in $\gamma_{v_0,v}$
satisfy $q(e,v) \ge 1$.

\end{enumerati}

A {\it decorated pseudo-rooted tree} (resp.\ a {\it decorated rooted tree\/}) is a tuple 
$(\Veul, \Aeul, \Eeul, f, q, v_0)$ such that $(\Veul, \Aeul, \Eeul, f, q) \in \DT$
and $v_0$ is a pseudo-root (resp.\ a root) of $(\Veul, \Aeul, \Eeul, f, q)$.
We write $\DTpr$ for the set of decorated pseudo-rooted trees,
and $\DTr$ for the set of decorated rooted trees; it is clear that $\DTr \subset \DTpr$.

If $\Teul = (\Veul, \Aeul, \Eeul, f, q, v_0) \in \DTpr$, we say that $v_0$ is {\it the\/} pseudo-root of $\Teul$.

If $\Teul = (\Veul, \Aeul, \Eeul, f, q, v_0) \in \DTr$, we say that $v_0$ is {\it the\/} root of $\Teul$.
\end{definition}

\begin{definition}  \label {092198c45w6evXgb4r8rhfShbg48}
Let $\Teul = (\Veul, \Aeul, \Eeul, f, q, v_0) \in \DTpr$.
For each nonempty subset $X$ of $\Aeul \setminus \Aeul_0$, we define an element $\Teul_X$ of $\DTpr$ as follows.
We begin by defining a tree $\Teul_X' \in \DT$ by declaring:
\begin{itemize}

\item the set of arrows of $\Teul_X'$ is $X$;

\item the vertices of $\Teul_X'$ are the vertices of $\Teul$ that are in $\gamma_{v_0,\alpha}$ for some $\alpha \in X$
(so $v_0$ is a vertex of $\Teul_X'$);

\item the edges of $\Teul_X'$ are the edges of $\Teul$ that are in $\gamma_{v_0,\alpha}$ for some $\alpha \in X$;

\item the arrows and edges of $\Teul_X'$ have the same decorations as in $\Teul$.

\end{itemize}
This defines $\Teul_X' \in \DT$.
For each vertex $v$ of $\Teul_X'$, define the integer $b_v = \prod_{e \in E_v^X} q(e,v)$
where $E_v^X$ denotes the set of edges $e$ of $\Teul$ that are incident to $v$ but
that are not edges of $\Teul_X'$, and where the numbers $q(e,v)$ are those of $\Teul$.

Finally, we obtain $\Teul_X$ from $\Teul_X'$ as follows: for each vertex $v$ of $\Teul_X'$ such that $b_v \neq 1$,
add an arrow $\alpha_v$ decorated by $(0)$ and an edge $\{v,\alpha_v\}$ whose decoration near $v$ (resp.\ near $\alpha_v$)
is equal to $b_v$ (resp.\ to $1$).
We make $\Teul_X$ an element of $\DTpr$ by declaring that $v_0$ is its pseudo-root.

If $|X|=1$, we simplify the notation: for each $\alpha \in \Aeul \setminus \Aeul_0$, we define $\Teul_\alpha = \Teul_{ \{\alpha\} }$.
\end{definition}

\begin{example}  \label {u98mwnbbpwe7gcc342}
Let $\Teul \in \DTpr$ be the following tree:
$$
\scalebox{.7}{\begin{picture}(110,26)(-55,-13)
\multiput(-40,0)(20,0){5}{\circle{1}}
\put(0,10){\circle{1}}
\put(-40,10){\circle{1}}
\put(-40,-10){\circle{1}}
\multiput(-39.5,0)(20,0){4}{\line(1,0){19}}
\put(0,.5){\line(0,1){9}}
\put(-40,-9.5){\line(0,1){9}}
\put(-40,.5){\line(0,1){9}}
\put(-40.5,0){\vector(-1,0){9.5}}
\put(-39.5,-10){\line(1,0){19}}
\put(-20,-10){\circle{1}}
\put(-40.5,-10){\vector(-1,0){9.5}}

\put(20,0.5){\vector(0,1){9.5}}
\put(20,-0.5){\vector(0,-1){9.5}}

\put(40,-0.5){\vector(0,-1){9.5}}
\put(40.5,0){\vector(1,0){9.5}}

\put(0,11){\makebox(0,0)[b]{\tiny $v_0$}}
\put(-49.5,1){\makebox(0,0)[b]{\tiny $\beta$}}
\put(21,9){\makebox(0,0)[l]{\tiny $\alpha$}}

\put(20,11){\makebox(0,0)[b]{\tiny $(2)$}}
\put(20,-11){\makebox(0,0)[t]{\tiny $(1)$}}

\put(51,0){\makebox(0,0)[l]{\tiny $(1)$}}
\put(40,-11){\makebox(0,0)[t]{\tiny $(2)$}}

\put(-51,0){\makebox(0,0)[r]{\tiny $(3)$}}
\put(-51,-10){\makebox(0,0)[r]{\tiny $(0)$}}

\put(-38,-10.5){\makebox(0,0)[tl]{\tiny $7$}}
\put(-22,-10.5){\makebox(0,0)[tr]{\tiny $-1$}}
\put(-39.5,2){\makebox(0,0)[bl]{\tiny $2$}}
\put(-2,.5){\makebox(0,0)[br]{\tiny $5$}}
\put(.5,2){\makebox(0,0)[bl]{\tiny $2$}}
\put(22,.5){\makebox(0,0)[bl]{\tiny $3$}}
\put(38,.5){\makebox(0,0)[br]{\tiny $-1$}}
\put(42,.5){\makebox(0,0)[bl]{\tiny $4$}}
\end{picture}}
$$
where $v_0$ is the pseudo-root of $\Teul$ and where the edge decorations that are not indicated are equal to $1$.
Then $\Teul_\alpha$,  $\Teul_{ \{\alpha,\beta\} }$ and  $\Teul_{ \Aeul \setminus \Aeul_0 }$ are:
\begin{gather*}
\Teul_\alpha: \  
\raisebox{-8mm}{\scalebox{.7}{\begin{picture}(56,16)(-2,-13)
\multiput(0,0)(20,0){3}{\circle{1}}
\multiput(0.5,0)(20,0){2}{\line(1,0){19}}
\put(40.5,0){\vector(1,0){9.5}}
\put(40, -.5){\vector(0,-1){9.5}}
\put(20,-0.5){\vector(0,-1){9.5}}
\put(40,-0.5){\vector(0,-1){9.5}}
\put(0,1){\makebox(0,0)[b]{\tiny $v_0$}}
\put(49,1){\makebox(0,0)[b]{\tiny $\alpha$}}
\put(20,-11){\makebox(0,0)[t]{\tiny $(0)$}}
\put(40,-11){\makebox(0,0)[t]{\tiny $(0)$}}
\put(20.5,-2){\makebox(0,0)[tl]{\tiny $5$}}
\put(18,.5){\makebox(0,0)[br]{\tiny $2$}}
\put(40.5,-2){\makebox(0,0)[tl]{\tiny $3$}}
\put(51,0){\makebox(0,0)[l]{\tiny $(2)$}}
\end{picture}}}\ \quad
\Teul_{ \{\alpha,\beta\} }: \ 
\raisebox{-8mm}{\scalebox{.7}{\begin{picture}(78,28)(-55,-14)
\multiput(-40,0)(20,0){4}{\circle{1}}
\put(0,10){\circle{1}}
\multiput(-39.5,0)(20,0){3}{\line(1,0){19}}
\put(0,.5){\line(0,1){9}}
\put(-40.5,0){\vector(-1,0){9.5}}
\put(20,0.5){\vector(0,1){9.5}}
\put(20,-0.5){\vector(0,-1){9.5}}
\put(-40,-0.5){\vector(0,-1){9.5}}
\put(0,11){\makebox(0,0)[b]{\tiny $v_0$}}
\put(-49.5,1){\makebox(0,0)[b]{\tiny $\beta$}}
\put(21,9){\makebox(0,0)[l]{\tiny $\alpha$}}
\put(20,11){\makebox(0,0)[b]{\tiny $(2)$}}
\put(20,-11){\makebox(0,0)[t]{\tiny $(0)$}}
\put(-51,0){\makebox(0,0)[r]{\tiny $(3)$}}
\put(-2,.5){\makebox(0,0)[br]{\tiny $5$}}
\put(.5,2){\makebox(0,0)[bl]{\tiny $2$}}
\put(20.5,-2){\makebox(0,0)[tl]{\tiny $3$}}
\put(-40,-11){\makebox(0,0)[t]{\tiny $(0)$}}
\put(-39.5,-2){\makebox(0,0)[tl]{\tiny $2$}}
\end{picture}}} \\
\Teul_{ \Aeul \setminus \Aeul_0 }: \ \raisebox{-8mm}{\scalebox{.7}{\begin{picture}(110,26)(-55,-13)
\multiput(-40,0)(20,0){5}{\circle{1}}
\put(0,10){\circle{1}}
\multiput(-39.5,0)(20,0){4}{\line(1,0){19}}
\put(0,.5){\line(0,1){9}}
\put(-40.5,0){\vector(-1,0){9.5}}
\put(-40,-0.5){\vector(0,-1){9.5}}
\put(20,0.5){\vector(0,1){9.5}}
\put(20,-0.5){\vector(0,-1){9.5}}
\put(40,-0.5){\vector(0,-1){9.5}}
\put(40.5,0){\vector(1,0){9.5}}
\put(0,11){\makebox(0,0)[b]{\tiny $v_0$}}
\put(-49.5,1){\makebox(0,0)[b]{\tiny $\beta$}}
\put(21,9){\makebox(0,0)[l]{\tiny $\alpha$}}
\put(20,11){\makebox(0,0)[b]{\tiny $(2)$}}
\put(20,-11){\makebox(0,0)[t]{\tiny $(1)$}}
\put(51,0){\makebox(0,0)[l]{\tiny $(1)$}}
\put(40,-11){\makebox(0,0)[t]{\tiny $(2)$}}
\put(-51,0){\makebox(0,0)[r]{\tiny $(3)$}}
\put(-39.5,-2){\makebox(0,0)[tl]{\tiny $2$}}
\put(-2,.5){\makebox(0,0)[br]{\tiny $5$}}
\put(.5,2){\makebox(0,0)[bl]{\tiny $2$}}
\put(22,.5){\makebox(0,0)[bl]{\tiny $3$}}
\put(38,.5){\makebox(0,0)[br]{\tiny $-1$}}
\put(42,.5){\makebox(0,0)[bl]{\tiny $4$}}
\put(-40,-11){\makebox(0,0)[t]{\tiny $(0)$}}
\end{picture}}}
\end{gather*}
\end{example}

\begin{remark} \label {09o239r8d5beg56vq28gxcv9w}
Let $\Teul \in \DTpr$. If $\alpha \in X \subseteq \Aeul \setminus \Aeul_0$ then $(\Teul_X)_\alpha = \Teul_\alpha$.
\end{remark}

\begin{lemma}  \label {c56gmntsbnebeg5vrbqivuwy54678ik}
Let $\Teul \in \DTpr$, let $X$ be a nonempty subset of $\Aeul \setminus \Aeul_0$ and let $w \in \Veul \cup \Aeul$ be such that
\begin{itemize}

\item $\delta_w = 1$, $w \neq v_0$ and $w \notin X$,

\item if $v_w$ denotes the unique element of $\Veul \cup \Aeul$ adjacent to $w$ then\\
$q( \{w,v_w\} , v_w ) = 1$ or $\delta_{v_w}=2$.

\end{itemize}
Let $\Teul^-$ be the tree obtained from $\Teul$ by removing $w$ and the edge $\{w,v_w\}$ (without removing $v_w$).
Then $X$ is a subset of $\Aeul(\Teul^-) \setminus \Aeul_0(\Teul^-)$ and $\Teul_X = (\Teul^-)_X$.
\end{lemma}

\begin{proof}
Since $\delta_w=1$, $w \neq v_0$ and $w \notin X$, it follows that
\begin{equation}  \label {ermi8ae4h8B23cx5wfghh}
\text{for each $\alpha \in X$, $w$ is not in $\gamma_{v_0,\alpha}$;}
\end{equation}
so the removal of $w$ does not affect the formation of $\Teul_X'$, i.e., we have  $\Teul_X' = (\Teul^-)_X'$.
Let $\Seul$ denote the tree $\Teul_X' = (\Teul^-)_X'$ and let $\Veul(\Seul)$ be the set of vertices of $\Seul$.
For each $u \in \Veul(\Seul)$, let $b_u$ (resp.\ $b_u^-$) be the product of the numbers $q(e,u)$ (resp.\ $q^-(e,u)$)
where $e$ runs in the set of all edges of $\Teul$ (resp.\ $\Teul^-$) that are incident to $u$ and are not in $\Seul$,
and where $q(e,u)$ (resp.\ $q^-(e,u)$) is computed in $\Teul$ (resp.\ $\Teul^-$).
We claim that $b_u = b_u^-$ for all $u \in \Veul(\Seul)$. Indeed, let $u \in \Veul(\Seul)$.
We have $w \notin \Veul(\Seul)$ by \eqref{ermi8ae4h8B23cx5wfghh}, so $w \neq u$.
If $w$ is not adjacent to $u$ in $\Teul$, then $b_u = b_u^-$ is obvious.
Assume that $w$ is adjacent to $u$ in $\Teul$.
Then $v_w = u$, so $q( \{w,u\} , u ) = 1$ or $\delta_u=2$.
Also, we have $b_u = q( \{w,u\} , u ) b_u^-$, so it suffices to show that $q( \{w,u\} , u ) = 1$.
If $q( \{w,u\} , u ) \neq 1$ then $\delta_{u} = 2$, and  $u \neq v_0$ because some edge decoration near $u$ is not $1$;
so for each $\alpha \in X$, $u$ is not in $\gamma_{v_0,\alpha}$.
This contradicts $u \in \Veul(\Seul)$, so we have shown that $q( \{w,u\} , u ) = 1$. Thus, $b_u = b_u^-$ for all $u \in \Veul(\Seul)$.
This together with $\Teul_X' = (\Teul^-)_X'$ implies that $\Teul_X = (\Teul^-)_X$.
\end{proof}

Note that $\Teul_{ \Aeul \setminus \Aeul_0 }$ is not necessarily equal to $\Teul$ (see Ex.\ \ref{u98mwnbbpwe7gcc342}).
So the following is not trivial:

\begin{lemma} \label {923hCbf9182uBe902hAwga0w91a}
If $\Teul \in \DTpr$ is such that $\Aeul \setminus \Aeul_0 \neq \emptyset$   then $M(\Teul) = M(\Teul_{ \Aeul \setminus \Aeul_0 })$.
In particular, if  $|\Aeul \setminus \Aeul_0| = 1$ and $\alpha$ denotes the unique element of $\Aeul \setminus \Aeul_0$ then $M(\Teul) = M(\Teul_\alpha)$.
\end{lemma}

\begin{proof}
\renewcommand{\DTpru}{\text{\rm\bf DT}'_{\text{\rm pr}}}
Let $\DTpru$ be the set of $\Teul \in \DTpr$ such that $\Aeul \setminus \Aeul_0 \neq \emptyset$.
Let $\bbX$ be the set of $\Teul \in \DTpru$ such that  $M(\Teul) = M(\Teul_{\Aeul \setminus \Aeul_0})$.
We have to show that $\bbX = \DTpru$.

Given $\Teul \in \DTpru$, define $C(\Teul) = \setspec{ w \in \Veul \cup \Aeul_0 }{ \text{$w \neq v_0$ and $\delta_w=1$} }$.
Given $w \in C(\Teul)$, let $v_w$ be the unique element of $\Veul \cup \Aeul$ adjacent to $w$ and let $r_w = q( \{w,v_w\}, v_w )$;
note that $\delta_{v_w} \ge2$ (because $w \neq v_0$ and $\Aeul \setminus \Aeul_0 \neq \emptyset$), so in particular $v_w \in \Veul$.
Let $C^*(\Teul) = \setspec{ w \in C(\Teul) }{ r_w = 1 \text{ or } \delta_{v_w}=2 }$.
We first prove that
\begin{equation} \label {Didedjwi6ytg33a4rytdtfeU9kBBdhf8}
\text{$\bbX$ contains all $\Teul \in \DTpru$ such that $C^*(\Teul) = \emptyset$.}
\end{equation}
Let $\Teul \in \DTpru$ be such that $C^*(\Teul) = \emptyset$.
Let $X = \Aeul \setminus \Aeul_0$ and consider the trees $\Teul_X'$ and $\Teul_X$ as in Def.\ \ref{092198c45w6evXgb4r8rhfShbg48}.
Given a vertex $u$ of $\Teul_X'$, 
consider the set $\Gamma_u$ of all paths $(x_0, \dots, x_m)$ in $\Teul$ such that $x_0 = u$, $m\ge1$, and $x_1$ is not in $\Teul_{X}'$ 
(i.e., $x_1$ is an arrow or a vertex of $\Teul$ which is not in $\Teul_X'$).
Note that if  $(x_0, \dots, x_m) \in \Gamma_u$ then $x_1, x_2, \dots, x_m$ are not in $\Teul_X'$.
We claim:
\begin{equation}  \label {fbrgxsiy46sy57lthhql5ygvr}
\begin{minipage}[t]{.8\textwidth}
If $\Gamma_u \neq \emptyset$ then $\Gamma_u = \{ (u,x_u) \}$ for some $x_u \in C(\Teul)$,
and moreover $q( \{u,x_u\} , u ) \neq 1$.
\end{minipage}
\end{equation}
To see this, assume that $\Gamma_u \neq \emptyset$ and choose $\gamma = (x_0, \dots, x_m) \in \Gamma_u$ which maximizes $m$.
Then $x_m \in C(\Teul)$. Since $x_m \notin C^*(\Teul)$, we have $q( \{x_m, x_{m-1}\}, x_{m-1})\neq1$ and $\delta_{x_{m-1}} > 2$.
Let us prove that $m=1$.
Indeed, suppose that $m>1$. Then $x_{m-1}$ is not in $\Teul_X'$.
Since $v_0$ is in $\Teul_X'$, it follows that $x_{m-1} \neq v_0$ and that the edge $\{x_{m-2},x_{m-1}\}$ is in $\gamma_{v_0,x_{m-1}}$.
Since $v_0$ is the pseudo-root of $\Teul$,
\begin{equation}  \label {B8765rfghjxuyt532A56uehetdry}
\begin{minipage}[t]{.8\textwidth}
at most one edge $e$ of $\Teul$ incident to $x_{m-1}$ and distinct from $\{x_{m-2},x_{m-1}\}$ satisfies $q(e,x_{m-1}) \neq 1$.
\end{minipage}
\end{equation}
Since  $\delta_{x_{m-1}} > 2$, we may choose $x_m' \in \Veul \cup \Aeul_0$ such that $x_m'$ is adjacent to $x_{m-1}$ and is not in $\gamma$.
Then $(x_0,\dots,x_{m-1},x_m') \in \Gamma_u$, maximality of $m$ implies that $x_m' \in C(\Teul)$,
and $x_m' \notin C^*(\Teul)$ implies that $q( \{x_m', x_{m-1}\}, x_{m-1})\neq1$;
this contradicts \eqref{B8765rfghjxuyt532A56uehetdry}, so we have shown that $m=1$.
Note that each element of $\Gamma_u$ is a path $(u,x)$ for some $x \in C(\Teul) \setminus C^*(\Teul)$,
and this implies that $q( \{u,x\} , u ) \neq 1$ and hence that $u \neq v_0$.
Also note that the edge $\{u,x\}$ is not in $\gamma_{v_0,u}$, 
so the fact that  $v_0$ is the pseudo-root of $\Teul$ implies that $| \Gamma_u | = 1$.
This completes the proof of \eqref{fbrgxsiy46sy57lthhql5ygvr}.

Let $U$ be the set of vertices $u$ of $\Teul_X'$ such that $\Gamma_u \neq \emptyset$.
For each $u \in U$, an element $x_u$ of $\Veul \cup \Aeul_0$ is defined by \eqref{fbrgxsiy46sy57lthhql5ygvr}.
This $x_u$ is adjacent to $u$ and satisfies  $\delta_{x_u}=1$, $x_u \neq v_0$ and $q( \{u,x_u\} , u ) \neq 1$.
It follows that $\Teul_X$ is obtained from $\Teul$ by performing the following operation:
\begin{quote}
For each $u \in U$ such that $x_u \in \Veul$, remove the vertex $x_u$ and the edge $\{u,x_u\}$,
and add an arrow $\alpha_u$ decorated by $(0)$ and an edge $\{u,\alpha_u\}$
whose decorations near $u$ and $\alpha_u$ are $q( \{u,x_u\},u )$ and $1$ respectively.
\end{quote}
It is clear that the above operation does not change the value of $M(\Teul)$, i.e., $M(\Teul) = M(\Teul_X) = M(\Teul_{ \Aeul \setminus \Aeul_0 })$.
So $\Teul \in \bbX$, which proves \eqref{Didedjwi6ytg33a4rytdtfeU9kBBdhf8}.

Given $\Teul\in\DTpru$, let $\|\Teul\|$ denote the number of edges of $\Teul$ and note that $\|\Teul\| \ge 1$.
If $\|\Teul\| = 1$ then $C(\Teul) = \emptyset$, so $C^*(\Teul) = \emptyset$ and hence $\Teul \in \bbX$ by \eqref{Didedjwi6ytg33a4rytdtfeU9kBBdhf8}.
Let $n>1$ and suppose that $\Teul \in \bbX$ for all $\Teul \in \DTpru$ satisfying $\| \Teul \| < n$.
Consider  $\Teul \in \DTpru$ such that $\| \Teul \| = n$ and let us prove that $\Teul \in \bbX$.
In view of \eqref{Didedjwi6ytg33a4rytdtfeU9kBBdhf8}, we may assume that  $C^*(\Teul) \neq \emptyset$.
Choose $w \in C^*(\Teul)$ and note that $r_w = 1$ or $\delta_{v_w}=2$.
Let $\Teul^-$ be the tree obtained from $\Teul$ by deleting $w$ and the edge $\{w,v_w\}$ (we do not delete $v_w$).
Observe that the set $\Aeul \setminus \Aeul_0$ is the same for the two trees $\Teul$ and $\Teul^-$, and define $X = \Aeul \setminus \Aeul_0$.
Then $\Teul^- \in \DTpru$ and $\|\Teul^-\| < n$, so $\Teul^- \in \bbX$, i.e., $M(\Teul^-) = M((\Teul^-)_X)$.
To complete the proof, it suffices to verify that 
\begin{equation}  \label {782heQdnc837XsT46etcvn}
M(\Teul) = M(\Teul^-) \quad \text{and} \quad \Teul_X = (\Teul^-)_X .
\end{equation}
Indeed, if this is true then $M(\Teul) = M(\Teul^-) = M( (\Teul^-)_X ) = M( \Teul_X )$, so $\Teul \in \bbX$ and we are done.

Recall that $r_w = 1$ or $\delta_{v_w}=2$.
If $r_w=1$ then  $M(\Teul) = M(\Teul^-)$ follows from Lemma  \ref{0v82u395u4i83hrf993md4tj09}.
If $\delta_{v_w}=2$ then let $e$ be the unique edge of $\Teul$ incident to $v_w$ and distinct from $\{v_w,w\}$.
Let $N_{v_w}^-$ and $\delta_{v_w}^-$ denote the numbers $N_{v_w}$ and $\delta_{v_w}$ computed in $\Teul^-$.
Then $N_w = p(w, \{w,v_w\}) = p(v_w,e) = N_{v_w}^-$, so
\begin{align*}
-N_{v_w}(\delta_{v_w}-2) - N_w(\delta_{w}-2) &= -N_{v_w}(2-2) - N_w(1-2) = - N^-_{v_w}(1-2) \\
&= - N^-_{v_w}(\delta^-_{v_w}-2)
\end{align*}
and hence  $M(\Teul) = M(\Teul^-)$.
Lemma \ref{c56gmntsbnebeg5vrbqivuwy54678ik} implies that $\Teul_X = (\Teul^-)_X$, so \eqref{782heQdnc837XsT46etcvn} is proved.
As we already explained, this proves the Lemma.
\end{proof}


\begin{notation}
Let $\Teul = (\Veul, \Aeul, \Eeul, f, q) \in \DT$.
\begin{enumerata}

\item If $X, Y$ are subsets of $\Aeul \setminus \Aeul_0$, define
$\displaystyle I(X,Y)
\ = \sum\undersum{ \begin{smallmatrix} (\alpha,\beta) \in X \times Y \\ \alpha \neq \beta \end{smallmatrix} }  Q^*(\gamma_{\alpha,\beta}) f(\alpha) f(\beta)$.

\item  If $\alpha \in \Aeul \setminus \Aeul_0$, define
$$
I(\alpha)
\ = \ \ \sum\undersum{ \beta \in (\Aeul \setminus \Aeul_0) \setminus \{\alpha\} } Q^*(\gamma_{\alpha,\beta}) f(\alpha) f(\beta)
\ = \  I\big( \{\alpha\}, (\Aeul\setminus\Aeul_0) \setminus \{\alpha\} \big) .
$$

\end{enumerata}
\end{notation}

\begin{remark} \label {fgfvgertjuIPPWeCdEew8hju834hnf893krfcnfFVH6t8}
Let $\Teul \in \DT$.
\begin{enumerate}

\item If $X,Y \subseteq \Aeul \setminus \Aeul_0$ then $I(X,Y) = I(Y,X)$.
\item If $X,Y \subseteq \Aeul \setminus \Aeul_0$ and $X \cap Y = \emptyset$ then
\begin{align*}
I(X \cup Y, X \cup Y) &= I(X,X) + I(X,Y) + I(Y,X) + I(Y,Y) \\
&= I(X,X) + 2I(X,Y) + I(Y,Y).
\end{align*}

\item If $\Pscr$ is a partition of $\Aeul \setminus \Aeul_0$ then
$$
I( \Aeul \setminus \Aeul_0 , \Aeul \setminus \Aeul_0 )
= \sum_{X \in \Pscr} I(X,X)
\ +  \sum\undersum{ \begin{smallmatrix} (X,Y) \in \Pscr^2 \\ X \neq Y \end{smallmatrix} } I(X,Y) \, .
$$
\end{enumerate}
\end{remark}

\begin{notation}
If $\Teul \in \DT$ and $\alpha \in \Aeul \setminus \Aeul_0$, define
$$
M_\alpha(\Teul) = - \sum\undersum{v \in \Veul \cup \Aeul_0} x_{v,\alpha} ( \delta_v - 2 ).
$$
Clearly, $M(\Teul) = \sum_{ \alpha \in \Aeul \setminus \Aeul_0 } M_\alpha(\Teul)$.
\end{notation}

\begin{lemma}  \label {0d9f09gf909s990q9e0j3m3mn2mnm2moeiu83cdwx4}
Let $\Teul = (\Veul,\Aeul,\Eeul,f,q,v_0) \in \DTpr$ and suppose that $f(\alpha) = 1$ for all  $\alpha \in \Aeul \setminus \Aeul_0$.
Then
\begin{equation}  \label {9238uhe6c3newclaim92388e7gd}
M(\Teul_\alpha) - M_\alpha(\Teul) = I(\alpha), \quad \text{for every $\alpha \in \Aeul \setminus \Aeul_0$.}
\end{equation}
\end{lemma}

\begin{proof}
Let $\DTprp$ denote the set of $\Teul \in \DTpr$ satisfying $f(\alpha) = 1$ for all  $\alpha \in \Aeul \setminus \Aeul_0$.
We prove the Lemma by induction on $F(\Teul)$ (which is equal to $|\Aeul \setminus \Aeul_0|$ for all $\Teul \in \DTprp$).
We first show that
\begin{equation}  \label {9827cZbzxnBrvvEBsdfgRTh9}
\text{If $\Teul \in \DTprp$ and $F(\Teul) \le 1$ then $\Teul$ satisfies \eqref{9238uhe6c3newclaim92388e7gd}.}
\end{equation}
This is trivial if $F(\Teul) = 0$. If $F(\Teul)=1$ then $\Aeul \setminus \Aeul_0$ has a unique element $\alpha$,
so $M(\Teul) = M(\Teul_\alpha)$ by Lemma \ref{923hCbf9182uBe902hAwga0w91a};
since $I(\alpha)=0$ and $M_\alpha(\Teul) = M(\Teul)$, we see that $\Teul$ satisfies \eqref{9238uhe6c3newclaim92388e7gd}.
So \eqref{9827cZbzxnBrvvEBsdfgRTh9} is true.

Let $n>1$ and assume that
\begin{equation} \label {87afvCaeWvz00mmfQ3zxk6hxdCxvdh7d25x}
\text{if $\Teul \in \DTprp$ and $F(\Teul) < n$ then $\Teul$ satisfies \eqref{9238uhe6c3newclaim92388e7gd}.}
\end{equation}
Consider $\Teul = (\Veul,\Aeul,\Eeul,f,q,v_0) \in \DTprp$ such that $F(\Teul)=n$.
Let $\alpha \in  \Aeul \setminus \Aeul_0$ and let us prove that
\begin{equation} \label {56c5v6b2m239sf8gdfvlkka2gcjg5r65u}
M(\Teul_\alpha) - M_\alpha(\Teul) = I(\alpha) .
\end{equation}
Since $n>1$, we have $(\Aeul \setminus \Aeul_0) \setminus \{\alpha\} \neq \emptyset$.
We shall first prove \eqref{56c5v6b2m239sf8gdfvlkka2gcjg5r65u} under the additional assumption that the pair $(\Teul,\alpha)$ satisfies
\begin{equation} \label {9z8xiejhtgcdusajw4wajw3n4bc56d78kikjhfyds}
\begin{minipage}[t]{.85\textwidth}
there exists $\beta \in (\Aeul \setminus \Aeul_0) \setminus \{\alpha\}$ such that
the decoration of $\{\beta,v_\beta\}$ near $v_\beta$ is $1$,
where $v_\beta$ denotes the unique vertex adjacent to $\beta$.
\end{minipage}
\end{equation}
Assume that $(\Teul,\alpha)$ satisfies \eqref{9z8xiejhtgcdusajw4wajw3n4bc56d78kikjhfyds}
and let $\beta$ and $v_\beta$ be as in \eqref{9z8xiejhtgcdusajw4wajw3n4bc56d78kikjhfyds}.
Let $\Teul'$ be the tree obtained from $\Teul$ by deleting the arrow $\beta$ and the edge $\{\beta,v_\beta\}$ (without deleting $v_\beta$).
Since $\Teul' \in \DTprp$  and $F(\Teul')<n$, \eqref{87afvCaeWvz00mmfQ3zxk6hxdCxvdh7d25x} implies that 
$M((\Teul')_\alpha) - M_\alpha(\Teul') = I'(\alpha)$ (where $I'(\alpha)$ means $I(\alpha)$ computed in $\Teul'$).
Keeping in mind that the decoration of $\{\beta,v_\beta\}$ near $v_\beta$ is $1$,
we obtain $\Teul_\alpha = (\Teul')_\alpha$ (and hence $M(\Teul_\alpha) = M((\Teul')_\alpha)$) from Lemma \ref{c56gmntsbnebeg5vrbqivuwy54678ik},
and we see that $I(\alpha) - I'(\alpha) = I(\{\alpha\},\{\beta\})$ and 
\begin{align*}
M_\alpha(\Teul) - M_\alpha(\Teul')
&= -x_{v_\beta, \alpha}(\delta_{v_\beta}-2) + x'_{v_\beta, \alpha}(\delta'_{v_\beta}-2) \\
&= -x_{v_\beta, \alpha}(\delta_{v_\beta}-2) + x_{v_\beta, \alpha}(\delta_{v_\beta}-3) 
= -x_{v_\beta, \alpha} = - I(\{\alpha\},\{\beta\}) .
\end{align*}
These equalities together with $M((\Teul')_\alpha) - M_\alpha(\Teul') = I'(\alpha)$ imply that \eqref{56c5v6b2m239sf8gdfvlkka2gcjg5r65u}
is true (when $(\Teul,\alpha)$ satisfies \eqref{9z8xiejhtgcdusajw4wajw3n4bc56d78kikjhfyds}).

Next, we show that \eqref{56c5v6b2m239sf8gdfvlkka2gcjg5r65u} is true without assuming that $(\Teul,\alpha)$
satisfies \eqref{9z8xiejhtgcdusajw4wajw3n4bc56d78kikjhfyds}.
Let $\beta$ be an arbitrary element of $(\Aeul \setminus \Aeul_0) \setminus \{\alpha\}$,
let $v_\beta$ denote the unique vertex adjacent to $\beta$
and let $h = q( \{v_\beta,\beta\} , v_\beta )$.
Let $\Teul'$ be the tree obtained from $\Teul$ by replacing
\raisebox{-2mm}{\begin{picture}(23,6)(-3,-3.5)
\put(0,0){\circle{1}}
\put(0.5,0){\vector(1,0){14}}
\put(3,.5){\makebox(0,0)[b]{\tiny $h$}}
\put(0,-1){\makebox(0,0)[t]{\tiny $v_\beta$}}
\put(13.5,-1){\makebox(0,0)[t]{\tiny $\beta$}}
\put(15,0){\makebox(0,0)[l]{\tiny $(1)$}}
\end{picture}}
by
\raisebox{-2mm}{\begin{picture}(37,6)(-3,-3.5)
\put(0,0){\circle{1}}
\put(.5,0){\line(1,0){14}}
\put(15,0){\circle{1}}
\put(15.5,0){\vector(1,0){14}}
\put(3,.5){\makebox(0,0)[b]{\tiny $h$}}
\put(12,.5){\makebox(0,0)[b]{\tiny $1$}}
\put(18,.5){\makebox(0,0)[b]{\tiny $1$}}
\put(0,-1){\makebox(0,0)[t]{\tiny $v_\beta$}}
\put(15,-1){\makebox(0,0)[t]{\tiny $w_\beta$}}
\put(28.5,-1){\makebox(0,0)[t]{\tiny $\beta$}}
\put(30,0){\makebox(0,0)[l]{\tiny $(1)$}}
\end{picture}} (so we are adding a vertex $w_\beta$ in the middle of the edge $\{v_\beta,\beta\}$).
Observe that $\Teul' \in \DTprp$, $F(\Teul') = F(\Teul) = n$,
and $(\Teul',\alpha)$ satisfies \eqref{9z8xiejhtgcdusajw4wajw3n4bc56d78kikjhfyds};
so, by the preceding part of the proof, we have
\begin{equation}  \label {90s9suidg2s3drcflkErrkkiy99bxvqwf9}
M\big( (\Teul')_\alpha \big) - M_\alpha(\Teul') = I'(\alpha)
\end{equation}
where $I'(\alpha)$ means $I(\alpha)$ computed in $\Teul'$.
So, to prove \eqref{56c5v6b2m239sf8gdfvlkka2gcjg5r65u}, it suffices to show that 
\begin{equation}  \label {w9e878126trwtsg9183h3dv652fn39}
M\big( (\Teul')_\alpha \big) = M( \Teul_\alpha ),
\quad M_\alpha(\Teul') = M_\alpha(\Teul)
\quad \text{and} \quad I'(\alpha) = I(\alpha).
\end{equation}
We have $(\Teul')_\alpha = \Teul_\alpha$, so $M\big( (\Teul')_\alpha \big) = M(\Teul_\alpha)$.
It is clear that $I'(\alpha) = I(\alpha)$, so it suffices to verify that $M_\alpha(\Teul') = M_\alpha(\Teul)$. 
For each $s \in \Veul \cup \Aeul_0$ (in $\Teul$) we have $x_{s,\alpha} = x'_{s,\alpha}$ and $\delta_s = \delta'_s$, so 
\begin{align*}
M_\alpha(\Teul')
&= \textstyle -x'_{w_\beta,\alpha} (\delta'_{w_\beta} - 2) -\sum_{s \in \Veul \cup \Aeul_0}  x'_{s,\alpha} (\delta'_s - 2) \\
&= \textstyle -x'_{w_\beta,\alpha} (2 - 2) -\sum_{s \in \Veul \cup \Aeul_0}  x_{s,\alpha} (\delta_s - 2)
=M_\alpha(\Teul) .
\end{align*}
So \eqref{w9e878126trwtsg9183h3dv652fn39} is true, and the proof of \eqref{56c5v6b2m239sf8gdfvlkka2gcjg5r65u} is complete. 
This proves the Lemma.
\end{proof}

\begin{proposition}  \label {2wsxcdr45tfvbhy7ujnmEli7tfvbnmki90}
Let $\Teul \in \DTpr$ be such that every element of $\Aeul \setminus \Aeul_0$ is decorated by $(1)$
and let $\Pscr$ be a partition of $\Aeul \setminus \Aeul_0$.

\smallskip

\begin{enumerata}
\setlength{\itemsep}{1mm}

\item $\displaystyle M(\Teul)
\ = \  \sum\undersum{ \alpha \in \Aeul \setminus \Aeul_0} M(\Teul_\alpha) \, - \, I\big( \Aeul \setminus \Aeul_0 , \Aeul \setminus \Aeul_0 \big)$

\item $\displaystyle M(\Teul) = \sum_{X \in \Pscr} M( \Teul_X )
\ - \sum\undersum{ \begin{smallmatrix} (X,Y) \in \Pscr^2 \\ X \neq Y \end{smallmatrix} } I(X,Y)$

\end{enumerata}
\end{proposition}

\begin{proof}
(a) By Lemma \ref{0d9f09gf909s990q9e0j3m3mn2mnm2moeiu83cdwx4}, we have
$M_\alpha(\Teul) = M(\Teul_\alpha) - I(\alpha)$ for every $\alpha \in \Aeul \setminus \Aeul_0$. Thus,
$M(\Teul)
= \sum_{\alpha \in \Aeul \setminus \Aeul_0} M_\alpha(\Teul)
= \sum_{\alpha \in \Aeul \setminus \Aeul_0} M(\Teul_\alpha) - \sum_{\alpha \in \Aeul \setminus \Aeul_0} I(\alpha)$.
Since $\sum_{\alpha \in \Aeul \setminus \Aeul_0} I(\alpha)
= I\big( \Aeul \setminus \Aeul_0 , \Aeul \setminus \Aeul_0 \big)$, this proves (a).

(b) Let $X \in \Pscr$.
Applying part (a)  to $\Teul_X$ gives
$$
M(\Teul_X) = \textstyle \sum_{ \alpha \in X } \! M\big( (\Teul_X)_\alpha \big) \  - \ I_X( X , X ) ,
$$
where $I_X(X,X)$ means $I(X,X)$ computed in $\Teul_X$ (and $I(X,X)$ is computed in $\Teul$).
It is easy to see that $I_X(X,X) = I(X,X)$.
By Rem.\ \ref{09o239r8d5beg56vq28gxcv9w}, $(\Teul_X)_\alpha = \Teul_\alpha$ for all $\alpha \in X$.
So $M(\Teul_X) = \textstyle \sum_{ \alpha \in X } \! M(\Teul_\alpha) \  - \ I( X , X )$ for all $X \in \Pscr$. Thus,
\begin{align*}
\sum_{X \in \Pscr} M(\Teul_X)
&= \sum\undersum{ \alpha \in \Aeul \setminus \Aeul_0 } M(\Teul_\alpha) \  - \sum_{X \in \Pscr} I( X , X ) \\
&= M(\Teul) + I\big( \Aeul \setminus \Aeul_0 , \Aeul \setminus \Aeul_0 \big) - \sum_{X \in \Pscr} I( X , X ) 
\end{align*}
where the last equality follows by applying part (a) to $\Teul$. 
The desired conclusion now follows from  part (3) of  Rem.\ \ref{fgfvgertjuIPPWeCdEew8hju834hnf893krfcnfFVH6t8}.
\end{proof}

See Notation \ref{8927476hbgFfxfob93746dBBNE376rFuf82190uhsSQ} for the definition of the integers $\delta(\Teul)$ and $g(\Teul)$.

\begin{proposition}  \label {9837487239rpejf9d88}
Let $\Teul \in \DTpr$ be such that every element of $\Aeul \setminus \Aeul_0$ is decorated by $(1)$,
and let $\Pscr = \{ X_1, \dots, X_n \}$ be a partition of $\Aeul \setminus \Aeul_0$.

\smallskip

\begin{enumerata}
\setlength{\itemsep}{1mm}

\item $ \displaystyle \delta(\Teul)
= \sum_{ \alpha \in \Aeul \setminus \Aeul_0 } \! \delta(\Teul_\alpha) \  \  + \  {\textstyle \frac12} I( \Aeul \setminus \Aeul_0 , \Aeul \setminus \Aeul_0 )$

\item $\displaystyle \delta(\Teul) = \, \sum_{i=1}^n  \delta(\Teul_{X_i}) \, + \, \sum_{i < j} I(X_i,X_j) $

\item $ \displaystyle g(\Teul) -1 =
\sum\undersum{ \alpha \in \Aeul \setminus \Aeul_0 } (g(\Teul_\alpha)-1) \  \  + \  {\textstyle \frac12}
I( \Aeul \setminus \Aeul_0 , \Aeul \setminus \Aeul_0 )$

\item $\displaystyle g(\Teul) - 1 = \, \sum_{i=1}^n  (g(\Teul_{X_i})-1) \, + \, \sum_{i < j} I(X_i,X_j) $

\end{enumerata}
\end{proposition}

\begin{proof}
Assertion (a) (resp.\ (c)) is the special case of assertion (b) (resp.\ (d)) where each $X_i$ has exactly one element;
so let us only prove (b) and (d).
Since each element of $\Aeul \setminus \Aeul_0$ is decorated by $(1)$, we have
$F(\Teul) = | \Aeul \setminus \Aeul_0 | = \sum_{i=1}^n |X_i| = \sum_{i=1}^n F(\Teul_{X_i})$
and (by Prop.\ \ref{2wsxcdr45tfvbhy7ujnmEli7tfvbnmki90})
\begin{equation*}
M(\Teul) = \sum_{X \in \Pscr} M( \Teul_X )
\ - \sum\undersum{ \begin{smallmatrix} (X,Y) \in \Pscr^2 \\ X \neq Y \end{smallmatrix} } I(X,Y) 
\ = \ \sum_{i=1}^n M( \Teul_{X_i} ) \ - \ 2 \sum_{i < j} I(X_i,X_j) .
\end{equation*}
Thus,
\begin{align*}
2 \delta(\Teul) &= F(\Teul) - M(\Teul) 
= \sum_{i=1}^n F(\Teul_{X_i})  -  \sum_{i=1}^n M( \Teul_{X_i} ) \ + \ 2 \sum_{i < j} I(X_i,X_j) \\
&= 2 \sum_{i=1}^n \delta(\Teul_{X_i})  \ + \ 2 \sum_{i < j} I(X_i,X_j) , \\
2 g(\Teul) - 2 &=  - M(\Teul) - F(\Teul)  
= - \sum_{i=1}^n M( \Teul_{X_i} ) \ + \ 2 \sum_{i < j} I(X_i,X_j) -  \sum_{i=1}^n F(\Teul_{X_i})  \\
&=  \sum_{i=1}^n (2 g(\Teul_{X_i}) - 2) \ + \ 2 \sum_{i < j} I(X_i,X_j) .
\end{align*}
Dividing by $2$ gives (b) and (d), so we are done.
\end{proof}

We shall now consider what happens when, in Propositions \ref{2wsxcdr45tfvbhy7ujnmEli7tfvbnmki90} and \ref{9837487239rpejf9d88},
we drop the assumption that all elements of $\Aeul \setminus \Aeul_0$ are decorated by $(1)$.
We obtain results similar to the above, but where each formula contains a ``correction term''.
The correction term in the formula for $M(\Teul)$ (resp.\ $\delta(\Teul)$, $g(\Teul)$)
is denoted $\Mgoth(\Teul,X)$ (resp.\ $\Dgoth(\Teul,X)$, $\Ggoth(\Teul,X)$) and is defined in
Notation \ref{9S908uhggjgyDHjcB83yedgSUgh62ffBoks} (resp.\ \ref{87654xvbnqctfctdvuhdiu3o9d4efwdgRSDFWe439w8d}).

\begin{notation}  \label {9S908uhggjgyDHjcB83yedgSUgh62ffBoks}
Let $\Teul = (\Veul, \Aeul, \Eeul, f, q) \in \DT$.
For each subset $X$ of $\Aeul \setminus \Aeul_0$, define
$$
\Mgoth(\Teul,X)
= \sum_{ \alpha \in X } \textstyle \big( 1 - \frac1{f(\alpha)} \big) I\big( \{ \alpha \}, X^c \big)
$$
where $X^c = (\Aeul \setminus \Aeul_0) \setminus X$.
\end{notation}

\begin{remark} \label {o287348712hdbfwzaa2a1aqxpmcgmpcrtieuhj873sE}
Let $\Teul = (\Veul, \Aeul, \Eeul, f, q) \in \DT$.
\begin{enumerate}

\item If $X \subseteq \Aeul \setminus \Aeul_0$ is such that $f(\alpha)=1$ for all $\alpha \in X$, then $\Mgoth(\Teul,X)=0$.

\item If $X \subseteq \Aeul \setminus \Aeul_0$ then $\Mgoth(\Teul,X) \in \Integ$,
because $f(\alpha) \mid I\big( \{ \alpha \}, X^c \big)$ for all $\alpha \in X$.

\item For each $\alpha \in \Aeul \setminus \Aeul_0$,
$\Mgoth(\Teul, \{\alpha\} ) = \big( 1 - \frac1{f(\alpha)} \big) I(\alpha) = ( f(\alpha) - 1 ) p(\alpha,e_\alpha)$,
where $e_\alpha$ is the unique edge incident to $\alpha$.

\end{enumerate}
\end{remark}

\begin{theorem}  \label {sfcv073btFbhd8SwAQedfffGnUd21653OLlp743}
Let $\Teul \in \DTpr$ and let $\Pscr$ be a partition of $\Aeul \setminus \Aeul_0$.

\smallskip

\begin{enumerata}
\setlength{\itemsep}{1.5mm}

\item
$\displaystyle M(\Teul) = \sum\undersum{ \alpha \in \Aeul \setminus \Aeul_0 } M( \Teul_\alpha )
\,\,  + \sum\undersum{\alpha \in \Aeul \setminus \Aeul_0} \Mgoth(\Teul,\{\alpha\}) \ - \ I( \Aeul \setminus \Aeul_0 , \Aeul \setminus \Aeul_0 ) $

\item   \label {djbGF09347gfhbSRaDGh6TYHJ93ihfV5y2i20}
$\displaystyle M(\Teul) = \sum_{X \in \Pscr} M( \Teul_X ) + \sum_{X \in \Pscr} \Mgoth(\Teul,X)
\ - \sum\undersum{ \begin{smallmatrix} (X,Y) \in \Pscr^2 \\ X \neq Y \end{smallmatrix} } I(X,Y)$

\end{enumerata}
\end{theorem}

\begin{proof}
Observe that (a) is the special case of \eqref{djbGF09347gfhbSRaDGh6TYHJ93ihfV5y2i20}
where every element of $\Pscr$ has exactly one element.
Moreover, assertion (a) is equivalent to 
\begin{enumerata}
\item[(a$'$)]
$\displaystyle M(\Teul) = \sum\undersum{ \alpha \in \Aeul \setminus \Aeul_0 } M( \Teul_\alpha )
\,\,  + \sum\undersum{\alpha \in \Aeul \setminus \Aeul_0} {\textstyle \big( 1 - \frac1{f(\alpha)} \big) I(\alpha) }
\ - \ I( \Aeul \setminus \Aeul_0 , \Aeul \setminus \Aeul_0 ) $,
\end{enumerata}
because $\Mgoth(\Teul,\{\alpha\}) =  \textstyle \big( 1 - \frac1{f(\alpha)} \big) I(\alpha)$ for every $\alpha \in \Aeul \setminus \Aeul_0$.
Let us prove (a$'$).

Let $\bbX$ be the set of elements $\Teul$ of $\DTpr$ for which equality (a$'$) is true.
Given $\Teul = (\Veul,\Aeul,\Eeul,f,q,v_0) \in \DTpr$,
let $\langle \Teul \rangle$ denote the number of arrows $\alpha \in \Aeul \setminus \Aeul_0$
that satisfy $f(\alpha) \neq 1$ (so $\langle \Teul \rangle \in \Nat$).
It suffices to prove the following two claims:
\begin{gather}  \label {09f9091292j3v34x567y8gzzmjwi87}
\text{$\bbX$ contains all $\Teul \in \DTpr$ such that $\langle\Teul\rangle = 0$;} \\
\label {0990vcxzes2zqswfcsdiivgkle1e3gg5i12e}
\begin{minipage}[t]{.85\textwidth}
If $m>0$ and $\bbX$ contains all $\Teul \in \DTpr$ such that $\langle\Teul\rangle < m$
then $\bbX$ contains all $\Teul \in \DTpr$ such that $\langle\Teul\rangle = m$.
\end{minipage}
\end{gather}

If $\Teul \in \DTpr$ is such that $\langle\Teul\rangle = 0$
then $\big( 1 - \frac1{f(\alpha)} \big) I(\alpha)  =0$ for every $\alpha \in \Aeul \setminus \Aeul_0$,
so Proposition \ref{2wsxcdr45tfvbhy7ujnmEli7tfvbnmki90}(a) implies that equality (a$'$) holds for $\Teul$
and hence $\Teul \in \bbX$.
This proves \eqref{09f9091292j3v34x567y8gzzmjwi87}.

Let us prove \eqref{0990vcxzes2zqswfcsdiivgkle1e3gg5i12e}.
Let $m>0$ and suppose that $\bbX$ contains all $\Teul \in \DTpr$ such that $\langle\Teul\rangle < m$.
Let $\Teul = (\Veul,\Aeul,\Eeul,f,q,v_0) \in \DTpr$ be such that $\langle \Teul \rangle = m$.
We have to show that $\Teul \in \bbX$.
There exists $\alpha \in \Aeul \setminus \Aeul_0$ such that $f(\alpha) \neq 1$;
choose such an $\alpha$ and write $\gamma_{v_0,\alpha} = (v_0, v_1, \dots, v_n, \alpha)$ (where $n\ge0$).
Let $s = f(\alpha)$, let $e = \{v_n,\alpha\}$, let $p = p(\alpha,e)$ and recall that $F(\alpha) = \gcd(s,p)$;
observe that $F(\alpha) > 0$, because $s \neq 0$.
Let $\{\Teul_1,\Teul_2\} = \Spl(\Teul,e)$; then $\Teul_1, \Teul_2 \in \DT$ look like this:
\begin{equation} \label {ieirywu76cbeHgl3uy87wVh827t36273892}
\raisebox{-12\unitlength}{\begin{picture}(86,22)(-40,-17)

\put(-30,0){\circle{1}}
\put(-30,2){\makebox(0,0)[b]{\footnotesize $v_n$}}
\put(-29.5,0){\line(1,0){14}}
\put(-15,0){\circle*{1}}
\put(-14.5,0){\line(1,0){7.7}} \put(-7.1,0){\makebox(0,0)[l]{\footnotesize $<\! F(\alpha)$}}
\put(-15,-.5){\vector(0,-1){9.5}}
\put(-14,-9){\makebox(0,0)[l]{\footnotesize $(0)$}}
\put(-16,-9){\makebox(0,0)[r]{\tiny $\alpha_{1,0}$}}
\put(-18,-14){\makebox(0,0){$\displaystyle\underbrace{\rule{42\unitlength}{0mm}}_{\Teul_1}$}}

\put(-14,2){\makebox(0,0)[b]{\footnotesize $z_1$}}
\put(24,2){\makebox(0,0)[b]{\footnotesize $z_2$}}

\put(-17.5,1.5){\makebox(0,0)[dr]{\tiny $-a_2$}}
\put(-14.5,-3){\makebox(0,0)[l]{\tiny $a_1$}}

\put(27.5,1.5){\makebox(0,0)[dl]{\tiny $-a_1$}}
\put(25.5,-3){\makebox(0,0)[l]{\tiny $a_2$}}

\put(25.5,0){\vector(1,0){14.5}}
\put(41,0){\makebox(0,0)[l]{\footnotesize $(s)$}}
\put(39,1){\makebox(0,0)[b]{\footnotesize $\alpha$}}

\put(25,0){\circle*{1}}
\put(24.5,0){\line(-1,0){7.7}} \put(17.1,0){\makebox(0,0)[r]{\footnotesize $F(\alpha) >$}}
\put(25,-.5){\vector(0,-1){9.5}}
\put(26,-9){\makebox(0,0)[l]{\footnotesize $(0)$}}
\put(24,-9){\makebox(0,0)[r]{\tiny $\alpha_{2,0}$}}
\put(25,-14){\makebox(0,0){$\displaystyle\underbrace{\rule{36\unitlength}{0mm}}_{\Teul_2}$}}

\put(-30.35355,0.35355){\line(-1,1){4}}
\put(-30.35355,-0.35355){\line(-1,-1){4}}

\end{picture}} \quad \begin{array}{l}
a_1 = \frac{s}{F(\alpha)} \\ a_2 = \frac{p}{F(\alpha)}
\end{array}
\end{equation}
It is clear that  $\Teul_1 \in \DTpr$ where $\Teul_1$ has the same pseudo-root $v_0$ as $\Teul$.
Also note that $a_1 \neq 0$, because $s\neq0$.
It is easy to calculate that $M(\Teul_2) = 0$ and $F(\Teul_2) = 2F(\alpha)$; so Lemma \ref{dBfo23bn6i8404f8hfwA7ylfaroj} gives 
$M(\Teul_1) = M(\Teul)$ and $F(\Teul_1) = F(\Teul)$.
Since $\langle \Teul_1 \rangle = m-1$, the inductive hypothesis implies that $\Teul_1 \in \bbX$, so
\begin{equation}  \label {98fd8ufhhHsgfGnvh2836gcbxueiXDcfi81}
M(\Teul) = M(\Teul_1) =
\sum_{ \beta \in \bX } M( (\Teul_1)_\beta ) + 
\sum_{ \beta \in \bX } {\textstyle(1 - \frac1{f^{(1)}(\beta)}) I^{(1)}(\beta)} - I^{(1)}(\bX,\bX)
\end{equation}
where we use the notations $\Teul_1 = (\Veul^{(1)},\Aeul^{(1)},\Eeul^{(1)},f^{(1)},q^{(1)},v_0)$
and $\bX = \Aeul^{(1)} \setminus \Aeul_0^{(1)}$,
and where $I^{(1)}(\, \underline{\ \ }\, )$ means $I(\, \underline{\ \ }\, )$ computed in $\Teul_1$.
We claim:
\begin{equation}  \label {526egvloki6cOnbvcxcqwriwcbe8ry}
I^{(1)}(\bX,\bX) = I( \Aeul \setminus \Aeul_0 , \Aeul \setminus \Aeul_0 ) - a_1 a_2 F(\alpha) \big( F(\alpha) - 1 \big) .
\end{equation}
Indeed, define 
\begin{align*}
\bbA &= \setspec{ \beta \in \Aeul^{(1)} \setminus\Aeul^{(1)}_0 }{ \text{$\beta$ is not adjacent to $z_1$ in $\Teul_1$} }, \\
\bbB &= \setspec{ \beta \in \Aeul^{(1)} \setminus\Aeul^{(1)}_0 }{ \text{$\beta$ is adjacent to $z_1$ in $\Teul_1$} }.
\end{align*}
Then $\bX = \Aeul^{(1)} \setminus\Aeul^{(1)}_0 = \bbA \cup \bbB$,  $\bbA \cap \bbB = \emptyset$, 
$\Aeul \setminus\Aeul_0 = \bbA \cup \{ \alpha \}$ and  $\bbA \cap \{ \alpha \} = \emptyset$.
We have 
$ I( \Aeul \setminus \Aeul_0 , \Aeul \setminus \Aeul_0 )
= I( \bbA \cup \{\alpha\}, \bbA \cup \{\alpha\} ) 
= I(\bbA, \bbA) + 2 I(\{\alpha\},\bbA) 
= I(\bbA, \bbA) + 2 s p$,
so  $I^{(1)}(\bbA, \bbA) = I(\bbA, \bbA) = I( \Aeul \setminus \Aeul_0 , \Aeul \setminus \Aeul_0 ) - 2sp$ and hence
\begin{align*}
I^{(1)}(\bX,\bX) &=  I^{(1)}(\bbA, \bbA) + 2 I^{(1)}(\bbA, \bbB) + I^{(1)}(\bbB, \bbB) \\
&= I( \Aeul \setminus \Aeul_0 , \Aeul \setminus \Aeul_0 ) - 2sp + 2 I^{(1)}(\bbA, \bbB) + I^{(1)}(\bbB, \bbB) .
\end{align*}
For each $\beta \in \bbB$ we have $I^{(1)}(\bbA, \{ \beta \} ) = a_1 p$, so $I^{(1)}(\bbA, \bbB) = a_1 p F(\alpha) = sp$.
Since $I^{(1)}(\bbB, \bbB) = -a_1a_2F(\alpha)\big( F(\alpha) - 1 \big)$, we proved \eqref{526egvloki6cOnbvcxcqwriwcbe8ry}.

The fact that $f^{(1)}(\beta) = 1$ for all $\beta \in \bbB$ gives the first equality in
$$
\sum_{ \beta \in \bX } {\textstyle(1 - \frac1{f^{(1)}(\beta)}) I^{(1)}(\beta)} 
= \sum_{ \beta \in \bbA } {\textstyle(1 - \frac1{f^{(1)}(\beta)}) I^{(1)}(\beta)} 
= \sum_{ \beta \in \bbA } {\textstyle(1 - \frac1{f(\beta)}) I(\beta)} .
$$
For the second equality, the reader can check that  $I^{(1)}(\beta) = I(\beta)$  for all $\beta \in \bbA$ (and it is clear 
that $f^{(1)}(\beta) = f(\beta)$  for all $\beta \in \bbA$).  So,
\begin{equation}  \label {dp927c8agn184sjbvvcxy2w27gb}
\sum_{ \beta \in \bX } {\textstyle(1 - \frac1{f^{(1)}(\beta)}) I^{(1)}(\beta)} = 
- {\textstyle(1 - \frac1{f(\alpha)}) I(\alpha)} \ +
\sum\undersum{ \beta \in \Aeul \setminus \Aeul_0 } {\textstyle(1 - \frac1{f(\beta)}) I(\beta)} .
\end{equation}

Next, let us fix $\beta \in \bbB$ and compare $M( (\Teul_1)_\beta )$ to $M(\Teul_\alpha)$.
The trees $\Teul_\alpha$ and $(\Teul_1)_\beta$ are as follows:
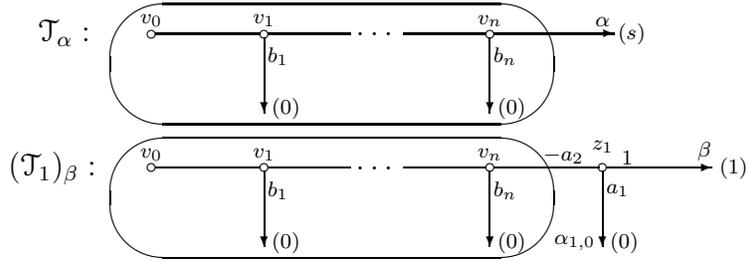
\begin{align*}
&{\begin{picture}(94,15)(-15,-12)
\put(-15,-1){$\Teul_\alpha :$}
\put(24,-4){\oval(59,16)}
\put(0,0){\circle{1}}
\put(15,0){\circle{1}}
\put(30,0){\makebox(0,0){$\dots$}}
\put(45,0){\circle{1}}
\put(.5,0){\line(1,0){14}}
\put(15.5,0){\line(1,0){11}}
\put(45.5,0){\vector(1,0){16}}
\put(44.5,0){\line(-1,0){11}}
\put(15,-.5){\vector(0,-1){10}}
\put(45,-.5){\vector(0,-1){10}}
\put(0,1){\makebox(0,0)[b]{\tiny $v_0$}}
\put(15,1){\makebox(0,0)[b]{\tiny $v_1$}}
\put(45,1){\makebox(0,0)[b]{\tiny $v_n$}}
\put(62,0){\makebox(0,0)[l]{\tiny $(s)$}}
\put(15.5,-3){\makebox(0,0)[l]{\tiny $b_1$}}
\put(45.5,-3){\makebox(0,0)[l]{\tiny $b_n$}}
\put(16,-10){\makebox(0,0)[l]{\tiny $(0)$}}
\put(46,-10){\makebox(0,0)[l]{\tiny $(0)$}}
\put(60,1){\makebox(0,0)[b]{\tiny $\alpha$}}
\end{picture}}
\\
&{\begin{picture}(94,15)(-15,-12)
\put(-19,-1){$(\Teul_1)_\beta :$}
\put(24,-4){\oval(59,16)}
\put(0,0){\circle{1}}
\put(15,0){\circle{1}}
\put(30,0){\makebox(0,0){$\dots$}}
\put(45,0){\circle{1}}
\put(60,0){\circle{1}}
\put(.5,0){\line(1,0){14}}
\put(15.5,0){\line(1,0){11}}
\put(45.5,0){\line(1,0){14}}
\put(44.5,0){\line(-1,0){11}}
\put(60.5,0){\vector(1,0){14}}
\put(15,-.5){\vector(0,-1){10}}
\put(45,-.5){\vector(0,-1){10}}
\put(60,-.5){\vector(0,-1){10}}
\put(0,1){\makebox(0,0)[b]{\tiny $v_0$}}
\put(15,1){\makebox(0,0)[b]{\tiny $v_1$}}
\put(45,1){\makebox(0,0)[b]{\tiny $v_n$}}
\put(73.5,1){\makebox(0,0)[b]{\tiny $\beta$}}
\put(75.5,0){\makebox(0,0)[l]{\tiny $(1)$}}
\put(15.5,-3){\makebox(0,0)[l]{\tiny $b_1$}}
\put(45.5,-3){\makebox(0,0)[l]{\tiny $b_n$}}
\put(16,-10){\makebox(0,0)[l]{\tiny $(0)$}}
\put(46,-10){\makebox(0,0)[l]{\tiny $(0)$}}
\put(60,2){\makebox(0,0)[b]{\tiny $z_1$}}
\put(62.5,.5){\makebox(0,0)[bl]{\tiny $1$}}
\put(57.5,.5){\makebox(0,0)[br]{\tiny $-a_2$}}
\put(60.5,-3){\makebox(0,0)[l]{\tiny $a_1$}}
\put(61,-10){\makebox(0,0)[l]{\tiny $(0)$}}
\put(59,-10){\makebox(0,0)[r]{\tiny $\alpha_{1,0}$}}
\end{picture}}
\end{align*}
where the content of the ovals is identical in the two trees,
and where the dead end incident to $v_i$ is present if and only if $b_i\neq1$ (i.e., if $b_i=1$ then $\delta_{v_i}=2$).
We write $N^{(\beta)}_{x}$ and $\delta^{(\beta)}_x$ (resp.\ $N^{(\alpha)}_{x}$ and $\delta^{(\alpha)}_x$) to denote quantities 
computed in $(\Teul_1)_\beta$ (resp.\ in $\Teul_\alpha$).
Let $\Oeul$ be the set of vertices and arrows in the ovals.
For each $x \in \Oeul$,
we have $\delta^{(\beta)}_x = \delta^{(\alpha)}_x$ and
$N^{(\beta)}_x = \frac{a_1}{s} N_x^{(\alpha)} = \frac{1}{F(\alpha)} N_x^{(\alpha)}$, so
\begin{equation*}
\textstyle
M( (\Teul_1)_\beta ) = - \sum_{ x \in \Oeul } N^{(\beta)}_x(\delta^{(\beta)}_x-2) - N^{(\beta)}_{z_1} + N^{(\beta)}_{\alpha_{1,0}}
= \frac{1}{F(\alpha)} M( \Teul_\alpha ) -  N^{(\beta)}_{z_1} + N^{(\beta)}_{\alpha_{1,0}} .
\end{equation*}
Since $N^{(\beta)}_{\alpha_{1,0}} = -a_2$ and $N^{(\beta)}_{z_1} = -a_1a_2$,
we have $M( (\Teul_1)_\beta ) = \frac{1}{F(\alpha)} M( \Teul_\alpha ) + a_1 a_2 - a_2$ for each $\beta \in \bbB$, so
$$
\textstyle
\sum_{\beta \in \bbB} M( (\Teul_1)_\beta ) =  M( \Teul_\alpha ) + F(\alpha) (a_1 a_2 - a_2) .
$$
Since $(\Teul_1)_\beta = \Teul_\beta$ for all $\beta \in \bbA$, it follows that 
\begin{equation}  \label {76d2f9kncbYTYqcb9q2307rtg40847}
\sum_{\beta \in \bX} M( (\Teul_1)_\beta )
=  \sum\undersum{ \beta \in \Aeul \setminus \Aeul_0 } M( \Teul_\beta ) \  + \  F(\alpha) (a_1 a_2 - a_2) .
\end{equation}

Putting together \eqref{98fd8ufhhHsgfGnvh2836gcbxueiXDcfi81}, \eqref{526egvloki6cOnbvcxcqwriwcbe8ry},
\eqref{dp927c8agn184sjbvvcxy2w27gb} and \eqref{76d2f9kncbYTYqcb9q2307rtg40847} gives
$$
M(\Teul) = \sum\undersum{ \beta \in \Aeul \setminus \Aeul_0 } M( \Teul_\beta )
\,\,  + \sum\undersum{\beta \in \Aeul \setminus \Aeul_0} {\textstyle \big( 1 - \frac1{f(\beta)} \big) I(\beta) }
\ - \ I( \Aeul \setminus \Aeul_0 , \Aeul \setminus \Aeul_0 ) \ + \  r ,
$$
where
\begin{align*}
r &= F(\alpha) (a_1 a_2 - a_2)  +  a_1 a_2 F(\alpha) ( F(\alpha) - 1 )  - {\textstyle \big( 1 - \frac1{ f(\alpha) } \big) }  I(\alpha) \\
&= F(\alpha) (- a_2)  +  a_1 a_2 F(\alpha)^2 - {\textstyle \big( 1 - \frac1{ s } \big) }  sp 
= -p + sp  - {\textstyle \big( 1 - \frac1{ s } \big) }  sp = 0 .
\end{align*}
So $\Teul \in \bbX$, i.e., we proved \eqref{0990vcxzes2zqswfcsdiivgkle1e3gg5i12e}. So (a$'$) is proved, and so is (a).

\medskip

(b) Consider $\Teul = (\Veul, \Aeul, \Eeul, f, q, v_0) \in \DTpr$ and a partition $\Pscr$ of $\Aeul \setminus \Aeul_0$.
Let $X \in \Pscr$.  We have $\Teul_X = (\Veul_X, \Aeul_X, \Eeul_X, f_X, q_X, v_0) \in \DTpr$,
so (a$'$) applied to $\Teul_X$ gives
$$
M(\Teul_X) = \sum_{ \alpha \in X } M( (\Teul_X)_\alpha )
+ \sum_{\alpha \in X} {\textstyle \big( 1 - \frac1{f_X(\alpha)} \big) I_X(\alpha) } \ - \ I_X(X,X) 
$$
where $I_X(\alpha)$ and $I_X(X,X)$ mean $I(\alpha)$ and $I(X,X)$ computed in $\Teul_X$.
We have $f_X(\alpha) = f(\alpha)$ and (by Rem.\ \ref{09o239r8d5beg56vq28gxcv9w}) $(\Teul_X)_\alpha = \Teul_\alpha$,
and it is easy to see that $I_X(\alpha) = I( \{\alpha\}, X )$ and $I_X(X,X)=I(X,X)$. Thus,
$$
M(\Teul_X) = \sum_{ \alpha \in X } M( \Teul_\alpha )
+ \sum_{\alpha \in X} {\textstyle \big( 1 - \frac1{f(\alpha)} \big) I( \{\alpha\}, X ) } \ - \ I(X,X) 
$$
for all $X \in \Pscr$.
This together with (a$'$) (applied to $\Teul$) shows that 
$M(\Teul) - \sum_{X \in \Pscr} M(\Teul_X) = A + B + C$ where
\begin{gather*}
A = \sum\undersum{ \alpha \in \Aeul \setminus \Aeul_0 } M( \Teul_\alpha ) - \sum_{X \in \Pscr} \sum_{ \alpha \in X } M( \Teul_\alpha ) = 0 , \\
B = \sum\undersum{\alpha \in \Aeul \setminus \Aeul_0} {\textstyle \big( 1 - \frac1{f(\alpha)} \big) I(\alpha) }
-  \sum_{X \in \Pscr} \sum_{\alpha \in X} {\textstyle \big( 1 - \frac1{f(\alpha)} \big) I( \{\alpha\}, X ) }, \\
C =  - \ I( \Aeul \setminus \Aeul_0 , \Aeul \setminus \Aeul_0 )  +  \sum_{X \in \Pscr} I(X,X)
= - \sum\undersum{ \begin{smallmatrix} (X,Y) \in \Pscr^2 \\ X \neq Y \end{smallmatrix} } I(X,Y) , 
\end{gather*}
the last equality by Rem.\ \ref{fgfvgertjuIPPWeCdEew8hju834hnf893krfcnfFVH6t8}.
Note that if $\alpha \in X$ then  $I(\alpha) - I( \{\alpha\}, X )
= I( \{\alpha\}, \Aeul \setminus \Aeul_0 ) - I( \{\alpha\}, X ) 
= I( \{\alpha\}, X^c )$; thus,
\begin{align*}
B &= \sum_{X \in \Pscr} \sum_{\alpha \in X} {\textstyle \big( 1 - \frac1{f(\alpha)} \big) [ I(\alpha) - I( \{\alpha\}, X ) ] } \\
&= \sum_{X \in \Pscr} \sum_{\alpha \in X} {\textstyle \big( 1 - \frac1{f(\alpha)} \big) I( \{\alpha\}, X^c ) }
= \sum_{X \in \Pscr} \Mgoth(\Teul,X)
\end{align*}
and consequently equality \eqref{djbGF09347gfhbSRaDGh6TYHJ93ihfV5y2i20} holds for $\Teul$.
This proves the Theorem.
\end{proof}

\begin{notation}  \label {87654xvbnqctfctdvuhdiu3o9d4efwdgRSDFWe439w8d}
Let $\Teul = (\Veul, \Aeul, \Eeul, f, q, v_0) \in \DTpr$.
For each subset $X$ of $\Aeul \setminus \Aeul_0$, define
\begin{align*}
\Dgoth( \Teul, X )
&= \sum_{\alpha \in X} \frac{\big( 1 - \frac1{f(\alpha)} \big) I(\{\alpha\},X^c) - ( F(\alpha) - F_X( \alpha )) }{2} \\
&= {\textstyle\frac12} \Mgoth(\Teul,X) \, - \, {\textstyle\frac12} \! \sum_{\alpha \in X} ( F(\alpha) - F_X( \alpha )) \\[1mm]
\Ggoth( \Teul, X )
&= \sum_{\alpha \in X} \frac{\big( 1 - \frac1{f(\alpha)} \big) I(\{\alpha\},X^c) + ( F(\alpha) - F_X( \alpha )) }{2} \\
&= {\textstyle\frac12} \Mgoth(\Teul,X) \, + \,  {\textstyle\frac12} \! \sum_{\alpha \in X} ( F(\alpha) - F_X( \alpha )) 
\end{align*}
where  $I(\{\alpha\},X^c)$ and $F(\alpha)$ are computed in $\Teul$, and $F_X(\alpha)$ means $F(\alpha)$ computed in $\Teul_X$.
\end{notation}

\begin{remark}  \label {34eDwctppCldnBpsd9iAsdtazcs2f}
Let the assumptions be as in \ref{87654xvbnqctfctdvuhdiu3o9d4efwdgRSDFWe439w8d}.
\begin{enumerate}

\item If $X \subseteq \Aeul\setminus\Aeul_0$ is such that $f(\alpha)=1$ for all $\alpha \in X$, then 
$$
\Dgoth( \Teul, X ) = 0 = \Ggoth( \Teul, X ) .
$$

\item If $X \subseteq \Aeul\setminus\Aeul_0$ then $\Dgoth( \Teul, X ), \, \Ggoth( \Teul, X ) \in \Integ$.

\item If $\alpha \in \Aeul \setminus \Aeul_0$ then
\begin{align*}
\Dgoth( \Teul, \{\alpha\} ) &= \frac{( f(\alpha) - 1 ) p(\alpha,e_{\alpha}) - F( \alpha ) + | f(\alpha) | }{2} \\[1mm]
\Ggoth( \Teul, \{\alpha\} ) &= \frac{( f(\alpha) - 1 ) p(\alpha,e_{\alpha}) + F( \alpha ) - | f(\alpha) | }{2} \, . 
\end{align*}

\end{enumerate}
Here is the proof of remark (2).  First note that if $\alpha \in X$ then
$I(\{\alpha\},X^c) = I(\{\alpha\},(\Aeul\setminus\Aeul_0)\setminus\{\alpha\}) - I(\{\alpha\},X\setminus\{\alpha\})  = f(\alpha)[ p(\alpha,e_\alpha) - p_X(\alpha,e_\alpha) ]$
where $e_\alpha$ is the unique edge of $\Teul$ incident to $\alpha$,
$p(\alpha,e_\alpha)$ is computed in $\Teul$, and $p_X(\alpha,e_{\alpha})$ means $p(\alpha,e_\alpha)$ computed in $\Teul_X$.
Using the abbreviations $s = f(\alpha)$, $p = p(\alpha,e_\alpha)$ and $p_X=p_X(\alpha,e_{\alpha})$, we see that the
number $\big( 1 - \frac1{f(\alpha)} \big) I(\{\alpha\},X^c) - ( F(\alpha) - F_X( \alpha ))$ is equal to
$ (s-1)(p-p_X) - \gcd(s,p) + \gcd(s,p_X) $, which is even for every choice of $s,p,p_X \in \Integ$.
This shows that  $\Dgoth( \Teul, X ) \in \Integ$.  Since $\Dgoth( \Teul, X ) + \Ggoth( \Teul, X ) = \Mgoth(\Teul,X) \in \Integ$,
this also shows that $\Ggoth( \Teul, X ) \in \Integ$.
\end{remark}

\begin{theorem}  \label {90323673738bbcgcbcnpss899aaw1a}
Let $\Teul \in \DTpr$ and let $\Pscr = \{ X_1, \dots, X_n \}$ be a partition of $\Aeul \setminus \Aeul_0$.

\smallskip

\begin{enumerata}
\setlength{\itemsep}{1mm}

\item $\displaystyle \delta(\Teul) =
\sum\undersum{ \alpha \in \Aeul \setminus \Aeul_0 } \! \delta(\Teul_\alpha) \  - 
\sum\undersum{ \alpha \in \Aeul \setminus \Aeul_0 } \! \Dgoth(\Teul,\{\alpha\}) \  + \  
 {\textstyle\frac12} I(\Aeul \setminus \Aeul_0 , \Aeul \setminus \Aeul_0) $

\item $\displaystyle \delta(\Teul) = \sum_{i=1}^n \delta(\Teul_{X_i}) \  - \  \sum_{i=1}^n \Dgoth(\Teul,X_i) \  + \  \sum_{i<j} I(X_i,X_j)$

\item $\displaystyle g(\Teul) - 1 =
\sum\undersum{ \alpha \in \Aeul \setminus \Aeul_0 } \! (g(\Teul_\alpha)-1) \  - 
\sum\undersum{ \alpha \in \Aeul \setminus \Aeul_0 } \! \Ggoth(\Teul,\{\alpha\}) \  + \  
 {\textstyle\frac12} I(\Aeul \setminus \Aeul_0 , \Aeul \setminus \Aeul_0) $

\item $\displaystyle g(\Teul)-1 = \sum_{i=1}^n (g(\Teul_{X_i})-1) \  - \  \sum_{i=1}^n \Ggoth(\Teul,X_i) \  + \  \sum_{i<j} I(X_i,X_j)$

\end{enumerata}
\end{theorem}

\begin{proof}
Assertion (a) (resp.\ (c)) is the special case of assertion (b) (resp.\ (d)) where each $X_i$ has exactly one element;
so it suffices to prove (b) and (d).
We have 
$$
\displaystyle M(\Teul) = \sum_{i=1}^n M( \Teul_{X_i} ) + \sum_{i=1}^n \Mgoth(\Teul,X_i) \ - 2 \sum_{i<j} I(X_i,X_j)
$$
by Thm \ref{sfcv073btFbhd8SwAQedfffGnUd21653OLlp743}(b), which gives the second equality in:
\begin{align*}
2 \delta(\Teul) & = F(\Teul) - M(\Teul) = F(\Teul) - \sum_{i=1}^n M(\Teul_{X_i}) - \sum_{i=1}^n \Mgoth(\Teul,X_i)
+  2 \sum_{i<j} I(X_i,X_j) \\
& = F(\Teul) - \sum_{i=1}^n \big( F(\Teul_{X_i}) - 2 \delta(\Teul_{X_i}) \big) - \sum_{i=1}^n \Mgoth(\Teul,X_i) + 2 \sum_{i<j} I(X_i,X_j) .
\end{align*}
So we have $2 \delta(\Teul) =  2 \sum_{i=1}^n \delta(\Teul_{X_i}) - A + 2 \sum_{i<j} I(X_i,X_j)$, where
\begin{align*}
A &= \sum_{i=1}^n \Mgoth(\Teul,X_i) \ - \ F(\Teul) \ + \ \sum_{i=1}^n F(\Teul_{X_i}) \\
&= \sum_{i=1}^n \Mgoth(\Teul,X_i) \ - \ \sum_{i=1}^n \sum_{\alpha \in X_i} F(\alpha) \ + \ \sum_{i=1}^n \sum_{\alpha \in X_i} F_{X_i}(\alpha) \\
&= \sum_{i=1}^n \left( \Mgoth(\Teul,X_i) \ - \ \sum_{\alpha \in X_i} \big( F(\alpha)  -  F_{X_i}(\alpha) \big) \right)
= \sum_{i=1}^n 2 \Dgoth(\Teul,X_i) .
\end{align*}
We get $2 \delta(\Teul) = 2 \sum_{i=1}^n \delta(\Teul_{X_i}) -  2 \sum_{i=1}^n \Dgoth(\Teul,X_i) + 2 \sum_{i<j} I(X_i,X_j)$,
so (b) follows.

The proof of (d) is similar to that of (b), and is left to the reader.
\end{proof}

\begin{example}  \label {o873646728378uejdwe9eo}
(a) Consider the following element $\Teul$ of $\DTpr$:
\begin{equation*}
\Teul\,:\ \   
\raisebox{-12mm}{\scalebox{.85}{\begin{picture}(94,22)(-9,-16)

\put(30,1){\makebox(0,0)[b]{\tiny $v_0$}}

\put(.7,.5){\makebox(0,0)[bl]{\tiny $-13$}}
\put(15.7,.5){\makebox(0,0)[bl]{\tiny $-3$}}
\put(44,.5){\makebox(0,0)[br]{\tiny $-1$}}
\put(59,.5){\makebox(0,0)[br]{\tiny $-7$}}
\put(74,.5){\makebox(0,0)[br]{\tiny $-24$}}

\multiput(0,0)(15,0){6}{\circle{1}}
\multiput(0.5,0)(15,0){5}{\line(1,0){14}}
\put(15,-12){\circle{1}}
\put(15,-.5){\line(0,-1){11}}
\put(-.5,0){\vector(-1,0){8}}

\put(0,-.5){\vector(0,-1){8}}
\put(.5,-9){\makebox(0,0)[l]{\tiny $(0)$}}
\put(.5,-2){\makebox(0,0)[lt]{\tiny $2$}}

\put(45,-.5){\vector(0,-1){8}}
\put(45.5,-9){\makebox(0,0)[l]{\tiny $(0)$}}
\put(45.5,-2){\makebox(0,0)[lt]{\tiny $2$}}

\put(60,-.5){\vector(0,-1){8}}
\put(60.5,-9){\makebox(0,0)[l]{\tiny $(0)$}}
\put(59.5,-2){\makebox(0,0)[rt]{\tiny $3$}}

\put(75.4472,.2236){\vector(2,1){8}}
\put(75.4472,-.2236){\vector(2,-1){8}}

\put(14.64645,-.35355){\vector(-1,-1){6}}

\put(15.4472,-11.7764){\vector(2,1){8}}
\put(15.4472,-12.2236){\vector(2,-1){8}}

\put(15.5,-2){\makebox(0,0)[lt]{\tiny $2$}}
\put(14.5,-10){\makebox(0,0)[rb]{\tiny $-2$}}

\put(-7,1){\makebox(0,0)[b]{\tiny $\alpha_1$}}
\put(7,-6){\makebox(0,0)[b]{\tiny $\alpha_2$}}
\put(23,-7){\makebox(0,0)[b]{\tiny $\alpha_3$}}
\put(23,-15){\makebox(0,0)[b]{\tiny $\alpha_4$}}
\put(82.5,5){\makebox(0,0)[b]{\tiny $\alpha_5$}}
\put(82.5,-3){\makebox(0,0)[b]{\tiny $\alpha_6$}}

\end{picture}}}
\end{equation*}
Let $X_1 = \{ \alpha_2, \alpha_3, \alpha_4, \alpha_5 \}$ and $X_2 = \{ \alpha_1, \alpha_6 \}$.
Then $\{X_1,X_2\}$ is a partition of $\Aeul \setminus \Aeul_0$ and the trees $\Teul_{X_1}$ and $\Teul_{X_2}$ (see Def.\ \ref{092198c45w6evXgb4r8rhfShbg48}) are:
\begin{gather*}
\Teul_{X_1}\,:\ \ \raisebox{-12mm}{\scalebox{.85}{\begin{picture}(79,22)(-9,-16)
\put(30,1){\makebox(0,0)[b]{\tiny $v_0$}}
\put(15.7,.5){\makebox(0,0)[bl]{\tiny $-3$}}
\put(44,.5){\makebox(0,0)[br]{\tiny $-1$}}
\put(59,.5){\makebox(0,0)[br]{\tiny $-7$}}
\put(74,.5){\makebox(0,0)[br]{\tiny $-24$}}
\multiput(15,0)(15,0){5}{\circle{1}}
\multiput(15.5,0)(15,0){4}{\line(1,0){14}}
\put(15,-12){\circle{1}}
\put(15,-.5){\line(0,-1){11}}
\put(45,-.5){\vector(0,-1){8}}
\put(45.5,-9){\makebox(0,0)[l]{\tiny $(0)$}}
\put(45.5,-2){\makebox(0,0)[lt]{\tiny $2$}}
\put(60,-.5){\vector(0,-1){8}}
\put(60.5,-9){\makebox(0,0)[l]{\tiny $(0)$}}
\put(59.5,-2){\makebox(0,0)[rt]{\tiny $3$}}
\put(75.4472,.2236){\vector(2,1){8}}
\put(14.64645,-.35355){\vector(-1,-1){6}}
\put(15.4472,-11.7764){\vector(2,1){8}}
\put(15.4472,-12.2236){\vector(2,-1){8}}
\put(15.5,-2){\makebox(0,0)[lt]{\tiny $2$}}
\put(14.5,-10){\makebox(0,0)[rb]{\tiny $-2$}}
\put(7,-6){\makebox(0,0)[b]{\tiny $\alpha_2$}}
\put(23,-7){\makebox(0,0)[b]{\tiny $\alpha_3$}}
\put(23,-15){\makebox(0,0)[b]{\tiny $\alpha_4$}}
\put(82.5,5){\makebox(0,0)[b]{\tiny $\alpha_5$}}
\end{picture}}} \\
\Teul_{X_2}\,:\ \  
\raisebox{-12mm}{\scalebox{.85}{\begin{picture}(79,22)(-9,-16)
\put(30,1){\makebox(0,0)[b]{\tiny $v_0$}}
\put(.7,.5){\makebox(0,0)[bl]{\tiny $-13$}}
\put(15.7,.5){\makebox(0,0)[bl]{\tiny $-3$}}
\put(44,.5){\makebox(0,0)[br]{\tiny $-1$}}
\put(59,.5){\makebox(0,0)[br]{\tiny $-7$}}
\put(74,.5){\makebox(0,0)[br]{\tiny $-24$}}
\multiput(0,0)(15,0){6}{\circle{1}}
\multiput(0.5,0)(15,0){5}{\line(1,0){14}}
\put(15,-.5){\vector(0,-1){8}}
\put(15.5,-9){\makebox(0,0)[l]{\tiny $(0)$}}
\put(-.5,0){\vector(-1,0){8}}
\put(0,-.5){\vector(0,-1){8}}
\put(.5,-9){\makebox(0,0)[l]{\tiny $(0)$}}
\put(.5,-2){\makebox(0,0)[lt]{\tiny $2$}}
\put(45,-.5){\vector(0,-1){8}}
\put(45.5,-9){\makebox(0,0)[l]{\tiny $(0)$}}
\put(45.5,-2){\makebox(0,0)[lt]{\tiny $2$}}
\put(60,-.5){\vector(0,-1){8}}
\put(60.5,-9){\makebox(0,0)[l]{\tiny $(0)$}}
\put(59.5,-2){\makebox(0,0)[rt]{\tiny $3$}}
\put(75.4472,-.2236){\vector(2,-1){8}}
\put(15.5,-2){\makebox(0,0)[lt]{\tiny $2$}}
\put(-7,1){\makebox(0,0)[b]{\tiny $\alpha_1$}}
\put(82.5,-3){\makebox(0,0)[b]{\tiny $\alpha_6$}}
\end{picture}}}
\end{gather*}
Prop.\ \ref{9837487239rpejf9d88}(d) asserts that 
\begin{equation}  \label {9876gfuhf2gsnxwderbfrghjinjfh}
g(\Teul)-1 = ( g(\Teul_{X_1}) - 1 ) + ( g(\Teul_{X_2}) - 1 ) + I(X_1,X_2) .
\end{equation}
Let us verify that this is the case.
The following numbers are computed directly from the above pictures:
$$
\begin{array}{ccc}
M(\Teul) = -4, & M(\Teul_{X_1}) = -2, & M(\Teul_{X_2}) = -2 , \\
F(\Teul) = 6, & F(\Teul_{X_1}) = 4, &  F(\Teul_{X_2}) = 2, \\
& \makebox[0mm]{$I(X_1,X_2) = 0$.}
\end{array}
$$
We have $g(\Teul) = \frac12 ( 2 - M(\Teul) - F(\Teul) ) = 0$, and similarly $g(\Teul_{X_1}) = 0$ and $g(\Teul_{X_2}) = 1$. 
So \eqref{9876gfuhf2gsnxwderbfrghjinjfh} is verified in this example.

(b) Consider the following element $\Teul$ of $\DTpr$
(the only difference with part (a) is that the decorations of $\alpha_1$ and $\alpha_6$ are doubled):
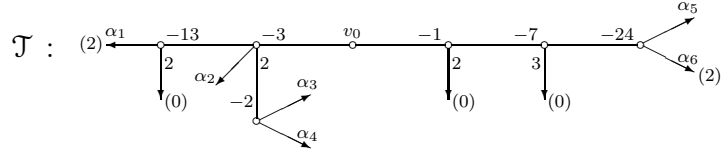
\begin{equation*}
\Teul\,:\ \   
\raisebox{-12mm}{\scalebox{.85}{\begin{picture}(97,22)(-12,-16)

\put(30,1){\makebox(0,0)[b]{\tiny $v_0$}}

\put(.7,.5){\makebox(0,0)[bl]{\tiny $-13$}}
\put(15.7,.5){\makebox(0,0)[bl]{\tiny $-3$}}
\put(44,.5){\makebox(0,0)[br]{\tiny $-1$}}
\put(59,.5){\makebox(0,0)[br]{\tiny $-7$}}
\put(74,.5){\makebox(0,0)[br]{\tiny $-24$}}

\multiput(0,0)(15,0){6}{\circle{1}}
\multiput(0.5,0)(15,0){5}{\line(1,0){14}}
\put(15,-12){\circle{1}}
\put(15,-.5){\line(0,-1){11}}
\put(-.5,0){\vector(-1,0){8}}

\put(0,-.5){\vector(0,-1){8}}
\put(.5,-9){\makebox(0,0)[l]{\tiny $(0)$}}
\put(.5,-2){\makebox(0,0)[lt]{\tiny $2$}}

\put(45,-.5){\vector(0,-1){8}}
\put(45.5,-9){\makebox(0,0)[l]{\tiny $(0)$}}
\put(45.5,-2){\makebox(0,0)[lt]{\tiny $2$}}

\put(60,-.5){\vector(0,-1){8}}
\put(60.5,-9){\makebox(0,0)[l]{\tiny $(0)$}}
\put(59.5,-2){\makebox(0,0)[rt]{\tiny $3$}}

\put(75.4472,.2236){\vector(2,1){8}}
\put(75.4472,-.2236){\vector(2,-1){8}}

\put(14.64645,-.35355){\vector(-1,-1){6}}

\put(15.4472,-11.7764){\vector(2,1){8}}
\put(15.4472,-12.2236){\vector(2,-1){8}}

\put(15.5,-2){\makebox(0,0)[lt]{\tiny $2$}}
\put(14.5,-10){\makebox(0,0)[rb]{\tiny $-2$}}

\put(-7,1){\makebox(0,0)[b]{\tiny $\alpha_1$}}
\put(7,-6){\makebox(0,0)[b]{\tiny $\alpha_2$}}
\put(23,-7){\makebox(0,0)[b]{\tiny $\alpha_3$}}
\put(23,-15){\makebox(0,0)[b]{\tiny $\alpha_4$}}
\put(82.5,5){\makebox(0,0)[b]{\tiny $\alpha_5$}}
\put(82.5,-3){\makebox(0,0)[b]{\tiny $\alpha_6$}}

\put(-9,0){\makebox(0,0)[r]{\tiny $(2)$}}
\put(84,-5){\makebox(0,0)[l]{\tiny $(2)$}}

\end{picture}}}
\end{equation*}
As in part (a), we take $X_1 = \{ \alpha_2, \alpha_3, \alpha_4, \alpha_5 \}$ and $X_2 = \{ \alpha_1, \alpha_6 \}$. 
Then $\Teul_{X_1}$ is identical to the $\Teul_{X_1}$ of part (a), and $\Teul_{X_2}$ is as follows:
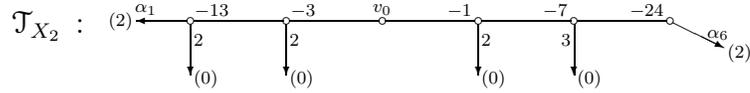
\begin{equation*}
\Teul_{X_2}\,:\ \  
\raisebox{-12mm}{\scalebox{.85}{\begin{picture}(82,22)(-12,-16)
\put(30,1){\makebox(0,0)[b]{\tiny $v_0$}}
\put(.7,.5){\makebox(0,0)[bl]{\tiny $-13$}}
\put(15.7,.5){\makebox(0,0)[bl]{\tiny $-3$}}
\put(44,.5){\makebox(0,0)[br]{\tiny $-1$}}
\put(59,.5){\makebox(0,0)[br]{\tiny $-7$}}
\put(74,.5){\makebox(0,0)[br]{\tiny $-24$}}
\multiput(0,0)(15,0){6}{\circle{1}}
\multiput(0.5,0)(15,0){5}{\line(1,0){14}}
\put(15,-.5){\vector(0,-1){8}}
\put(15.5,-9){\makebox(0,0)[l]{\tiny $(0)$}}
\put(-.5,0){\vector(-1,0){8}}
\put(0,-.5){\vector(0,-1){8}}
\put(.5,-9){\makebox(0,0)[l]{\tiny $(0)$}}
\put(.5,-2){\makebox(0,0)[lt]{\tiny $2$}}
\put(45,-.5){\vector(0,-1){8}}
\put(45.5,-9){\makebox(0,0)[l]{\tiny $(0)$}}
\put(45.5,-2){\makebox(0,0)[lt]{\tiny $2$}}
\put(60,-.5){\vector(0,-1){8}}
\put(60.5,-9){\makebox(0,0)[l]{\tiny $(0)$}}
\put(59.5,-2){\makebox(0,0)[rt]{\tiny $3$}}
\put(75.4472,-.2236){\vector(2,-1){8}}
\put(15.5,-2){\makebox(0,0)[lt]{\tiny $2$}}
\put(-7,1){\makebox(0,0)[b]{\tiny $\alpha_1$}}
\put(82.5,-3){\makebox(0,0)[b]{\tiny $\alpha_6$}}
\put(-9,0){\makebox(0,0)[r]{\tiny $(2)$}}
\put(84,-5){\makebox(0,0)[l]{\tiny $(2)$}}
\end{picture}}}
\end{equation*}
Thm \ref{90323673738bbcgcbcnpss899aaw1a}(d) asserts that 
\begin{equation}  \label {b--9876gfuhf2gsnxwderbfrghjinjfh}
\displaystyle g(\Teul)-1 = \sum_{i=1}^2 (g(\Teul_{X_i})-1) \  - \  \sum_{i=1}^2 \Ggoth(\Teul,X_i) \  + \  I(X_1,X_2) .
\end{equation}
We verify that this is the case.
Direct calculation from the pictures of $\Teul,\Teul_{X_1},\Teul_{X_2}$ gives:
$$
\begin{array}{ccc}
M(\Teul) = -6, & M(\Teul_{X_1}) = -2, & M(\Teul_{X_2}) = -4, , \\
F(\Teul) = 8, & F(\Teul_{X_1}) = 4, &  F(\Teul_{X_2}) = 4, \\
I(X_1,X_2) = 0, & \Ggoth(\Teul,X_1)=0, & \Ggoth(\Teul,X_2) = 0.
\end{array}
$$
(Note that $\Ggoth(\Teul,X_1)=0$ also follows from Rem.\ \ref{34eDwctppCldnBpsd9iAsdtazcs2f}(1).)
We have $g(\Teul) = \frac12 ( 2 - M(\Teul) - F(\Teul) ) = 0$, and similarly $g(\Teul_{X_1}) = 0$ and $g(\Teul_{X_2}) = 1$. 
So \eqref{b--9876gfuhf2gsnxwderbfrghjinjfh} is verified in this example.
\end{example}

\begin{example}  \label {87652ewtg2s5tsx2qadfvg2xdj3eoo009pokl}
Consider the following element $\Teul$ of $\DTpr$:
\begin{equation*}
\Teul\,:\ \ 
\raisebox{-17mm}{\scalebox{.85}{\begin{picture}(40,38.5)(-12,-6.5)

\put(0,-4.5){\line(0,1){9}}
\put(0,5.5){\line(0,1){9}}
\put(0,15.5){\line(0,1){14}}

\put(-.5,28){\makebox(0,0)[rt]{\tiny $5$}}
\put(-.5,17){\makebox(0,0)[rb]{\tiny $2$}}
\put(-.5,13){\makebox(0,0)[rt]{\tiny $3$}}
\put(-.5,-3){\makebox(0,0)[rb]{\tiny $3$}}
\put(-1,5){\makebox(0,0)[r]{\tiny $v_0$}}

\put(0,-5){\circle{1}}
\put(0,5){\circle{1}}
\put(0,15){\circle{1}}
\put(0,30){\circle{1}}

\put(-.5,-5){\vector(-1,0){8}}
\put(-9,-5){\makebox(0,0)[r]{\tiny $(0)$}}
\put(-2,-5.5){\makebox(0,0)[rt]{\tiny $2$}}

\put(-.5,30){\vector(-1,0){8}}
\put(-9,30){\makebox(0,0)[r]{\tiny $(0)$}}
\put(-2,30.5){\makebox(0,0)[rb]{\tiny $2$}}

\put(.5,-5){\vector(1,0){8}}
\put(9,-5){\makebox(0,0)[l]{\tiny $(3)$}}

\put(.5,30){\vector(1,0){8}}

\put(0.4472,15.2236){\line(2,1){15}}
\put(0.4472,14.7764){\line(2,-1){15}}
\put(16,23){\circle{1}}
\put(16,7){\circle{1}}

\put(16,22.5){\vector(0,-1){8}}
\put(16.5,14){\makebox(0,0)[l]{\tiny $(0)$}}
\put(16.5,21){\makebox(0,0)[lt]{\tiny $2$}}
\put(16.5,23){\vector(1,0){8}}
\put(25,23){\makebox(0,0)[l]{\tiny $(2)$}}

\put(16,6.5){\vector(0,-1){8}}
\put(16.5,-2){\makebox(0,0)[l]{\tiny $(0)$}}
\put(16.5,5){\makebox(0,0)[lt]{\tiny $2$}}
\put(16.5,7){\vector(1,0){8}}

\put(13,22.5){\makebox(0,0)[dr]{\tiny $13$}}
\put(13,8){\makebox(0,0)[tr]{\tiny $15$}}

\put(7,31){\makebox(0,0)[b]{\tiny $\alpha_1$}}
\put(23,24){\makebox(0,0)[b]{\tiny $\alpha_2$}}
\put(23,8){\makebox(0,0)[b]{\tiny $\alpha_3$}}
\put(7,-4){\makebox(0,0)[b]{\tiny $\alpha_4$}}

\end{picture}}}
\end{equation*}
Let $X_1 = \{\alpha_1, \alpha_2\}$ and $X_2 = \{\alpha_3, \alpha_4\}$.
Then $\{ X_1, X_2\}$ is a partition of $\Aeul \setminus \Aeul_0$ and $\Teul_{X_1}, \Teul_{X_2}$ are as follows:
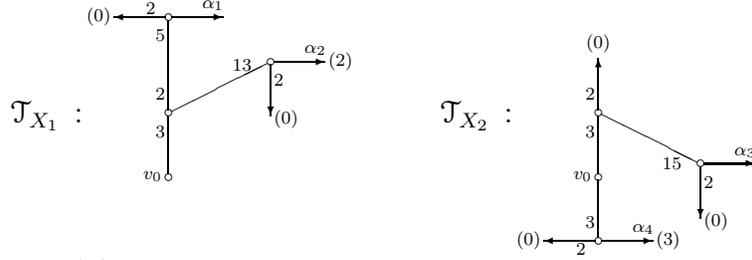
\begin{equation*}
\Teul_{X_1}\,:
\raisebox{-17mm}{\scalebox{.85}{\begin{picture}(40,38.5)(-12,-6.5)

\put(0,5.5){\line(0,1){9}}
\put(0,15.5){\line(0,1){14}}

\put(-.5,28){\makebox(0,0)[rt]{\tiny $5$}}
\put(-.5,17){\makebox(0,0)[rb]{\tiny $2$}}
\put(-.5,13){\makebox(0,0)[rt]{\tiny $3$}}
\put(-1,5){\makebox(0,0)[r]{\tiny $v_0$}}

\put(0,5){\circle{1}}
\put(0,15){\circle{1}}
\put(0,30){\circle{1}}

\put(-.5,30){\vector(-1,0){8}}
\put(-9,30){\makebox(0,0)[r]{\tiny $(0)$}}
\put(-2,30.5){\makebox(0,0)[rb]{\tiny $2$}}

\put(.5,30){\vector(1,0){8}}

\put(0.4472,15.2236){\line(2,1){15}}
\put(16,23){\circle{1}}

\put(16,22.5){\vector(0,-1){8}}
\put(16.5,14){\makebox(0,0)[l]{\tiny $(0)$}}
\put(16.5,21){\makebox(0,0)[lt]{\tiny $2$}}
\put(16.5,23){\vector(1,0){8}}
\put(25,23){\makebox(0,0)[l]{\tiny $(2)$}}

\put(13,22.5){\makebox(0,0)[dr]{\tiny $13$}}

\put(7,31){\makebox(0,0)[b]{\tiny $\alpha_1$}}
\put(23,24){\makebox(0,0)[b]{\tiny $\alpha_2$}}
\end{picture}}}
\qquad\quad
\Teul_{X_2}\,:
\raisebox{-17mm}{\scalebox{.85}{\begin{picture}(40,38.5)(-12,-6.5)

\put(0,-4.5){\line(0,1){9}}
\put(0,5.5){\line(0,1){9}}
\put(0,15.5){\vector(0,1){8}}

\put(0,24.5){\makebox(0,0)[b]{\tiny $(0)$}}

\put(-.5,17){\makebox(0,0)[rb]{\tiny $2$}}
\put(-.5,13){\makebox(0,0)[rt]{\tiny $3$}}
\put(-.5,-3){\makebox(0,0)[rb]{\tiny $3$}}
\put(-1,5){\makebox(0,0)[r]{\tiny $v_0$}}

\put(0,-5){\circle{1}}
\put(0,5){\circle{1}}
\put(0,15){\circle{1}}

\put(-.5,-5){\vector(-1,0){8}}
\put(-9,-5){\makebox(0,0)[r]{\tiny $(0)$}}
\put(-2,-5.5){\makebox(0,0)[rt]{\tiny $2$}}

\put(.5,-5){\vector(1,0){8}}
\put(9,-5){\makebox(0,0)[l]{\tiny $(3)$}}

\put(0.4472,14.7764){\line(2,-1){15}}
\put(16,7){\circle{1}}

\put(16,6.5){\vector(0,-1){8}}
\put(16.5,-2){\makebox(0,0)[l]{\tiny $(0)$}}
\put(16.5,5){\makebox(0,0)[lt]{\tiny $2$}}
\put(16.5,7){\vector(1,0){8}}

\put(13,8){\makebox(0,0)[tr]{\tiny $15$}}

\put(23,8){\makebox(0,0)[b]{\tiny $\alpha_3$}}
\put(7,-4){\makebox(0,0)[b]{\tiny $\alpha_4$}}
\end{picture}}}
\end{equation*}
According to Thm \ref{90323673738bbcgcbcnpss899aaw1a}(b),
\begin{equation}  \label {837yigv3r81uy1202ehe}
\displaystyle \delta(\Teul) = \sum_{i=1}^2 \delta(\Teul_{X_i}) \  - \  \sum_{i=1}^2 \Dgoth(\Teul,X_i) \  + \  I(X_1,X_2) .
\end{equation}
We verify that this is the case.
Direct calculation from the above pictures gives:
$$
\begin{array}{ccc}
M(\Teul) = -273, & M(\Teul_{X_1}) = -69, & M(\Teul_{X_2}) = -52, \\
F(\Teul) = 5, & F(\Teul_{X_1}) = 3, &  F(\Teul_{X_2}) = 2, \\
I(X_1,X_2) = 120, & \Dgoth(\Teul,X_1)=24, & \Dgoth(\Teul,X_2) = 20.
\end{array}
$$
We have $\delta(\Teul) = \frac12( F(\Teul) - M(\Teul) ) = 139$
and similarly $\delta(\Teul_{X_1}) =36$ and $\delta(\Teul_{X_2}) = 27$,
so equality does hold in \eqref{837yigv3r81uy1202ehe}.
\end{example}


\begin{parag}
{\bf Remark for the expert, related to Example \ref{o873646728378uejdwe9eo}.}
Let $s=xy+1$, $p=x^2y+x+1$, $u=s^2+y$. 
Fix $a_0, a_1 \in \Comp$ such that $a_0 \neq0$ and $a_1\neq1$.
For each $t \in \Comp$, define $F_t = p^2u+a_1ps+a_0s+t \in \Comp[x,y]$;
this $F_t$ is the case $n=2$ of the polynomial $F_n^t$ studied in \cite{BartoloCassouLuengo_1994}.
Thm 4 of that paper implies that $F_t$ is irreducible for every $t \in \Comp$.
For each $t \in \Comp$, \cite{BartoloCassouLuengo_1994} determines the Eisenbud-Neumann splice diagram of
the link at infinity of the curve $F_t=0$; let us denote this diagram by $\EN(F_t)$, for convenience.
It is shown that $\EN(F_t)$ is the same for all $t\in\Comp \setminus \{0,b\}$, where  $b = \frac14 (a_1-1)^2$.
Figure 2 (resp.\ 3, 4) of the cited paper displays $\EN(F_t)$ for $t \in \Comp \setminus\{0,b\}$ (resp.\ $\EN(F_0)$, $\EN(F_b)$).
These diagrams can be interpreted as elements of $\DTpr$ by replacing each vertex of valency $1$ other than the root by an arrow decorated by $(0)$.
If we do that, we find that $\EN(F_0)$ and $\EN(F_b)$ are, respectively, the trees $\Teul_{X_1}$ and $\Teul_{X_2}$ in part (a) of Ex.\ \ref{o873646728378uejdwe9eo}.

It is known that, given an irreducible polynomial $P \in \Comp[x,y]$, the genus of the curve ``$P = 0$'' can be computed from the diagram $\EN(P)$.
Moreover, the formula for doing so is consistent with our definition of the genus $g(\Teul)$
of the tree $\Teul$ obtained from $\EN(P)$ by replacing each vertex of valency $1$ other than the root by an arrow decorated by $(0)$.
It follows that $g(\Teul_{X_1}) = 0$ and $g(\Teul_{X_2}) = 1$ (in part (a) of Ex.\ \ref{o873646728378uejdwe9eo}) are the genera of the curves $F_0=0$ and $F_b=0$ respectively. 

The diagrams $\EN( F_0 F_b )$ and $\EN( F_0 F_b^2 )$ are not displayed in \cite{BartoloCassouLuengo_1994},
but one can see that $\EN( F_0 F_b )$ (resp.\ $\EN( F_0 F_b^2 )$) is the tree $\Teul$ of part (a) (resp.\ (b)) of Ex.\ \ref{o873646728378uejdwe9eo}.
Although it does not make sense to speak of the genus of a reducible curve such as $F_0 F_b = 0$ or $F_0 F_b^2 = 0$,
our theory assigns numbers $g(\Teul)$ to the corresponding trees.

Let us also point out that, in part (a) (resp.\ (b)), $I(X_1,X_2)$ is the intersection number of the curves
$F_0=0$ and $F_b=0$ (resp.\ $F_0=0$ and $F_b^2=0$), over all intersection points in $\Comp^2$.
Since these two curves are distinct fibers of a polynomial map $\Comp^2 \to \Comp$, they do not intersect in $\Comp^2$, which is consistent with the fact that $I(X_1,X_2)=0$.
\end{parag}

\bigskip
\bigskip

\begin{parag}
{\bf Remark for the expert, related to Example \ref{87652ewtg2s5tsx2qadfvg2xdj3eoo009pokl}.}
Let
$$
f =
((x^2-y^3)^2+xy^5)^2 \, (x^2-y^5) \, 
((x^2-2y^3)^2+x^5) \, (x^2-y^3)^3
\in \Comp[x,y]
$$
and consider the Eisenbud-Neumann splice diagram of the germ at the origin of the curve $f=0$.
The tree $\Teul$ in Ex.\ \ref{87652ewtg2s5tsx2qadfvg2xdj3eoo009pokl} is obtained from that diagram
by replacing each vertex of valency $1$ other than the root by an arrow decorated by $(0)$.
The trees $\Teul_{X_1}$ and $\Teul_{X_2}$ correspond, respectively, to the germs at the origin of the curves
$$
((x^2-y^3)^2+xy^5)^2(x^2-y^5) =0 \quad \text{and} \quad  ((x^2-2y^3)^2+x^5)(x^2-y^3)^3 =0 
$$
and $I(X_1,X_2)$ is the local intersection number of these two germs.
The example calculates the $\delta$-invariants of these three singularities.
\end{parag}


\section{The Alexander polynomial of a positive decorated tree}

\newcommand{\da}{\omega}

By analogy to the Alexander polynomial of a fibered multilink \cite[Prop.\ 5.1(iii)]{Eisenbud_Neumann_1985},
this section defines the Alexander polynomial $\AL(\Teul) \in \Integ[t]$ of $\Teul$ (where $t$ is an indeterminate),
for  $\Teul$ in the class of ``positive'' decorated trees defined in \ref{27hed93wtr37zn2wd0owdkg}.
We show, among other things, that $\AL(\Teul)$ is a finite product of cyclotomic polynomials.
We also study a polynomial $p(\Teul) \in \Integ[t]$ whose square divides $\AL(\Teul)$,
and whose definition is analogous to that of the polynomial $p(t)$ that appears in \cite[Thm 14.1]{Eisenbud_Neumann_1985}.
Proposition \ref{oiunuytdews23fegi7lkpo9iuuhgfvxcvxbhjj3ndgfe} gives a criterion for determining whether $p(\Teul)$ is equal to $1$.

\begin{definition}  \label {oi83yngvrdserwswstukjijo9ijple098yfds3sw12}
Given $\Teul = (\Veul,\Aeul,\Eeul,f,q) \in \DT$, define the following sets and numbers.\footnote{Recall that the gcd of the empty set is $0$.}
\begin{itemize}

\item $\da(\Teul) = \gcd\setspec{ f(\alpha) }{ \alpha \in \Aeul\setminus\Aeul_0 } = \gcd\setspec{ F(\alpha) }{ \alpha \in \Aeul\setminus\Aeul_0 }$

\item $\Neul = \Neul(\Teul) = \setspec{ v \in \Veul }{ N_v > 0 }$

\item $E = E(\Teul) = \setspec{ e \in \Eeul }{ \text{both endpoints of $e$ are in $\Neul$} }$

\item Given $e = \{x,y\} \in \Eeul$, let $\rho_e = \gcd( p(x,e) , p(y,e) )$.

\item Given $v \in \Veul$, let $\Eeul_v$ be the set of all edges $e \in \Eeul$ that are incident to $v$, and define $\rho_v = \gcd\setspec{ \rho_e }{ e \in \Eeul_v }$.

\item Given $v \in \Veul$ and $e \in \Eeul_v$, define $\Neul(v,e) = \setspec{ w \in \Neul }{ \text{$e$ is in $\gamma_{v,w}$} }$
and $c(v,e) = \gcd\setspec{ \rho_w }{ w \in \Neul(v,e) }$.

\end{itemize}
\end{definition}

\begin{lemma} \label {983456y92304xm43t9evg7}
Let $\Teul = (\Veul,\Aeul,\Eeul,f,q) \in \DT$, $\Neul = \Neul(\Teul)$ and $E = E(\Teul)$.
\begin{enumerata}

\item If $e \in \Eeul$ is incident to $v \in \Veul$ then $\rho_v \mid \rho_e \mid N_v$.

\item If $e \in \Eeul$ has at least one endpoint in $\Neul$ then $\rho_e > 0$.

\item For each $v \in \Neul$, $\rho_v > 0$.

\item If $e = \{v,w\} \in E$ then $c(v,e), c(w,e) > 0$, $c(v,e) \mid \rho_w \mid \rho_e$ and \linebreak
$c(w,e) \mid \rho_v \mid \rho_e$.

\end{enumerata}
\end{lemma}

\begin{proof}
(a) Since $N_v = Q(e,v) p(v,e) + q(e,v) p^*(v,e)$ and $\rho_e$ divides both $p(v,e)$ and $p^*(v,e)$, we have $\rho_e \mid N_v$.
The definition of $\rho_v$ immediately implies that $\rho_v \mid \rho_e$, so (a) follows.

(b) By assumption, $e$ is incident to some $v \in \Neul$.
Since $N_v>0$ and (by (a)) $\rho_e \mid N_v$, we have $\rho_e \neq 0$, i.e., $\rho_e>0$.

(c) Let $v \in \Neul$.
Clearly, $\Eeul_v \neq \emptyset$; choose $e \in \Eeul_v$, then $\rho_e>0$ by (b) and $\rho_v \mid \rho_e$ by (a), so $\rho_v \neq 0$ and hence $\rho_v>0$.

(d) Let $e = \{v,w\} \in E$. We have $w \in \Neul(v,e)$, so $c(v,e) \mid \rho_w$.
Part (a) gives $\rho_w \mid \rho_e$, so $c(v,e) \mid \rho_w \mid \rho_e$. Part (b) gives $\rho_e>0$, so $c(v,e)>0$.
By symmetry, $c(w,e) > 0$ and $c(w,e) \mid \rho_v \mid \rho_e$.
\end{proof}

\begin{definition} \label {27hed93wtr37zn2wd0owdkg}
We say that a tree $\Teul = (\Veul,\Aeul,\Eeul,f,q) \in \DT$ is {\it positive\/} if 
\begin{equation*}
\text{$\Veul \neq \emptyset$ \quad and \quad $N_x>0$ for all $x \in \Veul \cup \Aeul_0$.}
\end{equation*}
\end{definition}

\begin{lemma}  \label {9823ycswwwcslpd0dvcjidnfdhmlgxcgdgdhu4}
Let $\Teul = (\Veul,\Aeul,\Eeul,f,q)$ be a positive tree.
\begin{enumerata}

\item $\Aeul \setminus \Aeul_0 \neq \emptyset$ and $\Neul(\Teul) = \Veul$

\item $E(\Teul) = \setspec{ e \in \Eeul }{ \text{both endpoints of $e$ are in $\Veul$} }$

\item $\da(\Teul)  = \gcd\setspec{ \rho_v }{ v \in \Veul }\ge 1$

\end{enumerata}
\end{lemma}

\begin{proof}
(a) and (b) are clear; we prove (c).  Let $\rho =\gcd\setspec{ \rho_v }{ v \in \Veul }$.
Note that $\da(\Teul) \ge 1$, because $\Aeul \setminus \Aeul_0 \neq \emptyset$.
We have $\da(\Teul) \mid p(x,e)$ for every choice of $x \in \Veul \cup \Aeul$ and $e \in \Eeul$ such that $e$ is incident to $x$.
So $\da(\Teul) \mid \rho_e$ for all $e \in \Eeul$, and consequently 
$\da(\Teul) \mid \rho_v$ for all $v \in \Veul$. Thus, $\da(\Teul) \mid \rho$.
Conversely, consider $\alpha \in \Aeul \setminus \Aeul_0$. Let $e = \{v,\alpha\}$ be the unique edge incident to $\alpha$.
Then $\rho_e \mid f(\alpha)$ is clear; since $v \in \Veul$, we have $\rho \mid \rho_v \mid \rho_e \mid f(\alpha)$.
This being true for each $\alpha \in \Aeul \setminus \Aeul_0$, we get $\rho \mid \da(\Teul)$, so $\rho = \da(\Teul)$.
\end{proof}

\begin{definition} \label {w8i3ubc98b29716w7sv9nxbsai8}
Let $\Teul$ be a positive tree. Let $E = E(\Teul)$ and recall that 
$$
E = \setspec{ e \in \Eeul }{ \text{both endpoints of $e$ are in $\Veul$} } .
$$
Lemma \ref{983456y92304xm43t9evg7} implies that  $( \rho_v )_{ v \in \Veul }$ and  $( \rho_e )_{ e \in E }$ are families of strictly positive integers and that
$\rho_v \mid \rho_e$ for every pair $(v,e) \in \Veul \times E$ such that $e$ is incident to $v$.
Define
$$
p(\Teul) = \displaystyle  \frac{ (t^{\da(\Teul)}-1) \prod_{ e \in E } (t^{\rho_e}-1) }{ \prod_{ v \in \Veul } (t^{\rho_v}-1) } \, ,
$$
which is an element of $\Rat(t)$.
Proposition \ref{pdoij2b3i8cb1921oi4sen}, below, shows that $p(\Teul)$ is in fact a product of cyclotomic polynomials.
\end{definition}

\begin{remark}  \label {bfvaf43e8fxbnb3fvcedcxxsswflp}
{\it If $\Teul = (\Veul,\Aeul,\Eeul,f,q)$ is positive and satisfies $| \Veul | = 1$ then $p(\Teul)=1$.}
Indeed, if $\Veul = \{v\}$ then $\da(\Teul) = \rho_v$ by Lemma \ref{9823ycswwwcslpd0dvcjidnfdhmlgxcgdgdhu4}(c),
and the above definition gives $p(\Teul) = { (t^{\da(\Teul)}-1) } / { (t^{\rho_v}-1) }$.
\end{remark}

If $\Teul$ is positive then, a priori, $p(\Teul)$ is a rational function in $\Rat(t)$.
We will show in Prop.\ \ref{pdoij2b3i8cb1921oi4sen} that $p(\Teul)$ is in fact a polynomial.
More precisely, we will show that $p(\Teul) \in \Mpol$ (defined in \ref{p9i32ho9efnc023wh7d}).

\begin{notation}  \label {cyclotomic9283748737354676}
Let $t$ be an indeterminate over $\Comp$.
If $s$ is a positive integer, we write $\Phi_s = \Phi_s(t)$ for the $s$-th cyclotomic polynomial, i.e., $\Phi_s(t) = \prod_{z \in E_s} (t-z)$ where 
$E_s$ is the set of primitive $s$-th roots of unity in $\Comp$.  
It is well known that $\Phi_s(t)$ belongs to $\Integ[t]$, is monic, and is irreducible in $\Rat[t]$.
\end{notation}

\begin{parag} \label {p9i32ho9efnc023wh7d}
Let $\Mpol$ be the set of monic polynomials $f \in \Rat[t]$ such that every $z \in \Comp$ satisfying $f(z)=0$ is a root of unity.
Let $\M = \setspec{ f/g }{ f,g \in \Mpol }$. 
The following are easily verified.
\begin{enumerata}

\item $\Mpol$ is the set of finite products of cyclotomic polynomials (the empty product being equal to $1$).  In particular, $\Mpol \subset \Integ[t]$.

\item If $f,g \in \Mpol$ and $h$ is the gcd of $f$ and $g$ in $\Rat[t]$, where we choose $h$ monic,
then $h$, $f/h$ and $g/h$ belong to $\Mpol$.

\item $\M$ is a subgroup of the multiplicative group $\Rat(t)^*$.

\item \label {kjhbiuwn543h43j568h45uy} 
If $f,g \in \Mpol$ and $f \mid g$ in $\Comp[t]$ then  $f \mid g$ in $\Mpol$ (by which we mean that there exists $h \in \Mpol$ such that $fh=g$).
Thus, $\M \cap \Comp[t] = \Mpol$.


\end{enumerata}
%
\end{parag}

\begin{lemma} \label {2kj34h56789wijhegfdyusiwoeli8o990198jcegfedre}
Let $(V,E)$ be a tree, $V'$ a subset of $V$ and $E'$ the set of edges $e \in E$ that are incident to at least one element of $V'$.
If $V' \neq V$ then $| V' | \le |E'|$.
\end{lemma}

\begin{proof}
Define
\begin{align*}
E_1 &= \setspec{ e \in E }{ \text{exactly one endpoint of $e$ is in $V'$} }, \\
E_2 &= \setspec{ e \in E }{ \text{both endpoints of $e$ are in $V'$} } 
\end{align*}
and note that $E' = E_1 \cup E_2$ and $E_1 \cap E_2 = \emptyset$.
Consider the subgraph $F = (V',E_2)$ of the tree $(V,E)$.
As $F$ is a forest, we have $|V'| - |E_2| = |\Cscr|$, where $\Cscr$ denotes the set of connected components of $F$.
It follows that $|V'| - |E'| = |\Cscr| - |E_1|$; so, to prove the Lemma, it suffices to show that $|E_1| \ge |\Cscr|$.

For each $e \in E_1$, there is exactly one $C_e \in \Cscr$ such that $e$ has an endpoint in $C_e$.
We claim that the map $f : E_1 \to \Cscr$, $f(e) = C_e$, is surjective.
Indeed, let $C \in \Cscr$. Choose $v \in V \setminus V'$ and let $\gamma = (x_0,\dots,x_n)$ be the shortest path in $(V,E)$ such that $x_0=v$ and $x_n$ is in $C$.
Observe that $n>0$ and that if $x_{n-1} \in V'$ then $x_{n-1}$ is in $C$, contradicting the minimality of $\gamma$; so $x_{n-1} \notin V'$
and consequently the edge $e = \{x_{n-1},x_n\}$ belongs to $E_1$ and satisfies $f(e)=C$. So $f$ is surjective.
This implies that $|E_1| \ge |\Cscr|$, which proves the Lemma.
\end{proof}

\begin{lemma} \label {djhv8172rejm4wdqg672dwbcvtAv26}
Given a tree $(V,E)$ such that $V \neq \emptyset$ and two families $(\rho_v)_{v \in V}$ and $(\rho_e)_{e \in E}$ of positive integers,
define
$$
\displaystyle P\big( V,E, (\rho_e)_{e \in E}, (\rho_v)_{v \in V} \big) = \frac{ (t^d-1) \prod_{e \in E} (t^{\rho_e}-1)}{ \prod_{v \in V} (t^{\rho_v} - 1) }
$$
where $d = \gcd\setspec{\rho_v}{ v \in V }$. 
Suppose that $(\rho_v)_{v \in V}$ and $(\rho_e)_{e \in E}$ satisfy
\begin{equation} \label {p92398t45e2qsdf40dllkfvxd5}
\text{for every $e \in E$ and $v \in V$, if $e$ is incident to $v$ then $\rho_v \mid \rho_e$.}
\end{equation}
Then $P\big( V,E, (\rho_e)_{e \in E}, (\rho_v)_{v \in V} \big) \in \Mpol$.
\end{lemma}

\begin{proof}
Given $s \in \Nat \setminus \{0\}$ and $f \in \Mpol$, let $V_s(f)$ be the largest $n \in \Nat$ such that $\Phi_s^{n}$ divides $f$ in $\Rat[t]$.
Let $N = (t^d-1) \prod_{e \in E} (t^{\rho_e}-1)$ and $D = \prod_{v \in V} (t^{\rho_v} - 1)$.
The assumption $V \neq \emptyset$ implies that $d\ge1$, which implies that $N \in \Mpol$; it is clear that $D \in \Mpol$.
We have to show that $N/D \in \Mpol$;
since $N,D \in \Mpol$, it suffices to show that  $V_s(D) \le V_s(N)$ for all $s \in \Nat \setminus \{0\}$.

Let $s \in \Nat \setminus \{0\}$.
Let $V' = \setspec{ v \in V }{ \text{$s$ divides $\rho_v$} }$ and note that $V_s(D) = | V' |$.

If $V'=V$ then $s \mid \rho_v$ for all $v \in V$, so $s$ divides $d$ too,
and assumption \eqref{p92398t45e2qsdf40dllkfvxd5} implies that $s \mid \rho_e$ for all $e \in E$.
This gives $V_s(N) = 1 + |E|$ and $V_s(D) = |V|$.
We have $|V| = |E| + 1$ because $(V,E)$ is a tree, so $V_s(D) = V_s(N)$ in this case.

If $V' \neq V$ then Lemma \ref{2kj34h56789wijhegfdyusiwoeli8o990198jcegfedre} gives $| V' | \le |E'|$,
where $E'$ denotes the set of edges $e \in E$ that are incident to at least one element of $V'$.
Assumption \eqref{p92398t45e2qsdf40dllkfvxd5} implies that $s \mid \rho_e$ for all $e \in E'$,
so $V_s(N) \ge | E' | \ge |V'| = V_s(D)$. 

Thus, $V_s(D) \le V_s(N)$ for all $s \in \Nat \setminus \{0\}$. This proves the Lemma.
\end{proof}

\begin{proposition}  \label {pdoij2b3i8cb1921oi4sen}
If $\Teul$ is a positive tree then $p(\Teul) \in \Mpol$.
\end{proposition}

\begin{proof}
We noted in Def.\ \ref{w8i3ubc98b29716w7sv9nxbsai8} that  $( \rho_v )_{ v \in \Veul }$ and  $( \rho_e )_{ e \in E }$ are families of strictly positive integers and that
$\rho_v \mid \rho_e$ for every pair $(v,e) \in \Veul \times E$ such that $e$ is incident to $v$.
Lemma \ref{9823ycswwwcslpd0dvcjidnfdhmlgxcgdgdhu4}(c) states that $\da(\Teul) = \gcd\setspec{ \rho_v }{ v \in \Veul }$,
so  $p(\Teul) = P\big( \Veul,E, (\rho_e)_{e \in E}, (\rho_v)_{v \in \Veul} \big)$.
So Lemma \ref{djhv8172rejm4wdqg672dwbcvtAv26} implies that $p(\Teul) \in \Mpol$.
\end{proof}

The following notation is convenient for studying $p(\Teul)$ and the other polynomials considered in this section.

\newcommand{\F}{\Fgoth}

\begin{notation}  \label {o8y4762367248rud29tgfn9ds3782}
Given $b_1, \dots, b_s \in \Nat \setminus \{0\}$ where $s\ge1$, define the polynomial $\F(b_1,\dots,b_s) = t^b - 1 \in \Integ[t]$, where $b = \gcd(b_1,\dots,b_s)$.
Observe the following facts.
\begin{itemize}

\item If $n \in \Nat\setminus \{0\}$ then $\F(n) = \prod_{s \in X_n} \Phi_s$, where $X_n$ (temporarily) denotes the set of $s \in \Nat\setminus\{0\}$ satisfying $s \mid n$.

\item If $m,n$ are positive integers then $\F(m) \mid \F(n) \iff m \mid n$, and if $m\mid n$ then
$\frac{\F(n)}{\F(m)} = \prod_{s \in X_n \setminus X_m} \Phi_s$.

\item If $a,b,c$ are positive integers such that $\gcd(a,b)=c$, then $\frac{ \F(a) \F(b) }{ \F(c) } = \lcm( \F(a), \F(b) ) = \prod_{s \in X_a \cup X_b} \Phi_s$.

\end{itemize}
Also observe that if $\Teul = (\Veul,\Aeul,\Eeul,f,q)$ is a positive tree and $E = E(\Teul)$ then 
$$
p(\Teul) = \frac{ \F( \da(\Teul) ) \prod_{e \in E} \F(\rho_e) }{ \prod_{v \in \Veul} \F( \rho_v ) } .
$$
\end{notation}

\begin{definition}  \label {definofalexanderpolynomialo94i92gefk}
Given a positive tree $\Teul = (\Veul,\Aeul,\Eeul,f,q)$, we define:
\begin{align*}
\displaystyle \AL(\Teul) & = (t^{\da(\Teul)}-1) \prod_{ v \in \Veul \cup \Aeul_0 } (t^{N_v}-1)^{ \delta_v - 2} \\[.5mm]
\displaystyle \AL'(\Teul) & = (t^{\da(\Teul)}-1)^{-1} \prod_{ \alpha \in \Aeul \setminus \Aeul_0 } (t^{F(\alpha)}-1)\\[.5mm]
\AL''(\Teul) & = \AL(\Teul) / \AL'(\Teul) 
\quad \text{and} \quad \Qscr(\Teul) = \AL''(\Teul) / p(\Teul)^2 .
\end{align*}
We know by Prop.\ \ref{pdoij2b3i8cb1921oi4sen} that  $p(\Teul)$ is a polynomial,
and Thm \ref{Eiuy76234763492c3n0dcmnjbbxznevcnzkCkxvgjcnhgjvy4y} shows that 
$\AL(\Teul)$,  $\AL'(\Teul)$,  $\AL''(\Teul)$ and  $\Qscr(\Teul)$ are polynomials.

We call $\AL(\Teul)$ the {\it Alexander polynomial\/} of $\Teul$,
by analogy to the Alexander polynomial of a fibered multilink \cite[Prop.\ 5.1(iii)]{Eisenbud_Neumann_1985}.
\end{definition}

\begin{definition} \label {scresredbtgjo9ikjyuhlpy0vl0lprkdjuqqas2pl202j8dgvc5x}
Let $\Teul$ be a positive tree.
If $e = \{v,w\} \in E(\Teul)$, define
$$
W(\Teul,e) = \prod_{s \in X(\Teul,e)} \Phi_s
$$
where $X(\Teul,e)$ denotes the set of $s \in \Nat \setminus \{0\}$ satisfying $s \mid \rho_e$ and $s \nmid c(v,e)$ and $s \nmid c(w,e)$.
Observe that $W(\Teul,e) \in \Mpol$.
\end{definition}

\begin{lemma}  \label {aow8uytuc62387r9234ump0eifoe}
Suppose that $\Teul$ is a positive tree and that $e = \{v_1,v_2\} \in E(\Teul)$.
Let $\{ \Teul_1, \Teul_2 \} = \ENSpl(\Teul,e)$, where the notation is such that $v_i$ is in $\Teul_i$.
Then $\Teul_1$ and $\Teul_2$ are positive trees and the following hold.
\begin{enumerata}

\item $\da(\Teul_1) = c(v_2,e)$ and $\da(\Teul_2) = c(v_1,e)$, where $c(v_1,e)$ and $c(v_2,e)$ are computed in $\Teul$.
Moreover, $\gcd( \da(\Teul_1) , \da(\Teul_2) )  = \da(\Teul)$.

\item $p(\Teul) = W(\Teul,e) p(\Teul_1) p(\Teul_2)$

\item $\AL''(\Teul) = W(\Teul,e)^2 \AL''(\Teul_1) \AL''(\Teul_2)$

\item $\Qscr(\Teul) = \Qscr(\Teul_1) \Qscr(\Teul_2)$

\end{enumerata}
\end{lemma}

\begin{proof}
Let $p_i = p(v_i,e)$, $i=1,2$.  We use the following notation:
\begin{equation}  \label {t6frcesshtubcxbnfgg367wysj7hdfruf2tftv}
\scalebox{.85}{\raisebox{-7\unitlength}{\begin{picture}(30,14)(-37,-10)
\put(-30,0){\circle{1}}
\put(-30,-2){\makebox(0,0)[t]{\tiny $v_1$}}
\put(-22.5,-1){\makebox(0,0)[t]{\tiny $e$}}
\put(-15,-2){\makebox(0,0)[t]{\tiny $v_2$}}
\put(-29.5,0){\line(1,0){14}}
\put(-15,0){\circle{1}}
\put(-30.35355,0.35355){\line(-1,1){4}}
\put(-30.35355,-0.35355){\line(-1,-1){4}}
\put(-14.64645,0.35355){\line(1,1){4}}
\put(-14.64645,-0.35355){\line(1,-1){4}}
\put(-22.5,-8){\makebox(0,0){$\displaystyle\underbrace{\rule{29\unitlength}{0mm}}_{\Teul}$}}
\end{picture}}}
\qquad\quad
\scalebox{.85}{\raisebox{-7\unitlength}{\begin{picture}(78,14)(-39,-10)

\put(-20,0){\circle{1}}
\put(-22,0){\makebox(0,0)[r]{\footnotesize $v_1$}}
\put(-19.5,0){\vector(1,0){12}}
\put(-7,0){\makebox(0,0)[l]{\footnotesize $(p_1)$}}
\put(-9,1){\makebox(0,0)[b]{\tiny $\alpha_1$}}
\put(9,1){\makebox(0,0)[b]{\tiny $\alpha_2$}}

\put(-15,-1){\makebox(0,0)[t]{\tiny $e_1$}}
\put(15,-1){\makebox(0,0)[t]{\tiny $e_2$}}

\put(20,0){\circle{1}}
\put(22,0){\makebox(0,0)[l]{\footnotesize $v_2$}}
\put(19.5,0){\vector(-1,0){12}}
\put(7,0){\makebox(0,0)[r]{\footnotesize $(p_2)$}}


\put(-20.35355,0.35355){\line(-1,1){4}}
\put(-20.35355,-0.35355){\line(-1,-1){4}}

\put(20.35355,0.35355){\line(1,1){4}}
\put(20.35355,-0.35355){\line(1,-1){4}}

\put(-16,-8){\makebox(0,0){$\displaystyle\underbrace{\rule{27\unitlength}{0mm}}_{\Teul_1}$}}
\put(16,-8){\makebox(0,0){$\displaystyle\underbrace{\rule{27\unitlength}{0mm}}_{\Teul_2}$}}


\end{picture}}}
\end{equation}
Write $\Teul = (\Veul,\Aeul,\Eeul,f,q)$ and, for each $i \in \{1,2\}$,
$\Teul_i = (\Veul_i,\Aeul_i,\Eeul_i,f_i,q_i)$ and $\Aeul_{0,i} = \setspec{ \alpha \in \Aeul_i }{ f_i(\alpha) = 0 }$.
Then $\big\{ \Veul_1, \Veul_2 \big\}$ is a partition of $\Veul$
and $\{ \Veul_1 \cup \Aeul_{0,1} ,  \Veul_2 \cup \Aeul_{0,2} \}$ is a partition of $\Veul \cup \Aeul_0$.
For each $x \in \Veul_i \cup \Aeul_{0,i}$, the value of $N_{x}$ is the same in $\Teul_i$ and in $\Teul$;
so $\Teul_1$ and $\Teul_2$ are positive trees.

It is clear that $\rho_{e_1} = \rho_e = \rho_{e_2}$, where $\rho_{e_i}$ is computed in $\Teul_i$ and $\rho_e$ is computed in $\Teul$.
If $\epsilon$ is an edge of $\Teul$ other than $e$, then $\epsilon$ is an edge of $\Teul_{j(\epsilon)}$ for a unique ${j(\epsilon)} \in \{1,2\}$
and $\rho_\epsilon$ has the same value in $\Teul$ and in $\Teul_{ j(\epsilon) }$.
It follows that if $v \in \Veul_i$ then $\rho_{v}$ has the same value in $\Teul_i$ and in $\Teul$.
Since $\Veul_1 = \Neul(v_2,e)$ and $\Veul_2 = \Neul(v_1,e)$ (where the right-hand-sides are taken in $\Teul$),
it also follows (using Lemma \ref{9823ycswwwcslpd0dvcjidnfdhmlgxcgdgdhu4}(c)) that
$\da(\Teul_1) = \gcd\setspec{ \rho_v }{ v \in \Veul_1 } = c(v_2,e)$ and $\da(\Teul_2) = \gcd\setspec{ \rho_v }{ v \in \Veul_2 } = c(v_1,e)$, where the $c(v_i,e)$ are computed in $\Teul$.
Also note that $\gcd( \da(\Teul_1) , \da(\Teul_2) )  = \gcd \setspec{ \rho_v }{ v \in \Veul_1 \cup \Veul_2 }$;
since $\Veul_1 \cup \Veul_2 = \Veul$, another application of Lemma \ref{9823ycswwwcslpd0dvcjidnfdhmlgxcgdgdhu4}(c) gives
$\gcd( \da(\Teul_1) , \da(\Teul_2) )  = \da(\Teul)$, which proves assertion (a).

Let us write $E = E(\Teul)$ and $E_i = E(\Teul_i)$ for $i=1,2$.  Then $\big\{ E_1, \{ e \} , E_2 \big\}$ is a partition of $E$ and:
\begin{align}  \label {i83876fvvnbwgyuear73e8ugch}
p(\Teul) &= \textstyle \F( \da(\Teul) ) \frac{ \prod_{ \epsilon \in E } \F(\rho_\epsilon) }{ \prod_{ v \in \Veul } \F(\rho_v) } \\
\notag &=  \textstyle \frac{ \F( \da(\Teul) ) \F(\rho_e) }{ \F( \da(\Teul_1) ) \F( \da(\Teul_2) ) }
\left( \frac{ \F( \da(\Teul_1) ) \prod_{ \epsilon \in E_1 } \F(\rho_\epsilon) }{ \prod_{ v \in \Veul_1 } \F(\rho_v) } \right)
\left( \frac{ \F( \da(\Teul_2) ) \prod_{ \epsilon \in E_2 } \F(\rho_\epsilon) }{ \prod_{ v \in \Veul_2 } \F(\rho_v) } \right) \\
\notag &=  \textstyle \frac{ \F( \da(\Teul) ) \F(\rho_e) }{ \F( \da(\Teul_1) ) \F( \da(\Teul_2) ) } p(\Teul_1) p(\Teul_2) \, .
\end{align}
Let us prove:
\begin{equation}  \label {y7543ws2aryji8homlp0d3XiCm2nx2nge5xa34ja}
\textstyle
\frac{ \F( \da(\Teul) ) \F(\rho_e) }{ \F( \da(\Teul_1) ) \F( \da(\Teul_2) ) } = W(\Teul,e) .
\end{equation}
For each $i=1,2$, let $A_i$ be the set of $s \in \Nat \setminus \{0\}$ satisfying $s \mid \da(\Teul_i)$.
Since $\gcd( \da(\Teul_1) , \da(\Teul_2) ) = \da(\Teul)$, we have
$\frac{ \F( \da(\Teul_1) ) \F( \da(\Teul_2) ) }{\F( \da(\Teul) )} = \lcm( \F( \da(\Teul_1) ) , \F( \da(\Teul_2) ) )
= \prod_{s \in A_1 \cup A_2} \Phi_s$.
In $\Teul$, the edge $e$ is incident to $v_i$, so $\rho_{v_i} \mid \rho_e$; since $\da(\Teul_i) \mid \rho_{v_i}$, we get $\da(\Teul_i) \mid \rho_e$ (for each $i=1,2$).
Thus, $A_1 \cup A_2 \subseteq A$, where $A$ is the set of positive divisors of $\rho_e$.
Consequently,
\begin{equation*}
\textstyle \frac{ \F( \da(\Teul) ) \F(\rho_e) }{ \F( \da(\Teul_1) ) \F( \da(\Teul_2) ) } 
= \textstyle \frac{ \F(\rho_e) }{ \F( \da(\Teul_1) ) \F( \da(\Teul_2) ) / \F(\da(\Teul)) } 
= \frac{ \prod_{ s \in A } \Phi_s }{ \prod_{ s \in A_1 \cup A_2 } \Phi_s } 
= \prod_{ s \in A \setminus (A_1 \cup A_2) } \Phi_s \, .
\end{equation*}
Since $\da(\Teul_1) = c(v_2,e)$ and  $\da(\Teul_2) = c(v_1,e)$, we have $A \setminus (A_1 \cup A_2) = X(\Teul,e)$, 
which proves \eqref{y7543ws2aryji8homlp0d3XiCm2nx2nge5xa34ja}.
Assertion (b) follows from \eqref{i83876fvvnbwgyuear73e8ugch} and \eqref{y7543ws2aryji8homlp0d3XiCm2nx2nge5xa34ja}.

(c) Let $X_i = \setspec{ \alpha \in \Aeul \setminus \Aeul_0 }{ \text{$\alpha$ is in $\Teul_i$ } }$ for $i=1,2$.
Note that $\{ X_1 , X_2 \}$ is a partition of $\Aeul \setminus \Aeul_0$ and that (for each $i$)
$\{ X_i , \{\alpha_i\} \}$ is a partition of $\Aeul_i \setminus \Aeul_{0,i}$.
We already observed that $\{ \Veul_1 \cup \Aeul_{0,1} ,  \Veul_2 \cup \Aeul_{0,2} \}$ is a partition of $\Veul \cup \Aeul_0$.
Thus,
\begin{align}  \label {87t81278eb76cesqb2vmtnybulMbyretv3iubhMHu8iwu39}
\AL''(\Teul) &= \textstyle \F( \da(\Teul) )^2 \frac{ \prod_{ v \in \Veul \cup \Aeul_0 } \F(N_v)^{ \delta_v - 2} }{ \prod_{ \alpha \in \Aeul \setminus \Aeul_0 } \F(F(\alpha)) } \\
\notag &= \textstyle \F( \da(\Teul) )^2 
\left( \frac{ \prod_{ v \in \Veul_1 \cup \Aeul_{0,1} } \F(N_v)^{ \delta_v - 2} }{ \prod_{ \alpha \in X_1 } \F(F(\alpha)) } \right)
\left( \frac{ \prod_{ v \in \Veul_2 \cup \Aeul_{0,2} } \F(N_v)^{ \delta_v - 2} }{ \prod_{ \alpha \in X_2 } \F(F(\alpha)) } \right)
\end{align}
If $v \in \Veul_i \cup \Aeul_{0,i}$ then $N_v$ and $\delta_v$ have the same values in $\Teul$ and in $\Teul_i$.
If $\alpha \in X_i$ then $F(\alpha)$ has the same value in $\Teul$ and in $\Teul_i$.
The number $F(\alpha_i)$ computed in $\Teul_i$ is equal to $\rho_e$, so
$\prod_{ \alpha \in \Aeul_i \setminus \Aeul_{0,i} } \F(F(\alpha)) = \F( \rho_e ) \prod_{ \alpha \in X_i } \F(F(\alpha))$.
In view of \eqref{87t81278eb76cesqb2vmtnybulMbyretv3iubhMHu8iwu39}, we obtain
\begin{multline*}  
\AL''(\Teul) = \\
\scalebox{.9}{$\left( \frac{ \F( \da(\Teul) ) \F( \rho_e ) }{ \F( \da(\Teul_1) )  \F( \da(\Teul_2) ) } \right)^2
\left(  \F( \da(\Teul_1) )^2 \frac{ \prod_{ v \in \Veul_1 \cup \Aeul_{0,1} } \F(N_v)^{ \delta_v - 2} }{ \prod_{ \alpha \in \Aeul_1\setminus\Aeul_{0,1} } \F(F(\alpha)) } \right)
\left(  \F( \da(\Teul_2) )^2 \frac{ \prod_{ v \in \Veul_2 \cup \Aeul_{0,2} } \F(N_v)^{ \delta_v - 2} }{ \prod_{ \alpha \in \Aeul_2\setminus\Aeul_{0,2} } \F(F(\alpha)) } \right)$} \\
= \textstyle \left( \frac{ \F( \da(\Teul) ) \F( \rho_e ) }{ \F( \da(\Teul_1) )  \F( \da(\Teul_2) ) } \right)^2 \AL''(\Teul_1) \AL''(\Teul_2) 
= \textstyle W(\Teul,e)^2 \AL''(\Teul_1) \AL''(\Teul_2) ,
\end{multline*}
the last equality by \eqref{y7543ws2aryji8homlp0d3XiCm2nx2nge5xa34ja}.
This proves (c).

Parts (b) and (c) give
$$
p(\Teul) = W(\Teul,e) p(\Teul_1) p(\Teul_2) \quad \text{and} \quad  \AL''(\Teul) = W(\Teul,e)^2 \AL''(\Teul_1) \AL''(\Teul_2) ,
$$
so $\Qscr(\Teul) = \frac{\AL''(\Teul)}{p(\Teul)^2} 
= \frac{W(\Teul,e)^2 \AL''(\Teul_1) \AL''(\Teul_2) }{ [ W(\Teul,e) p(\Teul_1) p(\Teul_2) ]^2 } = \Qscr(\Teul_1) \Qscr(\Teul_2)$, which proves (d).
\end{proof}

\begin{theorem}  \label {Eiuy76234763492c3n0dcmnjbbxznevcnzkCkxvgjcnhgjvy4y}
If $\Teul$ is a positive tree then
$$
\AL(\Teul) = \AL'(\Teul) \cdot \AL''(\Teul) = \AL'(\Teul) \cdot \Qscr(\Teul) \cdot p(\Teul)^2,
$$
where $\AL(\Teul)$,  $\AL'(\Teul)$,  $\AL''(\Teul)$,  $\Qscr(\Teul)$ and $p(\Teul)$ all belong to $\Mpol$.  Moreover,
\begin{gather*}
\deg \AL(\Teul) = \da(\Teul) - M(\Teul), \quad  \deg \AL'(\Teul) = -\da(\Teul) + F(\Teul), \\
\text{and}\ \ \deg \AL''(\Teul) = 2\da(\Teul) - M(\Teul) - F(\Teul) .
\end{gather*}
\end{theorem}

\begin{proof}
\newcommand{\ha}{\hat a}
\newcommand{\hb}{\hat b}
The claims about the degrees immediately follow from the definitions of $\AL(\Teul)$, $\AL'(\Teul)$, $\AL''(\Teul)$, $M(\Teul)$ and $F(\Teul)$
(see Definitions \ref{definofalexanderpolynomialo94i92gefk} and \ref{c9v39rf0eX9e4np8glr9t8} and Notation \ref{DdjkfblowicblwFT983764t1}).
Also, we have $p(\Teul) \in \Mpol$ by Proposition \ref{pdoij2b3i8cb1921oi4sen}.
So the real content of the Theorem is the claim that $\AL(\Teul),  \AL'(\Teul),  \AL''(\Teul), \Qscr(\Teul) \in \Mpol$. 
We proceed to prove this claim.  The first step is to show that $\Qscr(\Teul) \in \Mpol$.

Given a positive tree $\Teul = (\Veul,\Aeul,\Eeul,f,q)$, define $\| \Teul \| = | \Veul |$. 
Let us prove, by induction on $\| \Teul \|$, that $\Qscr(\Teul) \in \Mpol$ for all positive trees $\Teul$.

Let  $\Teul = (\Veul,\Aeul,\Eeul,f,q)$ be a positive tree with $\| \Teul \| = 1$.
Let $v$ be the unique vertex of $\Teul$ and let $\alpha_1, \dots, \alpha_m, \beta_1, \dots, \beta_n$ be the distinct arrows of $\Teul$,
where each $\alpha_i$ is decorated by $(0)$ and each $\beta_i$ is decorated by $(s_i)$ for some $s_i \in \Integ \setminus \{0\}$.
Note that $m\ge0$ and $n>0$.
Also let $a_i = q( \{v,\alpha_i\},v )$ for all $i \in \{1,\dots,m\}$ and $b_i = q( \{v,\beta_i\},v )$ for all $i \in \{1,\dots,n\}$.
$$
{\begin{picture}(34,16)(-16,-8)
\put(-30,0){$\Teul\,:$}

\put(0,0){\circle{1}}
\put(0,-2){\makebox(0,0)[t]{\tiny $v$}}

\put(-10.5,6.5){\makebox(0,0)[b]{\footnotesize $\alpha_1$}}
\put(-13,6){\makebox(0,0)[r]{\footnotesize $(0)$}}
\put(-10.5,-6.5){\makebox(0,0)[t]{\footnotesize $\alpha_m$}}
\put(-13,-6){\makebox(0,0)[r]{\footnotesize $(0)$}}
\put(-5,2.25){\makebox(0,0)[bl]{\tiny $a_1$}}
\put(-2.5,-1.25){\makebox(0,0)[br]{\tiny $a_m$}}
\put(-9,1){\makebox(0,0){\footnotesize $\vdots$}}
\put(9,1){\makebox(0,0){\footnotesize $\vdots$}}

\put(11,6.5){\makebox(0,0)[b]{\footnotesize $\beta_1$}}
\put(13,6){\makebox(0,0)[l]{\footnotesize $(s_1)$}}
\put(11,-6.5){\makebox(0,0)[t]{\footnotesize $\beta_n$}}
\put(13,-6){\makebox(0,0)[l]{\footnotesize $(s_n)$}}
\put(4.5,2.25){\makebox(0,0)[br]{\tiny $b_1$}}
\put(2.5,-1.5){\makebox(0,0)[bl]{\tiny $b_n$}}

\put(0.4472,0.2236){\vector(2,1){12}}
\put(0.4472,-0.2236){\vector(2,-1){12}}
\put(-0.4472,0.2236){\vector(-2,1){12}}
\put(-0.4472,-0.2236){\vector(-2,-1){12}}

\end{picture}}
$$
Define $d = \da(\Teul) = \gcd(s_1,\dots,s_n)$.
We have $p(\Teul) = 1$ by Rem.\ \ref{bfvaf43e8fxbnb3fvcedcxxsswflp}, so
\begin{align}  \label {82876e2789dn282h8126gdimBHtygBerygHGFri28fsRGUWyo97389}
&\Qscr(\Teul) = \AL''(\Teul)
=  \F( \da(\Teul) )^2 \, \frac{ \prod_{ x \in \Veul \cup \Aeul_0 } \F(N_x)^{ \delta_x - 2} }{ \prod_{ \alpha \in \Aeul \setminus \Aeul_0 } \F(F(\alpha)) } \\
\notag &= \textstyle \F( d )^2 \, \frac{ \F(N_v)^{ m+n - 2} \prod_{ i = 1 }^m  \F(N_{\alpha_i})^{ - 1} }{ \prod_{ j=1 }^n \F(F(\beta_j)) } 
=  \textstyle \frac{ \left( \prod_{ i = 1 }^m \frac{ \F(N_v) }{ \F(N_{\alpha_i}) } \right)   \left( \prod_{ j=1 }^n \frac{ \F(N_v) }{ \F(F(\beta_j)) } \right) }{ [ \F(N_v) / \F(d) ]^2 } \, .
\end{align}
It is clear from \eqref{82876e2789dn282h8126gdimBHtygBerygHGFri28fsRGUWyo97389} that if $m+n=1$ then $\Qscr(\Teul)=1$. From now-on, assume that $m+n\ge2$ and
define $(d_1,\dots,d_{m+n}) = ( F(\beta_1) , \dots, F(\beta_n) ,  N_{\alpha_1} , \dots, N_{\alpha_m} )$.
We have $d_i \ge1$ and $d_i \mid N_v$  for each $i \in \{1, \dots, m+n \}$, and \eqref{82876e2789dn282h8126gdimBHtygBerygHGFri28fsRGUWyo97389} gives
$$
\Qscr(\Teul) =  \frac{ \prod_{ i = 1 }^{m+n} \frac{ \F(N_v) }{ \F(d_i) } }{ [ \F(N_v) / \F(d) ]^2 } \, .
$$
Given a positive integer $c$, let $X_c$ be the set of positive divisors of $c$.
Then $\F(N_v) / \F(d) = \prod_{s \in X_{N_v} \setminus X_d} \Phi_s$ and, for each $i \in \{1,\dots,m+n\}$,
$\F(N_v) / \F(d_i) = \prod_{s \in X_{N_v} \setminus X_{d_i}} \Phi_s$.
So, to prove that $\Qscr(\Teul) \in \Mpol$, it suffices to show that for each $s \in  X_{N_v} \setminus X_d$,
at least two elements $i$ of $\{1,\dots,m+n\}$ are such that $s \in X_{N_v} \setminus X_{d_i}$; so it suffices to prove the following two claims:
\begin{gather}
\label {98y8765f4ec4tutruryjyuoy9ykplo7425q2eGF6F4uty5EF6D0u0rC6uBF6A}  \gcd(d_1, \dots, d_{m+n}) = d , \\ 
\label {iew7uhohmnljDmgc87069f6ku7nb69Cv5jhBA}  \begin{minipage}[t]{.7\textwidth}
$\gcd\setspec{ d_i }{ i \in I } = d$ for every choice of a subset $I$ \linebreak of $\{1, \dots, m+n \}$ such that $|I|=m+n-1$.
\end{minipage}
\end{gather}
We leave \eqref{98y8765f4ec4tutruryjyuoy9ykplo7425q2eGF6F4uty5EF6D0u0rC6uBF6A} to the reader; let us prove \eqref{iew7uhohmnljDmgc87069f6ku7nb69Cv5jhBA}.
Define $a = \prod_{i=1}^m a_i$.
For each $i \in \{1,\dots,m\}$, let $\ha_i = \prod_{ k \in \{1,\dots,m\} \setminus \{i\} } a_k$.
For each $j \in \{1,\dots,n\}$, let $\hb_j = \prod_{ k \in \{1,\dots,n\} \setminus \{j\} } b_k$.
Define $b = \sum_{ j = 1 }^n \hb_j s_j$. We have $N_v = ab$ and (for each $i$) $N_{\alpha_i} = \ha_i b$.
Since the edge decorations near $v$ are pairwise relatively prime, we have
$\gcd( \ha_1 , \dots, \ha_m ) = 1$, so $\gcd( N_{\alpha_1}, \dots, N_{\alpha_m}) = b$.

Let $I$ be as in \eqref{iew7uhohmnljDmgc87069f6ku7nb69Cv5jhBA} and let $s = \gcd\setspec{ d_i }{ i \in I }$.
To prove \eqref{iew7uhohmnljDmgc87069f6ku7nb69Cv5jhBA}, it suffices to show that $s \mid d$.
If $\{1,\dots,n\} \subseteq I$ then $s$ divides $\gcd(  F(\beta_1) , \dots, F(\beta_n) ) = d$.
So we may assume that $\{1,\dots,n\} \nsubseteq I$; without loss of generality, we may assume that $I = \{2, \dots,m+n\}$.
Then $s$ divides each $N_{\alpha_i}$ and hence $s \mid b$.

We claim that $s \mid s_1$.
This is clear if $n=1$, because $b=s_1$ in that case.
Assume that $n>1$.
Since $s$ divides each $F(\beta_j)$ with $j>1$, we have $s \mid s_j$ for each $j>1$; since $s$ divides 
$b = \sum_{ j = 1 }^n \hb_j s_j$, it follows that $s \mid \hb_1 s_1$.
For each $j \in \{2,\dots,n\}$, we have $s \mid F(\beta_j) \mid p( \beta_j , \{ \beta_j , v \} )$, so
$0 \equiv p( \beta_j , \{ \beta_j , v \} ) \equiv a s_1 \hb_{1,j} \pmod{s}$, where $\hb_{1,j} =  \prod_{ k \in \{2,\dots,n\} \setminus \{j\} } b_k$.
It follows that $s$ divides $a s_1 \gcd( \hb_{1,2}, \dots, \hb_{1,n} ) = a s_1$. Thus, $s \mid \gcd( \hb_1 s_1 , a s_1 ) = s_1$.
We showed that $s \mid s_1$ in all cases. Since for each $j>1$ we have $s \mid F(\beta_j) \mid s_j$,
we conclude that $s$ divides $\gcd(s_1,s_2,\dots,s_n) = d$.  This proves \eqref{iew7uhohmnljDmgc87069f6ku7nb69Cv5jhBA} and completes the proof that $\Qscr(\Teul) \in \Mpol$
(where $\Teul$ is an arbitrary positive tree with $\| \Teul \| = 1$).

Let  $\Teul = (\Veul,\Aeul,\Eeul,f,q)$ be a positive tree satisfying $\| \Teul \| > 1$ and:
$$
\text{$\Qscr(\Teul') \in \Mpol$ for every positive tree $\Teul'$ satisfying $\| \Teul' \| < \| \Teul \|$.}
$$
There exists an edge $e = \{ v_1, v_2 \}$ of $\Teul$ with $v_1,v_2 \in \Veul$.
Let $\{ \Teul_1, \Teul_2 \} = \ENSpl(\Teul,e)$.
Lemma \ref{aow8uytuc62387r9234ump0eifoe} implies that  $\Teul_1$, $\Teul_2$ are positive trees and $\Qscr(\Teul) = \Qscr(\Teul_1) \Qscr(\Teul_2)$.
We have  $\| \Teul_1 \| , \| \Teul_2 \| < \| \Teul \|$, so  $\Qscr(\Teul_1), \Qscr(\Teul_2) \in \Mpol$ and hence $\Qscr(\Teul) \in \Mpol$.

This completes the proof by induction; so if $\Teul$ is a positive tree then $\Qscr(\Teul) \in \Mpol$.
Since $p(\Teul) \in \Mpol$ by Prop.\ \ref{pdoij2b3i8cb1921oi4sen}, we have $\AL''(\Teul) = \Qscr(\Teul) \cdot p(\Teul)^2 \in \Mpol$.
Since $\Aeul \setminus \Aeul_0 \neq \emptyset$ and $\da(\Teul) \mid F(\alpha)$ for each $\alpha \in \Aeul \setminus \Aeul_0$,
we have $\AL'(\Teul) = \F(\da(\Teul))^{-1} \prod_{ \alpha \in \Aeul \setminus \Aeul_0 } \F(F(\alpha)) \in \Mpol$
and it follows that $\AL(\Teul) = \AL'(\Teul) \cdot \AL''(\Teul) \in \Mpol$.
So $\AL(\Teul)$,  $\AL'(\Teul)$,  $\AL''(\Teul)$,  $\Qscr(\Teul)$ and $p(\Teul)$ all belong to $\Mpol$. 
\end{proof}

Refer to Notation \ref{c9v39rf0eX9e4np8glr9t8} for $\Aeul(x,e)$.

\begin{definition}
Let $\Teul = (\Veul,\Aeul,\Eeul,f,q) \in \DT$.
\begin{enumerata}

\item A {\it separating edge\/} is an edge $e = \{x,y\} \in \Eeul$ such that $x,y \in \Veul$, $\Aeul(x,e) \neq \emptyset$ and $\Aeul(y,e) \neq \emptyset$.

\item A {\it separating vertex\/} is a vertex $v \in \Veul$ such that at least one separating edge is incident to $v$.

\end{enumerata}
\end{definition}

\begin{lemma}   \label {gtvszxtiuiuiuhnwgqfftwrjryeyh3edeiojn0tjg}
Let $\Teul = (\Veul,\Aeul,\Eeul,f,q) \in \DT$.
Consider an edge $e = \{x,y\} \in \Eeul$ such that $x,y \in \Veul$ and $\Aeul(x,e) = \emptyset$.
\begin{enumerata}

\item $\rho_e = \rho_y$ and $\rho_x \mid \rho_y$.

\item If $z \in \Veul$ is such that $e$ is in $\gamma_{x,z}$ then $\rho_x \mid \rho_e \mid \rho_z$.

\end{enumerata}
\end{lemma}

\begin{proof}
(a) Since $\Aeul(x,e) = \emptyset$, it follows that $\Aeul(y,\epsilon) = \emptyset$ and hence $p(y,\epsilon)=0$ for all $\epsilon \in \Eeul_y \setminus \{e\}$.
Consequently, if $\epsilon = \{y,z\} \in \Eeul_y \setminus \{e\}$ then $\rho_\epsilon = p(z,\epsilon) = a p(y,e)$
where $a$ is the product of the $q(\epsilon',y)$ for all $\epsilon' \in  \Eeul_y \setminus \{e,\epsilon\}$,
so $\rho_e \mid p(y,e) \mid\rho_\epsilon$.
Thus, $\rho_e \mid \rho_\epsilon$ for all $\epsilon \in \Eeul_y \setminus \{e\}$;
this means that $\rho_e \mid \rho_y$ and consequently  $\rho_e = \rho_y$.
Since $\rho_x \mid \rho_e$, we obtain $\rho_x \mid \rho_y$.

(b) Write $\gamma_{x,z} = (x_0,\dots,x_n)$, let $e_i = \{x_i, x_{i+1}\}$ for $i \in \{0, \dots, n-1\}$, 
and note that  $\Aeul(x_0,e_0) = \emptyset$ by assumption.
So, for each $i \in \{0, \dots, n-1\}$, we have $\Aeul(x_i,e_i) = \emptyset$ and hence $\rho_{x_i} \mid \rho_{x_{i+1}}$ by part (a).
So $\rho_e = \rho_y = \rho_{x_1}$ divides $\rho_z = \rho_{x_n}$ and hence $\rho_x \mid \rho_e \mid \rho_z$.
\end{proof}

\begin{proposition}   \label {oiunuytdews23fegi7lkpo9iuuhgfvxcvxbhjj3ndgfe}
Let $\Teul$ be a positive tree and $E = E(\Teul)$.
\begin{enumerata}

\item For each $e = \{v,w\} \in E$, the following are equivalent:
\begin{enumerata}
\item $W(\Teul,e) = 1$,
\item $\rho_e = c(v,e)$ or $\rho_e = c(w,e)$.
\end{enumerata}

\item Each $e \in E$ that is not separating satisfies the equivalent conditions of part {\rm (a)}.

\item The following are equivalent:
\begin{enumerati}

\item $p(\Teul) = 1$ 

\item Every $e \in E$ satisfies the equivalent conditions of part {\rm (a)}.

\item Every $e \in E$ that is separating satisfies the equivalent conditions of part {\rm (a)}.

\end{enumerati}
\end{enumerata}
\end{proposition}

\begin{proof}
(a)  Let $e = \{v,w\} \in E$.
If $W(\Teul,e)=1$ holds then $X(\Teul,e)=\emptyset$, so $\rho_e \notin X(\Teul,e)$, so $\rho_e \mid c(v,e)$ or $\rho_e \mid c(w,e)$.
Since $c(v,e) \mid \rho_e$ and $c(w,e) \mid \rho_e$ by Lemma \ref{983456y92304xm43t9evg7}(d),
we have $\rho_e = c(v,e)$ or $\rho_e = c(w,e)$.  So (a-i) implies (a-ii), and the converse is obvious.

(b) Suppose that $e \in E$ is not separating; let $v,w$ be such that $e = \{v,w\}$ and $\Aeul(v,e) = \emptyset$.
If $z \in \Neul(v,e)$ then $\gamma_{v,z}$ contains $e$, so $\rho_e \mid \rho_z$ by Lemma \ref{gtvszxtiuiuiuhnwgqfftwrjryeyh3edeiojn0tjg}(b).
This shows that $\rho_e \mid c(v,e)$. We have $c(v,e) \mid \rho_e$ by Lemma \ref{983456y92304xm43t9evg7}(d), so $\rho_e = c(v,e)$ and $e$ satisfies condition (a-ii).

(c) Part (b) implies that (c-ii) $\Leftrightarrow$ (c-iii); let us prove that (c-i) $\Leftrightarrow$ (c-ii).

Suppose that $p(\Teul)=1$ and consider $e \in E$.
Let $\{ \Teul_1, \Teul_2 \} = \ENSpl(\Teul,e)$.
Lemma \ref{aow8uytuc62387r9234ump0eifoe} implies that $\Teul_1$ and $\Teul_2$ are positive trees (so $p(\Teul_1), p(\Teul_2) \in \Mpol$ by Prop.\ \ref{pdoij2b3i8cb1921oi4sen})
and that $p(\Teul) = W(\Teul,e) p(\Teul_1) p(\Teul_2)$. Since $W(\Teul,e) \in \Mpol$, it follows that $W(\Teul,e) = 1$.
So (c-i) implies (c-ii).

Given a positive tree $\Teul = (\Veul,\Aeul,\Eeul,f,q)$, define $\| \Teul \| = | \Veul |$. 
We prove that (c-ii) implies (c-i) by induction on $\|\Teul\|$, the case $\|\Teul\| = 1$ being clear by Rem.\ \ref{bfvaf43e8fxbnb3fvcedcxxsswflp}.
Suppose that the positive tree $\Teul$ satisfies (c-ii) and $\|\Teul\| > 1$,
and that every positive tree $\Teul'$ satisfying (c-ii) and $\|\Teul'\| < \|\Teul\|$ also satisfies $p(\Teul')=1$.
Choose $e = \{v_1,v_2\} \in E(\Teul)$ such that $\Neul(v_1,e) = \{ v_2 \}$ and consider
$\{ \Teul_1, \Teul_2 \} = \ENSpl(\Teul,e)$, where the notation is such that $v_i$ is in $\Teul_i$.
Lemma \ref{aow8uytuc62387r9234ump0eifoe} implies that $\Teul_1, \Teul_2$ are positive and  $p(\Teul) = W(\Teul,e) p(\Teul_1) p(\Teul_2)$.
Note that $v_2$ is the only vertex of $\Teul_2$, so  $p(\Teul_2)=1$ by Rem.\ \ref{bfvaf43e8fxbnb3fvcedcxxsswflp}.
We have $W(\Teul,e)=1$ because $\Teul$ satisfies (c-ii), so $p(\Teul) = p(\Teul_1)$.

We claim that $\Teul_1$ satisfies (c-ii).
To see this, consider $\epsilon \in E( \Teul_1 )$;
let $v,w$ be such that $\epsilon = \{v,w\}$ and such that $w$ is in $\gamma_{v,v_1}$.
Note that $\epsilon \in E(\Teul)$ and that $\rho_\epsilon$ has the same value in $\Teul_1$ and in $\Teul$.
For each $x \in\{v,w\}$, consider the numbers $c(x,\epsilon)$ and $c^{(1)}(x,\epsilon)$, computed in $\Teul$ and $\Teul_1$ respectively.
We have to show that $\rho_\epsilon = c^{(1)}(x,\epsilon)$ for some $x \in \{v,w\}$.
It is clear that $c(w,\epsilon) = c^{(1)}(w,\epsilon)$ and that $c(v,\epsilon) \mid c^{(1)}(v,\epsilon)$,
so  $c(x,\epsilon) \mid c^{(1)}(x,\epsilon)$ for all $x \in \{v,w\}$. 
Since $\Teul$ satisfies (c-ii), there exists $x \in \{v,w\}$ such that $\rho_\epsilon = c(x,\epsilon)$ and hence $\rho_\epsilon \mid c^{(1)}(x,\epsilon)$;
since $c^{(1)}(x,\epsilon) \mid \rho_\epsilon$ by Lemma \ref{983456y92304xm43t9evg7}(d),
we obtain $\rho_\epsilon = c^{(1)}(x,\epsilon)$, which proves that $\Teul_1$ satisfies (c-ii).

Since $\Teul_1$ satisfies (c-ii) and $\|\Teul_1\| < \|\Teul\|$, the inductive hypothesis implies that $p(\Teul_1)=1$
and hence that $p(\Teul)=1$. This completes the proof by induction and hence proves that (c-ii) implies (c-i).
\end{proof}

\begin{corollary}  \label {iygnytfbcvcrwsercf5yujk9oljkhvfui7gwr689029f2a}
If $\Teul$ is a positive tree such that no element of $E(\Teul)$ is separating, then $p(\Teul)=1$.
\end{corollary}

\begin{proof}
Observe that $\Teul$ satisfies condition (c-iii) of Prop.\ \ref{oiunuytdews23fegi7lkpo9iuuhgfvxcvxbhjj3ndgfe}.
\end{proof}

\begin{lemma}
Let $\Teul = (\Veul,\Aeul,\Eeul,f,q)$ be a positive tree and $E = E(\Teul)$. Let 
\begin{align*}
\Veul_{\text{\rm s}} &= \setspec{ v \in \Veul }{ \text{$v$ is a separating vertex} }, \\
E_{\text{\rm s}} &= \setspec{ e \in E }{ \text{$e$ is a separating edge} }.
\end{align*}
Then $E_{\text{\rm s}} = \emptyset \Leftrightarrow \Veul_{\text{\rm s}}=\emptyset$,
$(\Veul_{\text{\rm s}},E_{\text{\rm s}})$ is a subtree of $(\Veul,E)$, and the following hold:
\begin{enumerata}

\item If $E_{\text{\rm s}} = \emptyset$ then $p(\Teul)=1$.

\item If $E_{\text{\rm s}} \neq \emptyset$ then $\da(\Teul) = \gcd\setspec{ \rho_v }{ v \in \Veul_{\text{\rm s}} }$ and
$\displaystyle p(\Teul) = \frac{ \F( \da(\Teul) ) \prod_{e \in E_{\text{\rm s}}} \F(\rho_e) }{ \prod_{v \in \Veul_{\text{\rm s}}} \F( \rho_v ) }$.

\end{enumerata}
\end{lemma}

\begin{proof}
\newcommand{\Veuls}{\Veul_{\text{\rm s}}}
\newcommand{\Es}{E_{\text{\rm s}}}
We leave it to the reader to verify that $E_{\text{\rm s}} = \emptyset \Leftrightarrow \Veuls =\emptyset$ and that 
$(\Veuls,\Es)$ is a (possibly empty) subtree of $(\Veul,E)$.

(a) If $E_{\text{\rm s}} = \emptyset$ then $p(\Teul)=1$ by Cor.\ \ref{iygnytfbcvcrwsercf5yujk9oljkhvfui7gwr689029f2a}.

(b) Assume that $E_{\text{\rm s}} \neq \emptyset$.
Consider $z \in \Veul \setminus \Veuls$.
Since $\Veuls \neq \emptyset$, we can find $x \in \Veuls$ such that $\gamma_{x,z} = (x_0,\dots,x_n)$ satisfies $x_1,\dots,x_n \notin \Veuls$.
Since $x_1 \notin \Veuls$, the edge $e = \{x_0,x_1\}$ is not separating.
Since $x_0 \in \Veuls$, we have $\Aeul(x_1,e) \neq \emptyset$ and hence  $\Aeul(x_0,e) = \emptyset$.
By Lemma \ref{gtvszxtiuiuiuhnwgqfftwrjryeyh3edeiojn0tjg}(b) we have $\rho_x \mid \rho_{e} \mid \rho_{z}$.
This shows that for each $z \in \Veul \setminus \Veuls$ there exists $x \in \Veuls$ such that $\rho_x \mid \rho_z$.
It easily follows that $\da(\Teul) = \gcd\setspec{ \rho_v }{ v \in \Veul } = \gcd\setspec{ \rho_v }{ v \in \Veuls }$.

The following notation enables us to refer to the sets $\Veul$, $\Veuls$ and $\Es$ in different trees:
if $\Teul' = (\Veul', \Aeul', \Eeul', f', q')$ is any positive tree, we define $\Veul(\Teul') = \Veul'$ and
\begin{align*}
\Veuls(\Teul') &= \setspec{ v \in \Veul' }{ \text{$v$ is a separating vertex} } , \\
\Es(\Teul') &= \setspec{ e \in E(\Teul') }{ \text{$e$ is a separating edge} } .
\end{align*}
Let $\bbX$ be the set of positive trees $\Teul$ such that $\Es(\Teul) \neq \emptyset$.
Given $\Teul \in \bbX$, let $\Pscr(\Teul)$ be the assertion
$p(\Teul) = \frac{ \F( \da(\Teul) ) \prod_{e \in \Es(\Teul)} \F(\rho_e) }{ \prod_{v \in \Veuls(\Teul)} \F( \rho_v ) }$
and let  $\| \Teul \| = | \Veul(\Teul) \setminus \Veuls(\Teul) |$.
Let us prove, by induction on $\| \Teul \|$, that $\Pscr(\Teul)$  is true for all $\Teul \in \bbX$.
If $\| \Teul \| = 0$ then $\Veuls(\Teul)=\Veul(\Teul)$ and $\Es(\Teul)=E(\Teul)$, so the claim follows immediately from the definition of $p(\Teul)$.
Consider $\Teul \in \bbX$ such that $\| \Teul \| > 0$ and such that $\Pscr(\Teul')$ is true for every $\Teul' \in \bbX$ satisfying $\| \Teul' \| <\| \Teul \|$.

We have $\Veuls(\Teul) \neq \emptyset$ (because $\Teul \in \bbX$, so $\Es(\Teul) \neq \emptyset$)
and $\Veul(\Teul) \setminus \Veuls(\Teul) \neq \emptyset$ (because $\| \Teul \| > 0$);
so we can find an edge $e = \{v_1,v_2\}$ such that $v_1 \in \Veuls(\Teul)$ and $v_2 \in \Veul(\Teul) \setminus \Veuls(\Teul)$.
Let $\{ \Teul_1, \Teul_2 \} = \ENSpl(\Teul,e)$, where the notation is such that $v_i$ is in $\Teul_i$.
Note that $e$ is not a separating edge, and that consequently 
$W(\Teul,e)=1$ by Prop.\ \ref{oiunuytdews23fegi7lkpo9iuuhgfvxcvxbhjj3ndgfe}(b).
By Lemma \ref{aow8uytuc62387r9234ump0eifoe}, $\Teul_1,\Teul_2$ are positive,  $\da(\Teul_1) = c(v_2,e)$, and $p(\Teul) = p(\Teul_1) p(\Teul_2)$.
In $\Teul$ we have $\Aeul(v_1,e) = \emptyset$, and it easily follows that no edge of $\Teul_2$ is separating;
so $p(\Teul_2)=1$ by Cor.\ \ref{iygnytfbcvcrwsercf5yujk9oljkhvfui7gwr689029f2a} and consequently $p(\Teul) = p(\Teul_1)$.

Observe that $p_1 = 0$ in the picture \eqref{t6frcesshtubcxbnfgg367wysj7hdfruf2tftv}.
So the fact that $\Aeul(v_1,e) = \emptyset$ in $\Teul$ implies that for each $\epsilon = \{x,y\} \in E(\Teul_1)$ we have
$\Aeul^{(1)}(x,\epsilon) = \Aeul(x,\epsilon)$ and $\Aeul^{(1)}(y,\epsilon) = \Aeul(y,\epsilon)$,
so $\epsilon$ is separating in $\Teul_1$ if and only if it is separating in $\Teul$.
Since $\Es(\Teul) \subseteq \Es(\Teul_1)$ is clear, it follows that 
$\Es(\Teul_1) = \Es(\Teul)$ and consequently $\Veuls(\Teul_1) = \Veuls(\Teul)$.
Define $\Es = \Es(\Teul_1) = \Es(\Teul)$ and $\Veuls = \Veuls(\Teul_1) = \Veuls(\Teul)$,
and note that the members of the families $( \rho_\epsilon )_{ \epsilon \in \Es }$ and $( \rho_v )_{ v \in \Veuls }$ have the same values
in $\Teul_1$ and in $\Teul$.
We have $\Teul_1 \in \bbX$ and $\| \Teul_1 \| < \| \Teul \|$, so $\Pscr(\Teul_1)$ is true, which gives the second equality in
$$
\textstyle
p(\Teul) = p(\Teul_1) =  \frac{ \F( \da(\Teul_1) ) \prod_{\epsilon \in \Es} \F(\rho_\epsilon) }{ \prod_{v \in \Veuls} \F( \rho_v ) }.
$$
To complete the proof, it only remains to show that  $\da(\Teul_1) = \da(\Teul)$.
Since $\Veul(\Teul) \supseteq \Veul(\Teul_1) \supseteq \Veuls(\Teul)$, we have
$$
\gcd \setspec{ \rho_v }{ v \in \Veul(\Teul) } \mid \gcd \setspec{ \rho_v }{ v \in \Veul(\Teul_1) } \mid \gcd \setspec{ \rho_v }{ v \in \Veuls(\Teul) } .
$$
The first, second and third of these numbers are, respectively,
$\da(\Teul)$, $\da(\Teul_1)$, and (as we showed) $\da(\Teul)$.
So $\da(\Teul) = \da(\Teul_1)$ and the proof is complete.
\end{proof}

\section{Genus formula}

The main result of this section is Theorem \ref{FormuleDuGenre},
which generalizes the genus formula of a general fiber of a polynomial map $\Comp^2 \to \Comp$.

\begin{definition}  \label {9id98g84uhj3jm3m3mffCj38ddRY589}
Let $\Teul = (\Veul,\Aeul,\Eeul,f,q) \in \DT$.
\begin{enumerata}

\item Given a path $\gamma = (w_1, \dots, w_n)$ with $n\ge2$,
let $E^+_\gamma$ denote the set of edges $e$ of $\Teul$ incident to $w_n$ and distinct from $\{w_{n-1},w_n\}$.
We say that $\gamma$ \textit{satisfies} $(+)$ if
$q(e,w_n)\ge1$ for all $e \in E^+_\gamma$ and at most one $e \in E^+_\gamma$ satisfies $q(e,w_n)>1$.

\item An element $x$ of $\Veul \cup \Aeul$ is \textit{central} if it satisfies
\begin{center}
for every $v \in \Veul \setminus \{ x \}$, the path $\gamma_{x,v}$ satisfies $(+)$.
\end{center}

\item We define the notation $\Cent(\Teul) = \setspec{ x \in \Veul \cup \Aeul }{ \text{$x$ is central} }$.

\end{enumerata}
\end{definition}

For instance, $\Cent(\Teul) = \{u,v,\alpha_2\}$ for the tree $\Teul$
of \eqref{dpd99238gve912R2tw}, in paragraph \ref{pco9v0239jd0OiiJq0wjd}.

\begin{definition} \label {jdhbf2i3p9e}
Let $\Teul=(\Veul,\Aeul,\Eeul,f,q) \in \DT$.
A subset $S$ of $\Veul \cup \Aeul$ is {\it connected\/} if every path
$(x_0,\dots,x_n)$ in $\Teul$ that satisfies $x_0,x_n \in S$ also satisfies $x_i \in S$ for all $i$ such that $0<i<n$.
\end{definition}

Note that $\Veul$ is a connected set, in any decorated tree.

\begin{lemma} \label {p0d92iu5gOncmWh63eh38fkj03iru8r}
If $\Teul \in \DT$ then $\Cent(\Teul)$ is connected.
\end{lemma}

\begin{proof}
Let $(x_1, \dots, x_n)$ be a path in $\Teul$ such that $n\ge3$ and $x_1, x_n \in \Cent(\Teul)$,
and let $i$ be such that $1<i<n$. To show that $x_{i} \in \Cent(\Teul)$, consider $y \in \Veul \setminus \{x_{i}\}$.
For some $u \in \{x_1,x_n\}$, $x_{i}$ is in the path $\gamma_{u,y}$. Since $u$ is central, $\gamma_{u,y}$ satisfies $(+)$;
since  $\gamma_{u,y}$ and $\gamma_{x_i,y}$ have the same last edge, $\gamma_{x_i,y}$ satisfies $(+)$.
This shows that $x_i \in \Cent(\Teul)$.
\end{proof}

\begin{remark}
In Def.\ \ref{pod293805h4s7b2w9310jdcjd}, we defined what we mean by a root of a decorated tree $\Teul = (\Veul, \Aeul, \Eeul, f, q ) \in \DT$.
We can rephrase that definition as follows:
a {\it root\/} of $\Teul$ is a vertex $v_0 \in \Veul$ that is central and satisfies $q(e,v_0)=1$ for all edges $e$ incident to $v_0$. 
\end{remark}

Recall from Def.\ \ref{pod293805h4s7b2w9310jdcjd} that a \textit{decorated rooted tree} is a tuple $(\Veul, \Aeul, \Eeul, f, q, v_0)$ such that 
$(\Veul, \Aeul, \Eeul, f, q) \in \DT$ and $v_0$ is a root of $(\Veul, \Aeul, \Eeul, f, q)$;
we write $\DTr$ for the set of decorated rooted trees.
If  $\Teul = (\Veul, \Aeul, \Eeul, f, q, v_0) \in \DTr$, then $v_0$ is called \underline{the} root of $\Teul$
and the number
$$
\deg(\Teul) = N_{v_0}
$$
is called the {\it degree\/} of $\Teul$.

\begin{definition} \label {p0cb2398gfv821wid9bgd6}
Let  $\Teul = (\Veul, \Aeul, \Eeul, f, q, v_0) \in \DTr$.
We define a partial order on the set $\Veul \cup \Aeul$ by stipulating that, given distinct $x,y \in \Veul \cup \Aeul$, 
$$
x < y \iff \text{$x$ is in $\gamma_{v_0, y}$.}
$$
It follows in particular that $v_0 \le y$ for all $y \in \Veul \cup \Aeul$.
\end{definition}

In the context of decorated rooted trees, we often use the symbol ``$v_0$'' without defining it:  it always denotes the root.
We also use the order relation of Def.~\ref{p0cb2398gfv821wid9bgd6} without recalling what it is.


\begin{notation}  \label {liuyc34t7Hi87tgebfdcnmG1Yk4i37tg467km}
Let $\Teul = (\Veul,\Aeul,\Eeul,f,q) \in \DT$.
\begin{enumerata}

\item Given a path $\gamma = (x_0, \dots, x_n)$ such that $n\ge1$, define $\phi(\gamma)$ to be the product of the numbers $q(e,u)$ such that
$u \in \{x_1, \dots, x_n\}$, $e$ is an edge incident to $u$ and $e$ is not in $\gamma$.

\item Given $x,y \in \Veul \cup \Aeul$ such that $x \neq y$, define $\phi(x,y) = \phi( \gamma_{x,y} )$.

\item Given a pair $(x,e)$ such that $x \in \Veul \cup \Aeul$ and $e$ is an edge incident to $x$, define
$Y(x,e) = \setspec{ y \in \Veul \cup \Aeul_0 }{ \text{$e$ is in $\gamma_{x,y}$} }$.

\end{enumerata}
\end{notation}

\begin{definition}
Let $\Teul = (\Veul,\Aeul,\Eeul,f,q,v_0) \in \DTr$ and consider a pair $(x,e)$ such that $x \in \Veul \cup \Aeul$ and $e$ is an edge incident to $x$.
Define $x'$ by the condition $e = \{x,x'\}$ and note that either $x < x'$ or $x>x'$.
If $x<x'$ (resp. $x>x'$), we say that $(x,e)$ is an {\it out-pair\/} (resp.\ an {\it in-pair\/}) of $\Teul$.
\end{definition}

\begin{lemma}  \label {873gsDXNHty76fdXCVBNi2378462b7w8jse2}
Let $\Teul = (\Veul,\Aeul,\Eeul,f,q,v_0) \in \DTr$ be such that $f(\alpha) \in \{ 0, 1 \}$ for all $\alpha \in \Aeul$.
The following holds for each out-pair $(x,e)$ of $\Teul$:
$$
p(x,e) = 1 + \sum_{y \in Y(x,e)} \phi(x,y) (\delta_y - 2) .
$$
\end{lemma}

\begin{proof}
Let $x'$ be such that $e = \{x,x'\}$, let $\bar Y(x,e) = \setspec{ y \in \Veul \cup \Aeul }{ \text{$e$ is in $\gamma_{x,y}$} }$ and note that $x' \in \bar Y(x,e)$.
We proceed by induction on $| \bar Y(x,e) |$.

If $| \bar Y(x,e) | = 1$ then there are two cases: either $x' \in \Veul \cup \Aeul_0$ or $x' \in \Aeul \setminus \Aeul_0$.
In the first case we have $p(x,e)=0$ and  $\sum_{y \in Y(x,e)} \phi(x,y) (\delta_y - 2) = \phi(x,x')(\delta_{x'}-2) = 1(1-2)=-1$, so the result is valid.
In the second case we have $p(x,e) = f(x')=1$ (because $\image(f) \subseteq \{0,1\}$)
and $Y(x,e) = \emptyset$, so $\sum_{y \in Y(x,e)} \phi(x,y) (\delta_y - 2) = 0$, so the result is valid.
Hence, the claim is true whenever $| \bar Y(x,e) | = 1$.

Let $N>1$ be such that the claim is true for all out-pairs $(x,e)$ such that $| \bar Y(x,e) | < N$. 
Suppose that $(x,e)$ is an out-pair such that $| \bar Y(x,e) | = N$.
Let $x_1, \dots, x_n$ ($n \ge 1$) be the distinct elements of $(\Veul \cup \Aeul) \setminus \{x\}$ which are adjacent to $x'$
and consider the edges $e_i = \{ x', x_i \}$ ($1 \le i \le n$).
Since $(x,e)$ is an out-pair of $\Teul$, so are $(x',e_1), \dots, (x',e_n)$.
It is clear that (for each $i$) $| \bar Y(x',e_i) | < N$,
so the inductive hypothesis implies that (for each $i$) $p(x',e_i) = 1 + \sum_{y \in Y(x',e_i)} \phi(x',y) (\delta_y - 2)$.
Since $(x,e)$ is an out-pair of $\Teul$, at most one $i \in \{1,\dots,n\}$ is such that $q(e_i,x') \neq 1$;
so we may assume that the labelling of $x_1, \dots, x_n$ is such that $q(e_i,x') = 1$ for all $i>1$. Let $b = q(e_1,x')$. Then
\begin{align*}
p(x,e) &= p(x',e_1) + b \sum_{i=2}^n p(x',e_i) \\
&=\left[ 1 + \sum_{ y \in Y(x',e_1) } \phi(x',y)(\delta_y - 2) \right] + b \sum_{i=2}^n \left[ 1 + \sum_{ y \in Y(x',e_i) } \phi(x',y)(\delta_y - 2) \right]
\end{align*}
We have $\phi(x,y) = \phi(x',y)$ for all $y \in Y(x',e_1)$, and if $i>1$, $\phi(x,y) = b \phi(x',y)$ for all $y \in Y(x',e_i)$.
Thus,
$$
p(x,e) = 1 + b(n-1) + \sum_{y \in Y(x,e) \setminus \{x'\}} \phi(x,y) (\delta_y - 2) \, .
$$
Since $\phi(x,x') = b$ and $\delta_{x'} = n+1$, we get $p(x,e) = 1 + \sum_{y \in Y(x,e) } \phi(x,y) (\delta_y - 2)$.
\end{proof}

\begin{smallremark}
It is easy to see that the converse of Lemma \ref{873gsDXNHty76fdXCVBNi2378462b7w8jse2} is true, i.e.,
if $p(x,e) = 1 + \sum_{y \in Y(x,e)} \phi(x,y) (\delta_y - 2)$ for every out-pair $(x,e)$ of $\Teul \in \DTr$
then $f(\alpha) \in \{0,1\}$ for every $\alpha \in \Aeul$.
\end{smallremark}

\begin{lemma}  \label {8273ygv4r87gfo124d40239ur}
Let $\Teul = (\Veul,\Aeul,\Eeul,f,q) \in \DT$ and let $\eta \in \Cent(\Teul)$. 
There exists a unique tree $\Teul^{(\eta)} \in \DT$ which satisfies the following conditions {\rm(i--v)}.
\begin{enumerati}

\item $\Teul^{(\eta)} = (\Veul,\Aeul,\Eeul,f,q^*)$ for some $q^*$.

\item $q^*(e,\eta) = q(e,\eta)$ for all edges $e$ incident to $\eta$.

\item $q^*(e,v) = q(e,v)$ for all  $v \in \Veul \setminus \{ \eta \}$ and all edges $e$ incident to $v$ and not in $\gamma_{\eta,v}$.

\item If $e = \{ \eta, v\}$ is an edge and $v \in \Veul$ then $q^*(e,v) = Q(e,v) - q(e,v)$. 

\item If $e = \{ u,v \}$ is an edge and $u,v \in \Veul \setminus\{\eta\}$ then $\det^*(e) = -\det(e)$, where $\det(e)$ is computed in $\Teul$ and $\det^*(e)$ in $\Teul^{(\eta)}$.

\end{enumerati}
Moreover, $\Teul^{(\eta)}$ has the following property:
\begin{enumerati}
\item[\rm(vi)] If $v \in \Veul \setminus \{ \eta \}$ and $e$ is the unique edge  in $\gamma_{\eta,v}$ which is incident to $v$, then
$Q(e,v)$ divides $q(e,v) + q^*(e,v)$.
\end{enumerati}
\end{lemma}

\begin{proof}
For each $v \in \Veul \setminus \{ \eta \}$, let $e_v$ be the unique edge in $\gamma_{\eta,v}$ which is incident to $v$
and define $\bar v$ by the condition $e_v = \{ v, \bar v \}$
(so $v \mapsto \bar v$ is a set map from $\Veul \setminus \{ \eta \}$ to $\Veul \cup \{ \eta \}$).
We also define, for each $v \in \Veul \setminus \{ \eta \}$,
$$
x_v = q(e_v, v), \quad
y_v = q(e_v, \bar v) = q^*(e_v, \bar v) , \quad Q_v = Q(e_v, v) = Q^*(e_v, v) .
$$
Observe:
\begin{equation} \label {87tD237654drChv82ehdbB4za2qw}
\text{for all $v \in \Veul \setminus \{\eta\}$ not adjacent to $\eta$ we have $y_v>0$, $Q_{\bar v}>0$ and $y_v \mid Q_{\bar v}$.}
\end{equation}
Indeed, if  $v \in \Veul \setminus \{\eta\}$ is not adjacent to $\eta$ then $\bar v \in \Veul \setminus \{\eta\}$,
so the fact that $\eta \in \Cent(\Teul)$ implies that $y_v = q(e_v, \bar v) > 0$ and $Q_{\bar v} > 0$, and it is clear that $y_v \mid Q_{\bar v}$.
So \eqref{87tD237654drChv82ehdbB4za2qw} is true.

To complete the definition of $q^*$, it suffices to specify the value of $q^*(e_v,v)$ for each $v \in \Veul \setminus \{ \eta \}$.
We say that a family of integers $\big( x_v^* \big)_{ v \in \Veul \setminus \{ \eta \} }$ is ``good'' if conditions (iv) and (v) are satisfied 
when we set $q^*(e_v,v) = x_v^*$ for all $v \in \Veul \setminus \{ \eta \}$.
One can see that conditions (iv) and (v) are equivalent to
\begin{align}
x_v + x_v^* &=  Q_v \ \ \text{for all $v \in \Veul \setminus \{ \eta \}$ adjacent to $\eta$,} \\
\label {87u4yrhflmknbvg1f5q34wg5y} x_v + x_v^* &= Q_v \left( \frac{Q_{\bar v}}{y_v} \right)^2  \left( \frac{x_{\bar v} + x^*_{\bar v}}{Q_{\bar v}} \right)
\ \ \text{for all $v \in \Veul \setminus \{ \eta \}$ not adjacent to $\eta$.}
\end{align}
By \eqref{87tD237654drChv82ehdbB4za2qw} and \eqref{87u4yrhflmknbvg1f5q34wg5y}, if $x^*_{\bar v} \in \Integ$ and $Q_{\bar v} \mid (x_{\bar v} + x^*_{\bar v})$
then $x^*_v \in \Integ$ and  $Q_{v} \mid (x_{v} + x^*_{v})$.
It follows by induction that there is a unique good family $\big( x_v^* \big)_{ v \in \Veul \setminus \{ \eta \} }$,
and that this family is such that $Q_v \mid (x_v + x^*_v)$ for all $v \in \Veul \setminus \{ \eta \}$.
Since $\gcd(x_v,Q_v)=1$ and $Q_v \mid (x_v + x^*_v)$, we have $\gcd(x^*_v,Q_v)=1$.
This guarantees that if we set $q^*(e_v,v) = x_v^*$ for each  $v \in \Veul \setminus \{ \eta \}$ then,
in the tree $\Teul^{(\eta)} = (\Veul,\Aeul,\Eeul,f,q^*)$,
the edge decorations near any given vertex are pairwise relatively prime; so $\Teul^{(\eta)} \in \DT$.
This proves the lemma.
\end{proof}

\begin{definition} \label {987654xcrv1tyd2ewdpomjlhkj98v7165f2exrzq4f}
Let $\Teul = (\Veul,\Aeul,\Eeul,f,q, v_0) \in \DTr$ be such that $\delta_{v_0} > 0$,
let $v_1, \dots, v_n$ ($n\ge1$) be the distinct elements of $\Veul \cup \Aeul$ that are adjacent to $v_0$,
and let $e_i = \{v_0,v_i\}$ for $i=1,\dots,n$.  We define $\Teul_1, \dots, \Teul_n \in \DT$ as follows.

Delete the edges $e_1,\dots,e_n$ from $\Teul$ and (for each $i \in \{1,\dots,n\}$)
let $\Ceul_i \in \DT$ be the connected component that contains $v_i$.  Let $\Teul_i'$ be the tree obtained from $\Ceul_i$ by:
\begin{itemize}
\item adding an arrow $\alpha_i$ decorated by $(0)$ and an edge $\epsilon_i = \{\alpha_i,v_i\}$;
\item letting the decorations of $\epsilon_i$ near $v_i$ and $\alpha_i$ be  $q(e_i,v_i)$ and $1$, respectively.
\end{itemize}
Since $\Teul_i' \in \DT$ and $\alpha_i \in \Cent( \Teul_i' )$, we may consider the tree $(\Teul_i')^{(\alpha_i)} \in \DT$ defined in Lemma \ref{8273ygv4r87gfo124d40239ur}.
For each $i \in \{1,\dots,n\}$, we define 
$$
\Teul_i = (\Teul_i')^{(\alpha_i)} .
$$
When considering $\Teul_1,\dots,\Teul_n$, we shall use the following notation.
We write $\Teul_i = (\Veul_i,\Aeul_i,\Eeul_i,f_i,q_i)$ and $\Aeul_{0,i} = \setspec{ \alpha \in \Aeul_i }{ f_i(\alpha) = 0 }$.
We write $Q(e,x)$, $p(x,e)$, $\det(e)$ and $N_x$ when these quantities are computed in $\Teul$, and $Q_i(e,x)$, $p_i(x,e)$, $\det_i(e)$ and $N_x^{(i)}$ when computed in $\Teul_i$.
We use the symbols $v_i, e_i, \alpha_i, \epsilon_i$ (defined in the above paragraph) without recalling their definition.
\end{definition}

\begin{lemma}  \label {6ft2xewmlkijp2n098h7h6zxcv4rt3dnjt56hajcgj}
Let $\Teul = (\Veul,\Aeul,\Eeul,f,q, v_0) \in \DTr$ and $n = \delta_{v_0}$.
Assume that $n>0$ and consider $\Teul_i = (\Veul_i,\Aeul_i,\Eeul_i,f_i,q_i) \in \DT$ $(1 \le i \le n)$ as in Def.\ \ref{987654xcrv1tyd2ewdpomjlhkj98v7165f2exrzq4f}.
For each  $x \in (\Veul \setminus \{v_0\}) \cup \Aeul_0$ we have
$$
N_x + N_x^{(i)} = \phi(v_0,x) \deg(\Teul) ,
$$
where $i$ is the unique element of $\{1,\dots,n\}$ such that $x \in \Veul_i \cup ( \Aeul_{0,i} \setminus \{\alpha_i\} )$.
\end{lemma}

\begin{proof}
Consider the case where $x$ is adjacent to $v_0$. 
There are two subcases: $x \in \Aeul_0$ or $x \in \Veul$.
The subcase $x \in \Aeul_0$ is as follows:
$$
{\begin{picture}(36,6)(-12,-3)
\put(-8,0){\makebox(0,0)[r]{$\Teul:$}}
\put(0,0){\circle{1}}
\put(0.5,0){\vector(1,0){19.5}}
\put(-.4472,.2236){\line(-2,1){4}}
\put(-.4472,-.2236){\line(-2,-1){4}}
\put(21,0){\makebox(0,0)[l]{\tiny $(0)$}}
\put(0,1){\makebox(0,0)[b]{\tiny $v_0$}}
\put(19,1){\makebox(0,0)[b]{\tiny $x$}}
\put(10,-.7){\makebox(0,0)[t]{\tiny $e_i$}}
\end{picture}}
\qquad\quad
{\begin{picture}(36,6)(-12,-3)
\put(-8,0){\makebox(0,0)[r]{$\Teul_i:$}}
\put(10,0){\vector(1,0){10}}
\put(10,0){\vector(-1,0){10}}
\put(-1,0){\makebox(0,0)[r]{\tiny $(0)$}}
\put(21,0){\makebox(0,0)[l]{\tiny $(0)$}}
\put(1,1){\makebox(0,0)[b]{\tiny $\alpha_i$}}
\put(19,1){\makebox(0,0)[b]{\tiny $x$}}
\put(10,-.7){\makebox(0,0)[t]{\tiny $\epsilon_i$}}
\end{picture}}
$$
so $N_x = p(x,e_i)$ and $N^{(i)}_x = 0$, so $N_x + N^{(i)}_x = p(x,e_i) = \deg(\Teul) = \phi(v_0,x) \deg(\Teul)$, i.e., the result is valid in this case.
In the second subcase ($x \in \Veul$) we have: 
$$
{\begin{picture}(36,6)(-12,-3)
\put(-8,0){\makebox(0,0)[r]{$\Teul:$}}
\put(0,0){\circle{1}}
\put(20,0){\circle{1}}
\put(0.5,0){\line(1,0){19}}
\put(-.4472,.2236){\line(-2,1){4}}
\put(-.4472,-.2236){\line(-2,-1){4}}
\put(20.4472,.2236){\line(2,1){4}}
\put(20.4472,-.2236){\line(2,-1){4}}
\put(0,1){\makebox(0,0)[b]{\tiny $v_0$}}
\put(20,1){\makebox(0,0)[b]{\tiny $x$}}
\put(10,-.7){\makebox(0,0)[t]{\tiny $e_i$}}
\put(17,1){\makebox(0,0)[br]{\tiny $t$}}
\put(23,0){\makebox(0,0)[l]{\tiny $T$}}
\end{picture}}
\qquad\quad
{\begin{picture}(36,6)(-12,-3)
\put(-8,0){\makebox(0,0)[r]{$\Teul_i:$}}
\put(20,0){\circle{1}}
\put(10,0){\line(1,0){9.5}}
\put(10,0){\vector(-1,0){10}}
\put(20.4472,.2236){\line(2,1){4}}
\put(20.4472,-.2236){\line(2,-1){4}}
\put(-1,0){\makebox(0,0)[r]{\tiny $(0)$}}
\put(1,1){\makebox(0,0)[b]{\tiny $\alpha_i$}}
\put(20,1){\makebox(0,0)[b]{\tiny $x$}}
\put(10,-.7){\makebox(0,0)[t]{\tiny $\epsilon_i$}}
\put(17,1){\makebox(0,0)[br]{\tiny $T-t$}}
\end{picture}}
$$
Let $t = q(e_i,x)$ and $T = Q(e_i,x)$, and note that $q_i(\epsilon_i,x) = T-t$.
Rem.\ \ref{52retfxcxntgBmA98WjmHkkj498736762} gives
$N_x = t p(v_0,e_i) + T p(x,e_i)$ and $N^{(i)}_x = (T-t) p_i(\alpha_i,\epsilon_i) = (T-t) p(v_0,e_i)$, 
so $N_x + N^{(i)}_x = T p(v_0,e_i) + T p(x,e_i) = T \deg(\Teul) = \phi(v_0,x) \deg(\Teul)$.

So the result is valid when $x$ is adjacent to $v_0$.

From now-on, assume that $x$ is not adjacent to $v_0$.
Let $e = \{u,x\}$ be the unique edge in $\gamma_{v_0,x}$ which is incident to $x$.
Note that $u$ belongs to $\Veul \setminus\{v_0\}$ and $\Veul_i \setminus\{\alpha_i\}$.
Proceeding by induction on the length of $\gamma_{v_0,x}$, we may assume that $N_{u} + N^{(i)}_{u} = \phi( v_0 , u ) \deg(\Teul)$.
Define
$$
b = q(e,u) = q_i(e,u) \quad \text{and} \quad B = Q(e,x) = Q_i(e,x) 
$$
and note that $b>0$, $B>0$ and $\phi(v_0,x) = (B/b) \phi(v_0,u)$.
Prop.\ \ref{kuwdhr12778} gives
$$
\left| \begin{smallmatrix} b & B \\ N_u & N_x \end{smallmatrix} \right| = \det(e) p(u,e)
\quad \text{and} \quad
\left| \begin{smallmatrix} b & B \\ N^{(i)}_u & N^{(i)}_x \end{smallmatrix} \right| = \textstyle \det_i(e) p_i(u,e).
$$
Since $p(u,e) = p_i(u,e)$, it follows that 
\begin{equation}  \label {8v7y2wmlkojqCe0fsqwdfxdtlgvk7v6g2x}
\textstyle
\left| \begin{smallmatrix} b & B \\ N_u+N^{(i)}_u & N_x+N^{(i)}_x \end{smallmatrix} \right| = (\det(e) + \det_i(e)) p(u,e) .
\end{equation}
Note that if $x \in \Veul$ then $\det(e) + \det_i(e) = 0$, and if $x \in \Aeul_0$ then $p(u,e)=0$;
so the right-hand-side of \eqref{8v7y2wmlkojqCe0fsqwdfxdtlgvk7v6g2x} is $0$.
This implies that 
$$
N_x + N^{(i)}_x = (B/b) (N_u+N^{(i)}_u) = (B/b) \phi(v_0,u) \deg(\Teul) = \phi(v_0,x) \deg(\Teul) .
$$
\end{proof}

\begin{theorem}  \label {FormuleDuGenre}
Let $\Teul = (\Veul,\Aeul,\Eeul,f,q, v_0) \in \DTr$ be such that $f(\alpha) \in \{0,1\}$ for all $\alpha \in \Aeul$.
Let $n = \delta_{v_0}$, assume that $n\ge1$,
and consider $\Teul_1, \dots, \Teul_n \in \DT$ as in Def.\ \ref{987654xcrv1tyd2ewdpomjlhkj98v7165f2exrzq4f}.
Let $d = \deg(\Teul)$. Then
$$
g(\Teul) =  \frac{(d-1)(d-2)}2 - \sum_{i=1}^n \delta(\Teul_i) \, .
$$
\end{theorem}

\begin{proof}
It suffices to show that $2g(\Teul) + \sum_{i=1}^n 2\delta(\Teul_i) = (d-1)(d-2)$.
We have $g(\Teul) = \frac12 ( 2 - M(\Teul) - F(\Teul) )$ and $\delta(\Teul_i) = \frac12 ( F(\Teul_i) - M(\Teul_i) )$ by definition,
and the hypothesis $\image(f) \subseteq \{0,1\}$ implies that $F(\Teul) = \sum_{i=1}^n F(\Teul_i)$;
so $2g(\Teul) + \sum_{i=1}^n 2\delta(\Teul_i) = 2 - M(\Teul) - \sum_{i=1}^n M(\Teul_i)$
and it suffices to show that
\begin{equation} \label {98g32uef64ewlkokpoi09iuyhmkloir9d09yt}
2 - M(\Teul) - \sum_{i=1}^n M(\Teul_i) = (d-1)(d-2) .
\end{equation}
Let $X_i = \Veul_i \cup (\Aeul_{i,0} \setminus \{ \alpha_i \})$ and note that $(\Veul \cup \Aeul_0) \setminus \{v_0\} = \bigcup_{i=1}^n X_i$ is a disjoint union.
Since $N_{v_0}(\delta_{v_0}-2) = d(n-2)$, we have
\begin{equation} \label {8h7y5dSE45ybnb5or9j38cg7trd2H}
- M(\Teul) = \sum_{x \in \Veul \cup \Aeul_0} N_x (\delta_x - 2) = d(n-2) +\sum_{i=1}^n \sum_{x \in X_i} N_x (\delta_x - 2) .
\end{equation}
Let $i \in \{1,\dots,n\}$.
For each $x \in X_i$ we have $N_x = \phi(v_0,x)d - N^{(i)}_x$ by Lemma \ref{6ft2xewmlkijp2n098h7h6zxcv4rt3dnjt56hajcgj}, so
\begin{align} \label {Atvyb782923tfVg5hn9H6oiPik665789n}
\sum_{x \in X_i} N_x (\delta_x - 2) &= \sum_{x \in X_i} (\phi(v_0,x)d - N^{(i)}_x) (\delta_x - 2) \\
\notag &= d \sum_{x \in X_i} \phi(v_0,x) (\delta_x - 2) - \sum_{x \in X_i} N^{(i)}_x (\delta_x - 2) .
\end{align}
Let $d_i = p(v_0,e_i)$ and note that $X_i = Y(v_0,e_i)$
(see Def.\ \ref{987654xcrv1tyd2ewdpomjlhkj98v7165f2exrzq4f} for $e_i$ and Notation \ref{liuyc34t7Hi87tgebfdcnmG1Yk4i37tg467km} for $Y(v_0,e_i)$);
we get $\sum_{x \in X_i} \phi(v_0,x) (\delta_x - 2) = p(v_0,e_i) - 1 = d_i-1$ by Lemma \ref{873gsDXNHty76fdXCVBNi2378462b7w8jse2}.
Also, $ - \sum_{x \in X_i} N^{(i)}_x (\delta_x - 2)
= M(\Teul_i) + N^{(i)}_{\alpha_i}(\delta_{\alpha_i}-2)
= M(\Teul_i) - N^{(i)}_{\alpha_i}
= M(\Teul_i) - p_i(\alpha_i,\epsilon_i)
= M(\Teul_i) - p(v_0,e_i)
= M(\Teul_i) - d_i$.
Substituting these values in \eqref{Atvyb782923tfVg5hn9H6oiPik665789n} gives $\sum_{x \in X_i} N_x (\delta_x - 2) = d(d_i-1) + M(\Teul_i) - d_i$.
Substituting this in \eqref{8h7y5dSE45ybnb5or9j38cg7trd2H} and noting that $\sum_{i=1}^n d_i = d$ gives
$$
- M(\Teul) =  d(n-2) + \sum_{i=1}^n \big[ d(d_i-1) + M(\Teul_i) - d_i \big] =  d^2-3d + \sum_{i=1}^n M(\Teul_i) .
$$
It follows that \eqref{98g32uef64ewlkokpoi09iuyhmkloir9d09yt} is true, and this proves the Theorem.
\end{proof}



\end{document}